\documentclass[10pt]{amsart}
\usepackage{amssymb,graphicx}
\usepackage{mathrsfs, bbm, slashed}
\usepackage{color}
\usepackage{enumerate}
\usepackage{a4wide}

\usepackage[utf8]{inputenc}
\usepackage[colorlinks=true,citecolor=blue,linkcolor=blue]{hyperref}

\newcommand{\R}{\mathbb{R}}
\newcommand{\C}{\mathbb{C}}
\newcommand{\Z}{\mathbb{Z}}
\newcommand{\N}{\mathbb{N}}
\newcommand{\T}{\mathbb{T}}
\newcommand{\bB}{\mathbb{B}}

\newcommand{\sA}{\mathscr{A}}
\newcommand{\sB}{\mathscr{B}}
\newcommand{\sD}{\mathscr{D}}
\newcommand{\sL}{\mathscr{L}}
\newcommand{\sH}{\mathscr{H}}
\newcommand{\sK}{\mathscr{K}}
\newcommand{\sM}{\mathscr{M}}
\newcommand{\sS}{\mathscr{S}}
\newcommand{\sV}{\mathscr{V}}
\newcommand{\sW}{\mathscr{W}}

 \newcommand{\psido}{$\Psi$DO}
 \newcommand{\psidos}{$\Psi$DOs}

\def\dashint{\Xint-}

\newcommand{\bint}{\ensuremath{\dashint}}
\newcommand{\op}{\operatorname} 

\newcommand{\Sp}{\op{Sp}}
\newcommand{\tr}{\op{tr}}
\newcommand{\Vol}{\op{Vol}}
\newcommand{\Tr}{\op{Tr}}
\newcommand{\ran}{\op{ran}}

\newcommand{\End}{\op{End}}

\newcommand{\id}{\mathrm{id}}
\newcommand{\Res}{\op{Res}}
\newcommand{\rk}{\op{rk}}
\newcommand{\car}{\mathbbm{1}}

\newcommand{\scal}[2]{\ensuremath{\left\langle #1 | #2 \right\rangle}} 
\newcommand{\acou}[2]{\ensuremath{\left\langle #1 , #2 \right\rangle}} 
\newcommand{\bigscal}[2]{\ensuremath{\big\langle #1 | #2 \big\rangle}} 

\newcommand{\dom}{\op{dom}}

\def\Xint#1{\mathchoice
{\XXint\displaystyle\textstyle{#1}}
{\XXint\textstyle\scriptstyle{#1}}
{\XXint\scriptstyle\scriptscriptstyle{#1}}
{\XXint\scriptscriptstyle\scriptscriptstyle{#1}}
\!\int}
\def\XXint#1#2#3{{\setbox0=\hbox{$#1{#2#3}{\int}$ }
\vcenter{\hbox{$#2#3$ }}\kern-.6\wd0}}

\numberwithin{equation}{section}

\newtheorem{theorem}{Theorem}[section]
\newtheorem{proposition}[theorem]{Proposition}
\newtheorem{corollary}[theorem]{Corollary}

\newtheorem*{conditionW}{Condition (W)}

\newtheorem{lemma}[theorem]{Lemma}

\newtheorem{theoremalph}{Theorem}
\newtheorem{lemmalph}[theoremalph]{Lemma}
\newtheorem{propositionalph}[theoremalph]{Proposition}
\newtheorem{corollaryalph}[theoremalph]{Corollary}

\theoremstyle{definition}
\newtheorem{definition}[theorem]{Definition}
\newtheorem{condition}[theorem]{Condition}

\theoremstyle{remark}
\newtheorem{remark}[theorem]{Remark}

\newcommand{\LlogL}{L\!\log\!L}

\newcommand{\shD}{\slashed{D}}
\newcommand{\shS}{\slashed{S}}

\newcommand{\sa}{\textup{sa}}

\newcommand{\PartII}{Part~II}   

\title{Noncommutative Geometry, Spectral Asymptotics, and Semiclassical Analysis, I. Main Results}
\date{}

\author{Rapha\"el Ponge}
 \address{Department of Mathematics and Statistics, University of Wyoming, Laramie, WY, USA}
 \address{Department of Mathematics and Statistics, University of Ottawa, Canada}
 \email{ponge.math@icloud.com}

\begin{document}
\begin{abstract}
In this paper, we combine techniques from functional analysis and spectral theory to obtain semiclassical Weyl laws and extensions of Connes' integration formula for a large class of spectral triples. These results generalize and simplify recent work of McDonald--Sukochev--Zanin: all regularity assumptions and dimension restrictions are removed, and their Tauberian condition is replaced by a weaker spectral condition, Condition~(W), which holds in far greater generality. We also establish Tauberian criteria for Condition~(W) involving zeta functions and heat traces, as sought by Alain Connes. The results are illustrated on bounded Euclidean domains, closed Riemannian manifolds, the Steklov problem, and quantum tori. In the latter setting they resolve a conjecture of Edward McDonald and the author. This is the first part of a two-part work. The proofs of the main technical results are given in Part II.
\end{abstract}

\maketitle 
  
\tableofcontents

\section{Introduction}\label{chap:Intro}
Noncommutative geometry and semiclassical analysis are two important fields within the broader area of quantum theory. Recently, bridges between them have been emerging (see, e.g.,~\cite{MSZ:LMP22, Po:MPAG22, Po:JNCG23, Ro:JST22, SZ:FAA23}). This has its roots, on the one hand, in the functional analytic approach to semiclassical analysis of Schr\"odinger operators pioneered by Birman--Solomyak~\cite{BS:FAA70, BS:AMST89} and Simon~\cite{Si:TAMS76}, and, on the other hand, in the paramount role of operator ideals in the noncommutative geometry program of Connes~\cite{Co:NCG}.

In the framework of noncommutative geometry, the role of manifolds is played by spectral triples $(\sA,\sH,D)$, where $\sA$ is a (unital) $*$-algebra represented by bounded operators on the Hilbert space $\sH$ and $D$ is an unbounded selfadjoint operator on $\sH$ with compact resolvent which commutes with elements of $\sA$ up to bounded operators. The notion of dimension is replaced by $p$-summability, in the sense that $D^{-1}$ lies in the weak Schatten class $\sL_{p,\infty}$. 
The main prototype of a spectral triple is the Dirac spectral triple $(C^\infty(M), L^2(M,\shS),\shD)$ associated with any closed Riemannian spin manifold. A closely related example is the square-root Laplacian spectral triple $(C^\infty(M),L^2(M),\sqrt{\Delta_g})$.  

In a recent article~\cite{MSZ:LMP22} McDonald--Sukochev--Zanin established a semiclassical Weyl law for spectral triples $(\sA,\sH,D)$, i.e., for the counting functions of Schr\"odinger operators $h^2D^2+V$ under the semiclassical limit $h\rightarrow 0^+$. This semiclassical Weyl law is established as a consequence of a general Tauberian theorem for non-commuting pairs of operators on Hilbert space (see~\cite[Theorem~1.2]{MSZ:LMP22}). In order to apply this Tauberian theorem three conditions for the spectral triple are required:
\begin{enumerate}
 \item[(a)] $p$-summability with $p>2$. 
 
 \item[(b)] Lipschitz regularity, i.e., $[|D|,a]\in \sL(\sH)$ for all $a\in \sA$.
 
 \item[(c)] A Tauberian condition for the zeta functions $\Tr[a^z|D|^{-z}]$, $a\in \overline{\sA}_+$, $\Re z>p$ (see \cite[Condition~6.1]{Po:Part2} for the precise statement). 
\end{enumerate}

It would be desirable to remove the technical conditions (a) and (b), especially the former, since they prevent us from dealing with 1-dimensional and 2-dimensional examples. It is actually conjectured in~\cite{MSZ:LMP22} that the main results there should hold for $p\leq 2$.  

In addition, it would be good to relate the Tauberian condition~(c) to more standard Tauberian conditions in terms of zeta functions $\Tr[a|D|^{-z}]$ or heat traces $\Tr[ae^{-tD^2}]$, $t>0$, which do not involve powers of $a$. The need for this was pointed out by Alain Connes during an online seminar in May 2021. He also pointed out that closedness under holomorphic functional calculus should play a role. 

The main aim of this article is to show that, in the setup of spectral triples, results such as semiclassical Weyl laws and Connes' integration formulas, as well as Weyl laws for Steklov eigenvalues, are mere consequences of \emph{classical} Weyl laws for some Laplace-type operators, namely, operators of the form $aD^2a$, where $a$ ranges over positive invertible elements of $\sA$. This simplifies the approach of~\cite{MSZ:LMP22} and leads to stronger results.  In particular, we remove the conditions~(a)--(b) in the approach of~\cite{MSZ:LMP22} and replace the Tauberian condition of~\cite{MSZ:LMP22} by a weaker spectral-theoretic condition, Condition~(W), which precisely requires Weyl laws for these operators (see Section~\ref{subsec:Intro.W} and Section~\ref{chap:Main-results}). All our main results are established under this condition, which is easier to check in practice and allows us to cover various new examples. We also exhibit Tauberian conditions in terms of zeta functions and heat traces, Condition~(Z) and Condition~(H), which imply Condition~(W) and answer the above-mentioned comments of Alain Connes (see Section~\ref{subsec:Intro.Tauberian}). Finally, a Cwikel-type condition, Condition~($\textup{C}_r$), allows us to extend our results to unbounded potentials (see Section~\ref{subsec:Intro.unbounded}). 
 
This article is the first part of a two-part work. The main operator-theoretic ingredients of our approach, namely, the Factorization Lemma, the Enhanced CSZ Lemma, and the Tauberian theorem for Conditions~(Z) and~(H) (see~\cite[Theorems~A, C, and~F]{Po:Part2}), are established in the companion paper \PartII{}~\cite{Po:Part2}, which can be read largely independently of the present paper. Their statements are recalled in Section~\ref{subsec:Intro.lemmas} and Section~\ref{subsec:Intro.Tauberian}, and they are used without proof in the present paper.

\subsection{Spectral Condition (W)}\label{subsec:Intro.W}  
Let $(\sA,\sH, D)$ be a $p$-summable (unital) spectral triple, where $p$ may be \emph{any} positive number. We denote by $\overline{\sA}$ the closure of $\sA$ in $\sL(\sH)$. We also let $\overline{\sA}_{++}$ be the cone of invertible positive elements of $\sA$ (i.e., selfadjoint elements with positive spectrum), and set $\sA_{++}=\sA \cap \overline{\sA}_{++}$. Moreover, given any $a\in \overline{\sA}_{++}$, we let 
\begin{equation*}
 \lambda_0(aD^{2} a)\leq \lambda_1(aD^{2} a)\leq \cdots
\end{equation*}
be the \emph{positive} eigenvalues of $aD^{2} a$ (counted with multiplicity). 

The approach of this article relies on trading the Tauberian condition of~\cite{MSZ:LMP22} for some spectral condition  (Condition (W) below) which does not involve any additional regularity conditions or restriction on the degree of summability for spectral triples. As we shall see, this condition is actually weaker than the Tauberian condition of~\cite{MSZ:LMP22} and will allow us to cover a wealth of new examples. This condition is formulated as follows. 

\begin{conditionW}
 For every $a\in {\sA}_{++}$, we have
\begin{equation}\label{eq:Intro.conf-Weyl}
 \lim_{j\rightarrow \infty}j^{-\frac{2}{p}}\lambda_j\big(aD^2a\big)= \tau\big[a^{-p}\big]^{-\frac{2}{p}}, 
\end{equation}
where $\tau:\overline{\sA}\rightarrow \C$ is a given non-zero positive linear map.
\end{conditionW}

We can interpret Condition~(W) as requiring Weyl laws for conformal deformations of the square $D^2$, where conformal deformation is meant in the sense of Connes--Moscovici~\cite{CM:Twisted}. 

\subsection{Spectral asymptotics for Birman--Schwinger operators} 
As we shall see, Condition~(W) allows us to get spectral asymptotics for Birman--Schwinger type operators 
$|D|^{-q/2}a|D|^{-q/2}$, $q>0$, as $a$ ranges over the $C^*$-algebra $\overline{\sA}$. Once we have these spectral asymptotics, routine arguments enable us to get semiclassical Weyl laws and integration formulas (see below). 

In what follows, given any $q>0$, we denote by $\mu_j(|D|^{-q/2}a|D|^{-q/2})$, $j\geq 0$, the singular values of 
$|D|^{-q/2}a|D|^{-q/2}$ (i.e., the eigenvalues of the absolute value $||D|^{-q/2}a|D|^{-q/2}|$). We use similar notation for the singular values of $a|D|^{-q}$. If $a^*=a$, then  $|D|^{-q/2}a|D|^{-q/2}$ is selfadjoint, and so we further denote by $\lambda_j^\pm(|D|^{-q/2}a|D|^{-q/2})$, $j\geq 0$, its positive/negative eigenvalues (counted with multiplicity). 

More precisely, we establish the following spectral asymptotics, 

\begin{propositionalph}[Spectral Asymptotics]\label{prop:Intro.Weyl-BirS}
 Assume Condition~\textup{(W)} holds. Given any $q>0$, for every $a\in \overline{\sA}$, we have
 \begin{gather}
 \label{eq:Intro.Weyl-|DpaDp|-tau}
 \begin{aligned}
 \lim_{j\rightarrow \infty}j^{\frac{q}{p}} \mu_j\big(|D|^{-\frac{q}{2}}a|D|^{-\frac{q}{2}}\big) & = \lim_{j\rightarrow \infty}j^{\frac{q}{p}} \mu_j\big(a|D|^{-q}\big) = \tau\left[|a|^{\frac{p}{q}}\right]^{\frac{q}{p}},
 \end{aligned}\\
 \label{eq:Intro.Weyl-DpaDp-tau}
 \lim_{j\rightarrow \infty} j^{\frac{q}{p}} \lambda_j^\pm\big(|D|^{-\frac{q}{2}}a|D|^{-\frac{q}{2}}\big) =\tau\left[\big(a_\pm\big)^{\frac{p}{q}}\right]^{\frac{q}{p}} \quad (\text{if}\ a^*=a). 
\end{gather}
Here $a_\pm=\frac12(|a|\pm a)$ are the positive/negative parts of $a$.  
\end{propositionalph}

In the special case $q=p/2$, a version of~(\ref{eq:Intro.Weyl-|DpaDp|-tau}) for $a\geq 0$ is provided by~\cite[Theorem~1.4]{MSZ:LMP22}, and a version of~(\ref{eq:Intro.Weyl-DpaDp-tau}) is provided by~\cite[Theorem~1.5]{MSZ:LMP22}. Those results in~\cite{MSZ:LMP22} further require Lipschitz regularity and $p$-summability with $p>2$. These two extra conditions are removed in Proposition~\ref{prop:Intro.Weyl-BirS}. The spectral asymptotics~(\ref{eq:Intro.Weyl-DpaDp-tau})--(\ref{eq:Intro.Weyl-|DpaDp|-tau}) for $q\neq p/2$ are new. 

\subsection{Semiclassical Weyl laws} 
Combining the spectral asymptotics~(\ref{eq:Intro.Weyl-DpaDp-tau}) with the Birman--Schwinger principle enables us to get semiclassical Weyl laws for (fractional) Schr\"odinger operators,
 \begin{equation*}
 H_V^{(q)}(h):= h^{2q}\left(D^2\right)^{q}+V, \qquad h>0. 
\end{equation*}
More precisely,  in our setup the spectrum of $H_V^{(q)}(h)$ is discrete and bounded from below. Given any energy level $\lambda\in \R$, we denote 
by $N(H_V^{(q)}(h); \lambda)$ the number of eigenvalues of $H_V^{(q)}(h)$ that are~$<\lambda$ (see Eq.~(\ref{eq:ST.counting-Schroedinger1}) for the precise definition). 
We then obtain the following semiclassical Weyl laws. 
 
\begin{theoremalph}[Semiclassical Weyl Laws]\label{thm:Intro.SC-Weyl-lawDq} 
 Assume Condition~\textup{(W)} holds.  Given any $q>0$ and  $V^*=V\in \overline{\sA}$, for any energy level $ \lambda\in \R$, we have 
\begin{gather}\label{eq:Intro.SC-Weyl-Dq}
 \lim_{h\rightarrow 0^+} h^{p} N\big(H_V^{(q)}(h);\lambda\big)= \tau\Big[\left(V-\lambda\right)_{-}^{\frac{p}{2q}}\Big]. 
\end{gather}
\end{theoremalph}

For $q=1$ we recover from~(\ref{eq:Intro.SC-Weyl-Dq}) the semiclassical Weyl law of~\cite{MSZ:LMP22} without any extra regularity or requiring $p$ to be~$>2$.  The semiclassical Weyl laws~(\ref{eq:Intro.SC-Weyl-Dq}) for $q\neq 1$ are new even for $p>2$. 

It is known that the semiclassical Weyl law for $q=1$ implies a classical Weyl law for the eigenvalue distribution of $D^2$.  Therefore, we may interpret Theorem~\ref{thm:Intro.SC-Weyl-lawDq} as a kind of converse of this statement. 

\subsection{Integration formulas} 
The spectral asymptotics~(\ref{eq:Intro.Weyl-|DpaDp|-tau})--(\ref{eq:Intro.Weyl-DpaDp-tau}) have further implications in noncommutative geometry, where the role of the integral is played by positive traces on the weak trace class $\sL_{1,\infty}$, such as Dixmier traces (see~\cite{Co:NCG, Di:CRAS66}). We refer to Section~\ref{chap:NCInt} for the notions of measurable, strongly measurable, and spectrally measurable operators, and of the NC integral $\bint A$ of a measurable operator $A$; spectral measurability, which is defined in terms of the asymptotics of eigenvalues, is the strongest of these notions. The prototypical example is Connes' integration formula: if $(M^n,g)$ is a closed Riemannian manifold, then, for every $f\in C^\infty(M)$, the operator $f\Delta_g^{-n/2}$ is strongly measurable, and we have
\begin{equation}\label{eq:Intro.int-form}
 \bint f\Delta_g^{-\frac{n}{2}}= c(n) \int_Mf(x) d\nu_g(x), \qquad c(n):=(2\pi)^{-n}|\bB^n|, 
\end{equation}
where $\nu_g$ is the Riemannian measure (see~\cite{Co:CMP88, KLPS:AIM13}). In fact, results of Birman--Solomyak~\cite{BS:VLU77, BS:VLU79, BS:SMJ79}, which predate~(\ref{eq:Intro.int-form}), show that the \psidos\ of order~$-\dim M$ on a closed manifold $M$ are spectrally measurable. In the same way, the spectral asymptotics~(\ref{eq:Intro.Weyl-|DpaDp|-tau})--(\ref{eq:Intro.Weyl-DpaDp-tau}) yield the following far-reaching extensions of Connes' integration formula to our class of spectral triples. 

\begin{theoremalph}[Integration Formulas]\label{thm:Intro.Integration}
Assume Condition~\textup{(W)} holds. For any $a\in  \overline{\sA}$, the operators 
$a|D|^{-p}$, $|a|D|^{-p}|$, $|D|^{-p/2}a|D|^{-p/2}$, and  $||D|^{-p/2}a|D|^{-p/2}|$ are spectrally measurable, and we have 
 \begin{gather}\label{eq:Intro.bint-aDp-tau}
 \bint a|D|^{-p} = \bint |D|^{-\frac{p}{2}}a|D|^{-\frac{p}{2}}=\tau\big[a\big],\\
  \bint \big| a|D|^{-p}\big| = \bint \big| |D|^{-\frac{p}{2}}a|D|^{-\frac{p}{2}}\big|=  \tau\left[ |a|\right]. 
  \label{eq:Intro.bint-|DpaDp|-tau}
\end{gather}
 \end{theoremalph}

The integration formula for $a|D|^{-p}$ is alluded to in~\cite[p.~77]{MSZ:LMP22}. The integration formulas for $|a|D|^{-p}|$ and $||D|^{-p/2}a|D|^{-p/2}|$ seem to be new. 

\subsection{Extensions to unbounded potentials}\label{subsec:Intro.unbounded} 
It is natural to ask to which extent the spectral asymptotics~(\ref{eq:Intro.Weyl-|DpaDp|-tau})--(\ref{eq:Intro.Weyl-DpaDp-tau}), the semiclassical Weyl laws~(\ref{eq:Intro.SC-Weyl-Dq}), and the integration formulas~(\ref{eq:Intro.bint-aDp-tau})--(\ref{eq:Intro.bint-|DpaDp|-tau}) continue to hold whenever their right-hand sides make sense, i.e., for \emph{unbounded} potentials. On $\R^n$ and on bounded domains of $\R^n$, such extensions to $L^r$-potentials follow from the CLR inequality and Cwikel's estimates, combined with the perturbation theory of Birman--Solomyak (see Section~\ref{sec:Unbounded} for references). In our setting, assuming that $\tau$ is the restriction of a (positive) faithful normal finite trace on the von Neumann algebra $\sM$ generated by $\sA$, we formulate Cwikel-type estimates as a condition, Condition~($\textup{C}_r$), involving a norm $\|\cdot\|_{(r)}$ on $\sA$ and the noncommutative $L^r$-spaces of $\sM$ (see Section~\ref{sec:Unbounded} for the precise statement). If this condition holds, we denote by $\sV_r$ the Banach space completion of $\sA$ with respect to $\|\cdot\|_{(r)}$; for instance, for bounded domains $\Omega\subseteq\R^n$ we may take $\sV_r=L^r(\Omega)$ for $r>1$ (see Section~\ref{subsec:BD}). For $V=V^*\in\sV_r$ with $r=p(2q)^{-1}$, the fractional Schr\"odinger operators $H_V^{(q)}(h)$ then make sense as form sums, and we obtain the following extension of the semiclassical Weyl laws~(\ref{eq:Intro.SC-Weyl-Dq}). 

\begin{theoremalph}[Semiclassical Weyl Laws; $\sV_r$-version]\label{thm:Intro.SC-Weyl-law-sVr}
 Suppose that Condition~\textup{(W)} is satisfied. Let $q>0$, and assume that Condition~\textup{($\textup{C}_{r}$)} holds with $r=p(2q)^{-1}$.  Then, 
  the semiclassical Weyl laws~(\ref{eq:Intro.SC-Weyl-Dq}) hold for all $V=V^*\in \sV_r$, i.e., 
  for any energy level $ \lambda\in \R$, we have 
 \begin{gather}\label{eq:Intro.SC-Weyl-Dq-sVr1}
 \lim_{h\rightarrow 0^+} h^{p} N\big(H_V^{(q)}(h);\lambda\big)= \tau\Big[\left(V-\lambda\right)_{-}^{\frac{p}{2q}}\Big]. 
\end{gather}
\end{theoremalph}

We further have the following extension of the integration formulas~(\ref{eq:Intro.bint-aDp-tau})--(\ref{eq:Intro.bint-|DpaDp|-tau}). 

\begin{theoremalph}[Integration Formulas; $\sV_1$-version]\label{thm:Intro.Integration-sVr}
 Assume that Condition~\textup{(W)}  and Condition~\textup{($\textup{C}_{1}$)} are satisfied. For every $x\in \sV_1$, the operators  $|D|^{-p/2}x|D|^{-p/2}$ and $||D|^{-p/2}x|D|^{-p/2}|$ are spectrally measurable, and we have 
 \begin{gather*}
  \bint |D|^{-\frac{p}{2}}x|D|^{-\frac{p}{2}}=\tau\big[x\big],\\
   \bint \big| |D|^{-\frac{p}{2}}x|D|^{-\frac{p}{2}}\big|=  \tau\left[ |x|\right]. 
\end{gather*}
 \end{theoremalph}
 
 \subsection{Tauberian conditions}\label{subsec:Intro.Tauberian} 
There are various ways to establish Weyl laws. They are often deduced from Tauberian theorems. The approach of~\cite{MSZ:LMP22} relies on a new Tauberian condition on the zeta functions $\Tr[a^s|D|^{-s}]$, where $a$ ranges over all positive elements of the $C^*$-algebra $\overline{\sA}$. As mentioned above, the need for a more standard Tauberian condition was pointed out by Alain Connes. From this perspective, in \PartII{} we establish the following Tauberian criteria for Condition~(W). They only need to be checked on a holomorphically closed Fr\'echet subalgebra containing $\sA$, such as $\sA$ itself when it is holomorphically closed. 

\begin{theoremalph}[Tauberian Theorem; see \PartII{}]\label{thm:Intro.Tauberian} 
Assume there exists a holomorphically closed Fr\'echet subalgebra $\sB\subseteq \overline{\sA}$ containing $\sA$ such that one of the following conditions is satisfied: 
\begin{itemize}
\item \textup{\textbf{Condition (Z)}}. For every $a\in \sB$, the function 
\begin{equation*}
\Tr\big[a|D|^{-s}\big] -p\tau(a)(s-p)^{-1}, \qquad \Re s>p,
\end{equation*}
 has a continuous extension to the closed half-plane $\Re s\geq p$.\smallskip  
 
  \item \textup{\textbf{Condition (H)}}.  
  There is $\delta>0$ such that, for all $a\in \sB$, as $t\rightarrow 0^+$ we have
 \begin{equation*}
  \Tr\big[ae^{-tD^2}\big]=  \Gamma\left(1+\frac{p}{2}\right) t^{-\frac{p}{2}}\left[\tau(a) + \op{O}\big(t^{\delta}\big)\right].   
\end{equation*}
\end{itemize}
 Then Condition~\textup{(W)} is satisfied. 
\end{theoremalph}

\begin{corollaryalph}
 The main results of this article, Proposition~\ref{prop:Intro.Weyl-BirS} through Theorem~\ref{thm:Intro.Integration-sVr}, hold true if we replace Condition~\emph{(W)} by Condition~\emph{(H)} or Condition~\emph{(Z)}. 
\end{corollaryalph}

Theorem~\ref{thm:Intro.Tauberian} is a refinement of the Tauberian theorem of~\cite{MSZ:LMP22}. First, all the extra regularity and summability assumptions in~\cite{MSZ:LMP22} are removed. Second, the proof of Theorem~\ref{thm:Intro.Tauberian} involves showing that Condition~(Z) and Condition~(H) imply a weaker version of the Tauberian condition of~\cite{MSZ:LMP22} and showing that this weaker version implies Condition~(W) (see \PartII{}). We thus obtain the Weyl laws~(\ref{eq:Intro.conf-Weyl}) under a weaker form of the Tauberian condition of~\cite{MSZ:LMP22}.  

Condition~(Z) and Condition~(H) are the types of Tauberian conditions sought by Alain Connes in his comments mentioned above. In practice, they are easier to check than the Tauberian condition of~\cite{MSZ:LMP22} and hold in much greater generality. As a result they significantly open the scope of applicability of the semiclassical Weyl laws in NCG, and so they allow us to deal with a variety of new examples (see below).  

 \subsection{Operator-theoretic tools}\label{subsec:Intro.lemmas}
 The main results of this paper are all consequences of Proposition~\ref{prop:Intro.Weyl-BirS} and Theorem~\ref{thm:Intro.Tauberian}. The proofs of these results rely on some technical lemmas whose proofs are postponed to \PartII{}. 
 
 The first technical lemma is the following lemma.  
 
\begin{lemmalph}[Factorization Lemma; see \PartII{}]\label{lem:Intro.factorization}
 Let $\alpha, \beta, \gamma\in [0,1]$ be such that,  either $\beta+\gamma<\alpha+1$, or $\beta+\gamma=\alpha+1$ and $\max(\beta,\gamma)<1$. 
 \begin{enumerate}
 \item[(i)] There is a continuous linear operator $\Phi_{\beta,\gamma}^{\alpha}:\sL(\sH)\rightarrow \sL(\sH)$ such that, for all $a\in \sA$, we have
 \begin{equation*}
  \left[ |D|^{-\alpha},a\right]  =|D|^{-\beta}\Phi_{\beta,\gamma}^{\alpha}\left([D,a]\right)|D|^{-\gamma} 
  + |D|^{-(\alpha+1)}F[D,a]\Pi_0 +\Pi_0[D,a]F|D|^{-(\alpha+1)}. 
\end{equation*}
Here $F=D|D|^{-1}$ is the sign of $D$ and $\Pi_0$ is the orthogonal projection onto $\ker D$. 

\item[(ii)] Assume further that $\beta+\gamma<\alpha+1$. If $\max(\beta,\gamma)<1$, then $\Phi_{\beta,\gamma}^{\alpha}$ gives rise to a continuous linear operator, 
\begin{equation*}
 \Phi_{\beta,\gamma}^{\alpha}: \sL(\sH)\longrightarrow \sL_{r,\infty}, \qquad \text{with}\ r:=(1+\alpha -\beta-\gamma)^{-1}p. 
\end{equation*}
If $\max(\beta,\gamma)=1$, then the same result holds for any $r>(1+\alpha -\beta-\gamma)^{-1}p$. 
\end{enumerate}
\end{lemmalph}

 The above lemma seems to be new. However, its proof relies on basic techniques of double operator integrals (DOIs) introduced by Birman--Solomyak~\cite{BS:PMF66, BS:PMF73, BS:IEOT03} in the 60s (see also~\cite{ST:LNM19}). An important consequence is the following commutator estimate, which is used in the proof of Proposition~\ref{prop:Intro.Weyl-BirS}. 

 \begin{propositionalph}[Commutator Estimates; see Section~\ref{sec:Com}]\label{prop:Intro.Com-Dq}
Given any $q>0$, for all $a\in \sA$, we have
 \begin{equation*}
 \left[ |D|^{-q},a\right] \in \sL_{(q+1)^{-1}p,\infty}.
\end{equation*}
In fact, there is a constant $C_{D,q}>0$ such that
\begin{equation}\label{eq:Com.Com-Dq-estimate}
 \left\| \left[ |D|^{-q},a\right]\right\|_{(q+1)^{-1}p,\infty}\leq C_{D,q}  \left\| \left[ D,a\right]\right\| \qquad \forall a \in \sA.   
\end{equation} 
 \end{propositionalph}

Proposition~\ref{prop:Intro.Com-Dq} subsumes and improves various commutator results for spectral triples (see, e.g., \cite{BJ:CRAS83, CP:CJM98, CM:CRAS86, MSX:PSPM23, PS:Crelle09, SWW:CRAS98, SYZ:JNCG26, Vo:JFA90}). It yields Schatten properties as sharp as those established by pseudodifferential techniques in the cases where suitable pseudodifferential calculi are available (see, e.g., \cite{MSX:CMP19}). However, it holds in far greater generality and requires very little regularity for the elements of $\sA$. For instance, for Dirac spectral triples over closed Riemannian spin manifolds, Proposition~\ref{prop:Intro.Com-Dq} holds for all Lipschitz functions. In addition, the type of estimate provided by~(\ref{eq:Com.Com-Dq-estimate}) is usually not reachable by pseudodifferential techniques.

Another consequence of the factorization Lemma, as well as an important ingredient in the proof of Proposition~\ref{prop:Intro.Weyl-BirS}, is an enhancement for spectral triples of a deep lemma of Connes--Sukochev--Zanin~\cite[Lemma~5.3]{CSZ:MS17} (hereafter called the CSZ Lemma). More precisely, by combining the approach of~\cite{CSZ:MS17, SZ:Ast23} and using a few simplifications we obtain the following enhancement of the CSZ Lemma. 
 
 \begin{lemmalph}[Enhanced CSZ Lemma; see \PartII{}]\label{lem:Intro.enhanced-CSZ}
 Given any $\alpha \in (0,1]$ and $s>1$, for every $a\in {\sA}_{++}$, we have
 \begin{equation*}
 \left( a^{\frac12}|D|^{-\alpha} a^{\frac12}\right)^{s} - |D|^{-\alpha s} a^{s} \in \sL_{(\alpha s +1-\epsilon)^{-1}p,\infty} \qquad \forall \epsilon\in (0,1). 
\end{equation*}
\end{lemmalph}
 
The original CSZ Lemma, stated without proof by Connes~\cite[Lemma~IV.3.11]{Co:NCG} and proved in~\cite{CSZ:MS17}, asserts in our setting that, for $0<q<p$, the differences $a^{p/2}|D|^{-p}a^{p/2}-(a^{1/2}|D|^{-q}a^{1/2})^{p/q}$, $a\in \sA_{++}$, are in $(\sL_{1,\infty})_0$, which, for $q=2$ (the case needed for the semiclassical Weyl laws of~\cite{MSZ:LMP22}), forces $p>2$. 

The enhancement is twofold. First, the combination of the Enhanced CSZ Lemma with the commutator estimates of Proposition~\ref{prop:Intro.Com-Dq} allows us to deal with the differences $(a^{\frac12}|D|^{-\alpha}a^{\frac12})^{s} - a^{\frac{s}{2}}|D|^{-\alpha s}a^{\frac{s}{2}}$ for all $\alpha\in (0,1]$ and $s>1$. As the exponent $\alpha s$ ranges over all of $(0,\infty)$, by playing with the values of $\alpha$ and $s$ we can relate the operators $a^{\frac12}|D|^{-q}a^{\frac12}$ for \emph{all values} $q>0$ to the case $q=2$ (see the proof of Proposition~\ref{prop:Intro.Weyl-BirS} in Section~\ref{sec:SpAsymp}). Second, we obtain better Schatten properties (e.g., for $q=p$ we get membership in ideals $\sL_{r,\infty}$ with $r<1$, which implies that the difference is trace-class). Incidentally, the above enhanced CSZ Lemma is instrumental in removing the restriction $p>2$ from the approach of~\cite{MSZ:LMP22}. 

Finally, the proof of the Tauberian theorem (Theorem~\ref{thm:Intro.Tauberian}) in \PartII{} relies on a holomorphic version of the Enhanced CSZ Lemma and its consequence for zeta functions. In the setting of unital spectral triples, the latter refines results of Sukochev--Zanin~\cite[Theorem~5.4.2]{SZ:Ast23} and McDonald--Sukochev--Zanin~\cite[Theorem~2.2]{MSZ:LMP22} without any extra regularity assumption or restriction on the degree of summability. 

\subsection{Examples}
There are numerous instances of spectral triples satisfying the Tauberian conditions (H) or (Z), and hence Condition~(W). In practice, the verification of these Tauberian conditions follows from standard heat kernel asymptotics or zeta-function meromorphic extensions. On the one hand, this allows us to extend the semiclassical Weyl laws and integration formulas to numerous new settings. In particular, this allows us to prove the full conjecture of~\cite{MP:JMP22} on semiclassical Weyl laws on quantum tori. On the other hand, this provides us with a unified approach to various well-known results.

We illustrate this with the classical examples below, which are presented in more detail in Section~\ref{sec:Examples}. In these examples we recover various well-known results, thereby obtaining a unified approach to these results. The application to quantum tori is described in Section~\ref{subsec:Intro.QT}, and further examples, developed in the companion papers~\cite{Po:Open, Po:Bir-Sol, Po:Carnot}, are described in Section~\ref{subsec:Intro.further}. 
It is also worth noting that these results were meant to be far-reaching extensions of the \emph{classical} Weyl laws for Laplacians. Thanks to our approach, they can now be established by techniques that were originally introduced to prove those classical Weyl laws. 
 
\newcounter{scex}
\renewcommand{\thescex}{(\alph{scex})}

\smallskip
\noindent\refstepcounter{scex}\thescex\label{ex:BD}\ \emph{Bounded Euclidean domains} (see Section~\ref{sec:Examples}). By classical results of Minakshisundaram~\cite{Mi:CJM49} (see also~\cite{BG:CPDE90, Ca:SMK34}), the spectral triples associated with the square roots of the Dirichlet and Neumann Laplacians on a bounded domain with smooth boundary satisfy Condition~(H), and the Cwikel--Solomyak estimates (see~\cite{BS:AMST80, Ro:SMD72, Ro:SM76, RS:EMS21, So:IJM94, So:PLMS95}) give Condition~($\textup{C}_r$). We thus recover the semiclassical Weyl laws in this setting, including their extensions to $L^p$- and $\LlogL$-potentials (\emph{cf}.~\cite{BS:FAA70}), while the integration formulas are new (Theorem~\ref{thm:BD.Integration-formulas}).\medskip

\noindent\refstepcounter{scex}\thescex\label{ex:CM}\ \emph{Closed Riemannian manifolds} (see Section~\ref{sec:Examples}). By the heat kernel asymptotics of Minak\-shi\-sun\-da\-ram--Pleijel~\cite{MP:CJM49}, Condition~(H) holds for spectral triples associated with Dirac-type operators and square roots of Laplace-type operators, and the Cwikel--Solomyak estimates on closed manifolds (see~\cite{Cw:AM77, BS:UMN77, So:IJM94, So:PLMS95, SZ:MS22}) give Condition~($\textup{C}_r$). We thus recover the semiclassical Weyl laws for fractional Schr\"odinger operators, including their extensions to $L^p$- and $\LlogL$-potentials (\emph{cf}.~\cite{BS:FAA70, Ro:SMD72, Ro:SM76, Ro:JST22, So:PLMS95}), as well as the recent extensions of the integration formula to $\LlogL$-functions (\emph{cf}.~\cite{Po:MPAG22, Ro:JST22, SZ:FAA23}).\medskip

\noindent\refstepcounter{scex}\thescex\label{ex:Steklov}\ \emph{Steklov eigenvalue problem} (see Section~\ref{sec:Examples}). On the boundary of a compact Riemannian manifold with boundary, the Dirichlet-to-Neumann operator gives rise to a spectral triple satisfying Condition~(Z). Proposition~\ref{prop:Intro.Weyl-BirS} then recovers the Weyl laws for weighted Steklov problems of~\cite{Ag:RJMP06, Ro:JST22, Su:RJMP99}. These spectral asymptotics also imply unweighted Steklov Weyl laws for Lipschitz boundaries (see~\cite{KLP:ARMA23, Ro:JST23}).\medskip

\subsection{Application to quantum tori}\label{subsec:Intro.QT}
Let $\T^n_\theta$ be the $n$-dimensional quantum torus associated with a real anti-symmetric $n\times n$-matrix $\theta$, and let $\Delta$ be its flat Laplacian. In~\cite[Conjecture~8.8]{MP:JMP22} Ed McDonald and the author conjectured semiclassical Weyl laws for the fractional Schr\"odinger operators $h^{2q}\Delta^q+V$, $q>0$, with $L^p$-potentials $V$. It was also suggested in~\cite[Remark~8.9]{MP:JMP22} and~\cite[Remark~6.7]{Po:MPAG22} that in the critical case $q=n/2$ these should hold for $\LlogL$-potentials. The case $q=1$ in dimension~$n\geq 3$ with continuous potentials was proved in~\cite{MSZ:LMP22}. However, as that approach requires $p$-summability with $p>2$, it does not cover 2-dimensional quantum tori, even for $q=1$. 

In our setting, the square-root Laplacian spectral triples satisfy Condition~(H), and Condition~($\textup{C}_r$) follows from the Cwikel-type estimates of~\cite{MSX:CMP19, MP:JMP22} and, in the critical case, from the $\LlogL$-Cwikel estimate of Liu--Zhang~\cite{LZ:arXiv26}. Theorem~\ref{thm:Intro.SC-Weyl-law-sVr} then gives these semiclassical Weyl laws in any dimension and for any $q>0$, including the $\LlogL$-case. Theorem~\ref{thm:Intro.QT} below thus gives a positive answer to the conjectures in~\cite{MP:JMP22} and~\cite{Po:MPAG22} alluded to above. 

In the statement below, $\tau_0$ is the standard trace of $\T^n_\theta$, and $L_r(\T^n_\theta)$ and $\LlogL(\T^n_\theta)$ are the associated noncommutative $L^r$-spaces and $\LlogL$-Orlicz space (see Section~\ref{sec:QT}). 

\begin{theoremalph}[Semiclassical Weyl Laws on $\T^n_\theta$; see Section~\ref{sec:QT}]\label{thm:Intro.QT}
Let $q>0$, and assume that one of the following holds:
\begin{enumerate}
 \item[(i)] $q< n/2$ and $V=V^*\in L_{\frac{n}{2q}}(\T^n_\theta)$. 
 
 \item[(ii)] $q> n/2$ and $V=V^*\in L_1(\T^n_\theta)$. 
 
 \item[(iii)] $q=n/2$ and $V=V^* \in \LlogL(\T^n_\theta)$. 
\end{enumerate}
Then, for every energy level $\lambda\in \R$, we have
\begin{equation*}
 \lim_{h\rightarrow 0^+} h^nN\big(h^{2q}\Delta^q+V;\lambda\big) = \hat{c}(n) \tau_0\big[(V-\lambda)_{-}^{\frac{n}{2q}}\big], \qquad \hat{c}(n):=|\bB^n|. 
\end{equation*}
\end{theoremalph}

We also obtain integration formulas on quantum tori for $\LlogL$-elements in terms of spectral measurability (Theorem~\ref{thm:QT.Integration-LlogL}). This refines the integration formulas of~\cite{MSZ:MA19, MP:JMP22, Po:JMP20} and Liu--Zhang~\cite{LZ:arXiv26}, which are formulated in terms of strong measurability.

We refer to~\cite{Po:curvedQT} for a curved version of Theorem~\ref{thm:Intro.QT}, which was conjectured in~\cite{MP:AIM23}, as well as for curved versions of the above integration formulas. Here curved quantum tori are quantum tori equipped with Riemannian metrics in the sense of~\cite{HP:JGP20, Ro:SIGMA13}.

\subsection{Further examples}\label{subsec:Intro.further}
The following examples are developed in the companion papers~\cite{Po:Open, Po:Bir-Sol, Po:Carnot}. In Examples~\ref{ex:open} and~\ref{ex:Carnot} we obtain new results, while in Example~\ref{ex:BirSol} we recover various well-known results.\medskip

\noindent\refstepcounter{scex}\thescex\label{ex:open}\ \emph{Open manifolds} (see~\cite{Po:Open}). Using results of Gol\'enia--Moroianu~\cite{GM:AHP08} and Moroianu~\cite{Mo:MA08}, we construct spectral triples satisfying Condition~(Z) on open manifolds with conformally cusp metrics, associated with Dirac operators and square roots of magnetic Laplacians. This yields semiclassical Weyl laws on open manifolds which, unlike in~\cite{Br:PAMS25, DY:arXiv25}, need not be complete, as well as an extension of Connes' integration formula to this setting.\medskip

\noindent\refstepcounter{scex}\thescex\label{ex:BirSol}\ \emph{Birman--Solomyak's spectral asymptotics} (see~\cite{Po:Bir-Sol}). Replacing $C^\infty(M)$ by the algebra $\Psi^0(M)$ of zeroth-order \psidos\ in the spectral triple $(C^\infty(M), L^2(M),\sqrt{\Delta_g})$ of Example~\ref{ex:CM} still yields a spectral triple, and Condition~(Z) follows from the construction of the noncommutative residue trace by Guillemin~\cite{Gu:AIM85} and Wodzicki~\cite{Wo:NCR}. Proposition~\ref{prop:Intro.Weyl-BirS} then provides a new, spectral-theoretic proof of the spectral asymptotics of Birman--Solomyak~\cite{BS:SMJ79} for negative order \psidos, recovers recent spectral asymptotics of Sukochev--Yang--Zanin~\cite{SYZ:JNCG26}, and yields an enhancement of the version of Connes' trace theorem of Kordyukov--Sukochev--Zanin~\cite{KSZ:JNCG24}.\medskip

\noindent\refstepcounter{scex}\thescex\label{ex:Carnot}\ \emph{Carnot manifolds} (see~\cite{Po:Carnot}). Carnot manifolds (also called filtered manifolds) include nilmanifolds, real hypersurfaces in $\C^n$, contact manifolds, manifolds with parabolic geometric structures, and equiregular sub-Riemannian manifolds (see, e.g., \cite{ABB:Book19, CP:DGA19, Mo:AMS02} and the references therein). Their main differential operators are not elliptic. The substitute for ellipticity is the Rockland condition on the principal symbol, which is formulated in terms of the irreducible unitary representations of the tangent groups, i.e., graded nilpotent Lie groups whose structure may vary from point to point (see, e.g.,~\cite{FR:Birkhauser16}).

Combining the results of this paper with results of Dave--Haller~\cite{DH:JGA20} on complex powers and heat kernel asymptotics of positive Rockland differential operators, it is shown in~\cite{Po:Carnot} that (fractional powers of) such operators on closed Carnot manifolds give rise to spectral triples satisfying Condition~(H), and hence Condition~(W). This yields semiclassical Weyl laws and integration formulas on arbitrary closed Carnot manifolds, including semiclassical Weyl laws for Schr\"odinger operators associated with sub-Laplacians on closed equiregular sub-Riemannian manifolds. Semiclassical Weyl laws on graded nilpotent Lie groups and nilmanifolds were obtained by Fischer--Mikkelsen~\cite{FM:AMP25} and McDonald--Sukochev--Zanin~\cite{MSZ:JFA23} (see also~\cite{LMSZ:arXiv26}). Our results are the first semiclassical Weyl laws in settings where the tangent group varies from point to point. Moreover, the integration formulas connect with a recent extension of Connes' trace theorem to the filtered pseudodifferential calculus by McDonald~\cite{McD:arXiv26}, and in the contact case they enhance recent results of Kordyukov--Sukochev--Zanin~\cite{KSZ:LNM24}, bypassing the $C^*$-algebraic approach to the principal symbol used there.

\subsection{Organization of the paper} 
This paper is organized as follows. In Section~\ref{chap:NCInt}, we review the main background on weak Schatten ideals, Connes' integration, and their relationships with Weyl laws and the Birman--Schwinger principle. In Section~\ref{chap:Main-results}, we introduce spectral triples and present Condition~(W) together with some of its consequences. In Section~\ref{sec:Com}, we establish the commutator estimates of Proposition~\ref{prop:Intro.Com-Dq}, which we use in Section~\ref{sec:SpAsymp} to prove the spectral asymptotics of Proposition~\ref{prop:Intro.Weyl-BirS}. In Section~\ref{sec:SCWeyl} and Section~\ref{sec:IntForm}, we deduce from these the semiclassical Weyl laws of Theorem~\ref{thm:Intro.SC-Weyl-lawDq} and the integration formulas of Theorem~\ref{thm:Intro.Integration}, respectively. In Section~\ref{sec:Unbounded}, we extend these results to unbounded potentials, thereby obtaining Theorem~\ref{thm:Intro.SC-Weyl-law-sVr} and Theorem~\ref{thm:Intro.Integration-sVr}. In Section~\ref{sec:Examples}, we illustrate the main results on bounded Euclidean domains, closed Riemannian manifolds, and the Steklov eigenvalue problem, and in Section~\ref{sec:QT} we apply them to quantum tori; in particular, we prove Theorem~\ref{thm:Intro.QT}. 

\subsection*{Acknowledgements}
The author wishes to thank Alan Carey, Yves Colin de Verdi\`ere, Joachim Cuntz, Magnus Goffeng, Gerd Grubb, Edward McDonald, Iosif Polterovich, Adam Rennie, Grigori Rozenblum, Fedor Sukochev, Emmanuel Tr\'elat, Teun van Nuland, and Zuoqin Wang for various discussions related to the subject matter of this article. The research for this article was partially funded by NSFC grant No.~11971328 (China). 

\section{Connes' Integration and Weyl Laws}\label{chap:NCInt} 
In this section, we review the main facts regarding Connes' integration and its relationship with Birman--Solomyak's perturbation theory.  The presentation closely follows that of~\cite{Po:JNCG23}. 

Throughout this section we let $\sH$ be a (separable) Hilbert space, and we denote by $\sK$ the (closed) ideal of compact operators on $\sH$. 

\subsection{Weak Schatten classes} We recall here the main definitions and properties regarding weak Schatten classes. We refer to~\cite{GK:AMS69, Si:AMS05} for more detailed accounts on this topic. 

Given any operator $T\in \sK$, we denote by $(\mu_j(T))_{j\geq 0}$ its sequence of \emph{singular values}, i.e., $\mu_j(T)$ is the $(j+1)$-th eigenvalue counted with multiplicity of the absolute value $|T|=\sqrt{T^*T}$. Recall that by the \emph{min-max principle} we have
\begin{align}
 \mu_j(T)&=\min \left\{\|T_{|E^\perp}\|;\ \dim E=j\right\}, \nonumber\\
 &= \min\{ \|T-R\|; \ R\in \sL(\sH), \ \rk(R)\leq j\}. 
 \label{eq:min-max} 
\end{align}
Note that $\mu_0(T)=\|T\|$.  We record the following properties of the singular values (see, e.g., \cite{GK:AMS69, Si:AMS05}), 
\begin{gather}
 \mu_j(T)=\mu_j(T^*)=\mu_j(|T|), \\
  \mu_{j+k}(S+T)\leq \mu_j(S) + \mu_k(T),
 \label{eq:Quantized.properties-mun2}\\
 \mu_j(ATB)\leq \|A\| \mu_j(T) \|B\|,\quad 
 \mu_j(U^*TU)= \mu_j(T),
\end{gather}
where $A, B\in \sL(\sH)$ and $U\in\sL(\sH)$ is unitary. 

By definition the \emph{Schatten class} $\sL_p$, $p>0$, consists of operators $T\in \sK$ such that $|T|^p$ is trace-class. It is equipped with the quasi-norm, 
\begin{equation*}
 \|T\|_{p}:=\Tr\big(|T|^p\big)^{\frac1p}= \bigg( \sum_{j\geq 0}\mu_j(T)^p\bigg)^{\frac1p}, \qquad T\in \sL_p. 
\end{equation*}
We obtain a quasi-Banach ideal. For $p\geq 1$ the above quasi-norm is actually a norm, and so in this case $\sL_p$ is a Banach ideal. 

The \emph{weak Schatten class}  $\sL_{p,\infty}$, $p>0$, is defined by
\begin{equation*}
 \sL_{p,\infty}:=\left\{T\in \sK; \ \mu_j(T)=\op{O}\big(j^{-\frac1p}\big)\right\}. 
\end{equation*}
 This is a two-sided ideal. We equip it with the quasi-norm,
\begin{equation*}
 \|T\|_{p,\infty}:=\sup_{j\geq 0}\;(j+1)^{\frac{1}{p}}\mu_j(T), \qquad T\in \sL_{p,\infty}. 
\end{equation*}
With this quasi-norm $\sL_{p,\infty}$ is a quasi-Banach ideal. In particular, we have
\begin{equation*}
 \|ATB\|_{p,\infty}\leq \|A\| \|T\|_{p,\infty} \|B\| \qquad A,B\in \sL(\sH), \ T\in \sL_{p,\infty}. 
\end{equation*}
For $p>1$ the above quasi-norm is equivalent to an actual norm. Thus, in this case $\sL_{p,\infty}$ is a Banach ideal. 

We denote by $(\sL_{p,\infty})_{0}$ the closure in $\sL_{p,\infty}$ of the ideal of finite-rank operators. Equivalently, 
\begin{equation*}
 \big(\sL_{p,\infty}\big)_{0}=\left\{T\in \sK; \ \mu_j(T)=\op{o}\big(j^{-\frac1p}\big)\right\}.
\end{equation*}
We have the following (strict) continuous inclusions, 
\begin{equation*}
 \sL_p \subsetneq  \big(\sL_{p,\infty}\big)_{0}\subsetneq \sL_{p,\infty}  \subsetneq \sL_{q}, \qquad 0<p<q.
\end{equation*}

We mention the following version of H\"older's inequality for weak Schatten classes. 

\begin{proposition}[H\"older's Inequality; see \cite{GK:AMS69, Si:AMS05, SZ:PAMS21}]\label{prop:Hoelder}
 Suppose that  $p^{-1}+q^{-1}=r^{-1}$.
 \begin{enumerate}
 \item If $S\in \sL_{p,\infty}$ and $T\in \sL_{q,\infty}$, then $ST\in \sL_{r,\infty}$, and we have
\begin{equation*}
 \|ST\|_{r,\infty} \leq  C_{pq}\|S\|_{p,\infty}  \|T\|_{q,\infty}, 
\end{equation*}
where $C_{pq}=p^{-\frac1{q}}q^{-\frac1{p}} (p+q)^{\frac1{p}+\frac1{q}}$ and the inequality is sharp.

\item If in addition $S\in (\sL_{p,\infty})_0$ or $T\in (\sL_{q,\infty})_0$, then 
$ST\in  (\sL_{r,\infty})_0$.
\end{enumerate}
\end{proposition}

We will also need the following version of the BKS inequality of Birman--Koplienko--Solomyak. 

\begin{proposition}[BKS Inequality; see~\cite{BKS:IVUZM75, BS:JSM92}] \label{prop:BKS}
Let $A$ and $B$ be positive operators on $\sH$ such that $A-B\in \sL_{p,\infty}$, $p>0$. Given any $\alpha \in (0,1)$, the difference 
$A^\alpha-B^\alpha$ is in $\sL_{\alpha^{-1}p,\infty}$, and we have 
\begin{equation*}
 \big\|A^\alpha-B^\alpha\big\|_{\alpha^{-1}p,\infty} \leq C_{p\alpha} \|A-B\|_{p,\infty}^\alpha,
\end{equation*}
 where the constant $C_{p\alpha}$ depends only on $p$ and $\alpha$. If in addition, $A-B\in (\sL_{p,\infty})_0$, then $A^\alpha-B^\alpha$ is actually in $(\sL_{\alpha^{-1}p,\infty})_0$. 
\end{proposition}

\begin{remark}
 The last part in the statement above is the content of~\cite[Theorem~3]{BKS:IVUZM75}.   
\end{remark}

\subsection{Connes' integration} 
One of the main goals of noncommutative geometry~\cite{Co:NCG} is to translate the main tools of differential geometry into the Hilbert space formalism of quantum mechanics.  In this framework the notion of integral corresponds to positive traces on the weak trace class $\sL_{1,\infty}$. Such traces vanish on finite-rank operators (see, e.g., \cite{LSZ:Book}) and are always continuous (see, e.g.,~\cite[Remark~2.3]{Po:JMP20}), and so they vanish on $(\sL_{1,\infty})_0$. 

There is a whole zoo of positive traces on $\sL_{1,\infty}$ (see, e.g., \cite{LSZ:Book, LSZ:Survey19} and the references therein). One important class of such traces is provided by Dixmier traces $\Tr_\omega:\sL_{1,\infty} \rightarrow \C$ (see~\cite{Di:CRAS66}; see also~\cite{Co:NCG, LSZ:Book, Po:JNCG23}).  We then say that an operator $A\in \sL_{1,\infty}$ is \emph{measurable} if the value $\Tr_\omega(A)$ does not depend on the choice of the Dixmier trace. Equivalently (see~\cite{LSZ:Book, Po:JNCG23}), the operator $A$ is measurable if and only if it satisfies the following Tauberian condition, 
\begin{equation}\label{eq:NCInt.integration}
 \bint A:=  \lim_{N\rightarrow \infty} \frac{1}{\log (N+1)}\sum_{j<N}\lambda_j(A) \quad \text{exists}. 
\end{equation}
Here $(\lambda_j(A))_{j\geq 0}$ is any eigenvalue sequence of $A$ in such a way that 
\begin{equation*}
|\lambda_0(A)|\geq |\lambda_1(A)|\geq \cdots \geq  |\lambda_j(A)|\geq \cdots \geq 0, 
\end{equation*}
where each eigenvalue is repeated according to its (algebraic) multiplicity. The limit $\bint A$ is then called the \emph{NC integral} of $A$. 
 
 There are numerous positive traces on $\sL_{1,\infty}$ that are not Dixmier traces (see, e.g.,~\cite{SSUZ:AIM15}). Therefore, it stands to reason to consider a stronger notion of measurability. We say that an operator $A\in \sL_{1,\infty}$ is \emph{strongly measurable}  (or  \emph{positively measurable}) if all positive normalized traces take on the same value on $A$. Here a trace $\varphi$ on $ \sL_{1,\infty}$ is called normalized if 
 \begin{equation*}
 \big(A\geq 0 \ \text{and} \ \lambda_j(A)=(j+1)^{-1} \big)\ \Longrightarrow \ \varphi(A)=1. 
\end{equation*}

The Dixmier traces are positive normalized traces. Thus, if $A\in \sL_{1,\infty}$ is strongly measurable, then $A$ is measurable and, for any positive normalized trace $\varphi$ on  $\sL_{1,\infty}$, we have
\begin{equation*}
 \varphi(A)=\bint A = \lim_{N\rightarrow \infty} \frac{1}{\log (N+1)}\sum_{j<N}\lambda_j(A).
\end{equation*}
We refer to~\cite[Section~7]{SSUZ:AIM15} for a characterization of strongly measurable operators in terms of their eigenvalue sequences, as well as for examples of measurable operators that are not strongly measurable. 

As an example, let $\sH=L^2(M)$, where $(M^n,g)$ is a closed Riemannian manifold. Denote by $\Delta_g$ the corresponding Laplacian acting on (smooth) functions. This operator is essentially selfadjoint and has pure discrete non-negative spectrum. The partial inverse $\Delta_g^{-1}$ is in the weak Schatten class $\sL_{n/2,\infty}$. A well-known result of Connes~\cite{Co:CMP88} asserts that, for every $f\in C^\infty(M)$, the operator $f\Delta_g^{-n/2}$ is measurable, and we have
\begin{equation}\label{eq:NCInt.integ-formula}
 \bint f \Delta_g^{-\frac{n}{2}} = c(n) \int_M f(x) d\nu_g(x), \qquad c(n):= (2\pi)^{-n}|\bB^{n}|,
\end{equation}
where $d\nu_g(x)$ is the Riemannian measure. The above formula is often called \emph{Connes' integration formula}. It shows that the NC integral recaptures the Riemannian density. In particular, $\bint \Delta_g^{-n/2}=c(n)\Vol_g(M)$. 

It was further shown in~\cite{KLPS:AIM13} that the above formula continues to hold for functions in $L^2(M)$ and the operators  $ f \Delta_g^{-\frac{n}{2}}$ are actually strongly measurable. Recently, versions of the formula for functions in the Orlicz $\LlogL$-space have been established (see~\cite{Ro:JST22, SZ:FAA23}; see also~\cite{Po:MPAG22}). 
We will address further extensions of Connes' integration formula in the Riemannian case in Section~\ref{subsec:CM}.  
 
\subsection{Connes' integration and Weyl laws} 
An important class of strongly measurable operators is provided by operators that satisfy some form of Weyl laws.  Given any $A=A^*\in \sK$ we denote by $\pm \lambda_j^\pm(A)$ its positive and negative eigenvalues in such a way that
\begin{equation*}
 \lambda_0^\pm(A) \geq  \lambda_1^\pm(A)\geq \cdots >0,
\end{equation*}
where each eigenvalue is repeated according to multiplicity. If $A\geq 0$, then we set $\lambda_j(A)=\lambda_j^+(A)$. 

Let $p>0$. We say that $A=A^*\in \sL_{p,\infty}$, $p>0$, is a \emph{Weyl operator} if $
\lim j^{1/p}\lambda_j^\pm(A)$ exist. In this case we set
\begin{equation*}
 \Lambda^\pm(A)=\lim_{j\rightarrow \infty}j^{\frac{1}{p}}\lambda_j^\pm(A). 
\end{equation*}
If $A\geq 0$, then we just set $\Lambda(A)=\Lambda^{+}(A)$. 

If $A\in \sL_{p,\infty}$ is not selfadjoint, we call it a \emph{Weyl operator} if its real part $\Re A:=\frac12(A+A^*)$ and its imaginary part $\Im A:=\frac1{2i}(A-A^*)$ are both Weyl operators. In this case we define
\begin{equation*}
 \Lambda^\pm(A)= \Lambda^\pm\big(\Re A) + i \Lambda^\pm\big(\Im A). 
\end{equation*}
 We denote by $\sW_{p,\infty}$ the class of Weyl operators in $\sL_{p,\infty}$.  
 
 We also denote by $|\sW|_{p,\infty}$ the class of operators $A\in \sL_{p,\infty}$ such that its absolute value $|A|$ is a Weyl operator. Thus, 
 \begin{equation*}
 A \in |\sW|_{p,\infty} \ \Longleftrightarrow \ \lim_{j\rightarrow \infty} j^{1/p} \mu_j(A)\ \text{exists}. 
\end{equation*}
Note that if $A \in |\sW|_{p,\infty}$, then the above limit is just $\Lambda(|A|)$. 
 
\begin{proposition}[Birman--Solomyak~\cite{BS:FAA70, BS:Book}]\label{prop:NCInt.Bir-Sol}  
 Let $p>0$. The following hold. 
 \begin{enumerate}
 \item $\sW_{p,\infty}$ and $|\sW|_{p,\infty}$ are both closed subsets of $\sL_{p,\infty}$ containing $(\sL_{p,\infty})_0$. 
 
 \item The maps $A \rightarrow \Lambda^\pm(A)$ and $A\rightarrow \Lambda(|A|)$ are continuous on  $\sW_{p,\infty}$ and $|\sW|_{p,\infty}$, respectively. 
 
 \item If $A\in\sW_{p,\infty}$ and $B\in (\sL_{p,\infty})_0$, then $A+B\in \sW_{p,\infty}$ and $\Lambda^\pm(A+B)=\Lambda^{\pm}(A)$.
 
\item If $A\in |\sW|_{p,\infty}$ and $B\in (\sL_{p,\infty})_0$, then $A+B\in |\sW|_{p,\infty}$ and $\Lambda(|A+B|)=\Lambda(|A|)$. 
\end{enumerate}
\end{proposition}
 
\begin{remark}
 The last two parts go back to Fan~\cite{Fa:PNAS51} (see also~\cite{GK:AMS69}). 
\end{remark}
 
For $p=1$, it follows from~(\ref{eq:NCInt.integration}) that any Weyl operator in $\sL_{1,\infty}$ is measurable. We actually have a stronger result. 

\begin{proposition}[see~\cite{Po:JNCG23}]\label{prop:NCInt.Weyl-strong}
 If $A\in \sW_{1,\infty}$, then $A$ is strongly measurable, and we have
\begin{equation}\label{eq:Bir-Sol.bint-Lambdapm0}
 \bint A =\Lambda^+(A)-\Lambda^{-}(A). 
\end{equation}
In particular, if $A=A^*\in \sW_{1,\infty}$, then 
\begin{equation}\label{eq:Bir-Sol.bint-Lambdapm}
 \bint A = \lim_{j\rightarrow \infty}j \lambda_j^+(A) -  \lim_{j\rightarrow \infty}j \lambda_j^-(A). 
\end{equation}
In addition, if $A\in |\sW|_{1,\infty}$ (i.e., $|A| \in \sW_{1,\infty}$), then 
\begin{equation*}
  \bint |A| = \lim_{j\rightarrow \infty}j \mu_j(A).  
\end{equation*}
\end{proposition}

The above proposition leads to the following definition. 

\begin{definition}
The operators in $\sW_{1,\infty}$ are called \emph{spectrally measurable}. 
\end{definition}

Proposition~\ref{prop:NCInt.Weyl-strong} thus asserts that any spectrally measurable operator is strongly measurable. Thus spectral measurability is an even stronger form of measurability than strong measurability. Moreover, if $A\in \sL_{1,\infty}$ is spectrally measurable, then from~(\ref{eq:Bir-Sol.bint-Lambdapm0})--(\ref{eq:Bir-Sol.bint-Lambdapm}) we get
\begin{align*}
 \bint A  =  \left[ \lim_{j\rightarrow \infty}j \lambda_j^+(\Re A) -  \lim_{j\rightarrow \infty}j \lambda_j^-(\Re A)\right] + i \left[  \lim_{j\rightarrow \infty}j \lambda_j^+(\Im A) -  \lim_{j\rightarrow \infty}j \lambda_j^-(\Im A)\right]. 
\end{align*}

\subsection{Connes' integration and Birman--Schwinger principle} 
Let $H_0$ be an unbounded selfadjoint operator with non-negative spectrum. We assume that $0$ is not in the essential spectrum of $H_0$, i.e., it is either not in $\Sp(H_0)$ or is an isolated eigenvalue with finite multiplicity. These conditions are automatically satisfied if $H_0$ has compact resolvent. In any case, this allows us to define the (partial) inverse $H_0^{-1}$ in the usual way. 

Let $V$ be a selfadjoint operator which is $H_0$-form compact, i.e., $\dom(V)\supseteq \dom(1+H_0)^{1/2}$ and $(1+H_0)^{-1/2}V(1+H_0)^{-1/2}$ is a compact operator. We then define the operator $H_V:=H_0+V$ as a form-sum (see, e.g., \cite{Si:AMS15}).  This operator is selfadjoint, bounded from below, and has the same essential spectrum as $H_0$. This implies that the negative part of its spectrum consists of finitely many isolated eigenvalues. Moreover, if $H_0$ has a compact resolvent, then $H_V$ has a compact resolvent as well (\emph{loc.\ cit.}). 

For $h>0$, denote by $N^{-}(h^2H_0+V)$ the number of negative eigenvalues of the Schr\"odinger operator $h^2H_0+V$ counted with multiplicity. In the case of Schr\"odinger operators on $\R^n$ this is precisely the number of bound states. We are interested in the behavior of $N^{-}(h^2H_0+V)$ under the semiclassical limit $h\rightarrow 0^+$.  We shall use the following form of the Birman--Schwinger principle. 

\begin{proposition}\label{prop:NCInt.BSP-NCInt}
If $H_0^{-1/2}VH_0^{-1/2}$ is in $\sW_{p,\infty}$, $p>0$, then 
\begin{align}\label{eq:SC.NHV=NKV-SC1}
 \lim_{h\rightarrow 0^+} h^{2p}N^{-}\big(h^2H_0+V\big)  &= \lim_{j\rightarrow \infty} j\lambda_j^{-}\left(H_0^{-\frac12}VH_0^{-\frac12}\right)^p\\ 
 & = \bint \big(H_0^{-\frac12}VH_0^{-\frac12}\big)_{-}^p
 \label{eq:SC.NHV=NKV-SC2} 
\end{align}
\end{proposition}

\begin{remark}
The equality~(\ref{eq:SC.NHV=NKV-SC1}) is a well-known consequence of the Birman--Schwinger principle in its abstract form given in~\cite[Lemma~1.4]{BS:AMST89} (see also \cite[Theorem~10.1]{BS:TMMS72}).
 In fact, the equality holds whenever either of the limits exists. The equality~(\ref{eq:SC.NHV=NKV-SC2}) is a direct consequence of Proposition~\ref{prop:NCInt.Weyl-strong}. 
For the reader's convenience a proof of Proposition~\ref{prop:NCInt.BSP-NCInt} is included in~\cite{Po:JNCG23}.
 \end{remark}
 
\section{Spectral Triples and Condition (W)}\label{chap:Main-results} 
In this section, after recalling the main definitions pertaining to spectral triples, we present Condition~(W) and derive some of its consequences. 

 \subsection{Spectral triples}
In the framework of noncommutative geometry, noncommutative manifolds are represented by spectral triples. We shall use the following definition of spectral triples. 

\begin{definition}
 A \emph{spectral triple} is a triple $(\sA, \sH, D)$, where $\sA$ is a unital $*$-algebra represented by bounded operators on the (separable) Hilbert space $\sH$ and $D$ is a selfadjoint unbounded operator on $\sH$ with compact resolvent such that 
\begin{equation*}
a\left(\dom(D)\right)\subseteq \dom(D) \quad \text{and} \quad [D,a]\in \sL(\sH) \qquad \forall a \in \sA. 
\end{equation*}
We further say that $(\sA,\sH, D)$ is \emph{$p$-summable}, with $p>0$, if the partial inverse $D^{-1}$ is in $\sL_{p,\infty}$. 
\end{definition}

The prototype of a spectral triple is given by the Dirac spectral triple $(C^\infty(M), L^2(M, \shS), \shD)$, where $M$ is a closed Riemannian spin manifold, $L^2(M,\shS)$ is the Hilbert space of $L^2$-spinors on $M$ and $\shD$ is the Dirac operator of $M$ acting on spinors. Another example is the square-root Laplacian spectral triple $(C^\infty(M),L^2(M),\sqrt{\Delta_g})$, where $\Delta_g$ is the (positive) Laplacian of $(M,g)$ acting on functions. Both spectral triples are $p$-summable with $p=\dim M$. 

Some authors also require some regularity properties for the spectral triple. A spectral triple $(\sA, \sH, D)$ is called \emph{Lipschitz regular} if 
\begin{equation*}
 \big[|D|,a\big] \in \sL(\sH) \qquad \forall a \in \sA. 
\end{equation*}
In other words, $\sA$ is contained in the domain of the derivation $\delta(T):=[|D|,T]$. We also say that  
$(\sA, \sH, D)$ is $QC^\infty$-\emph{regular} if 
\begin{equation*}
 \text{$a$ and $[D,a]$ are in} \  \bigcap_{j\geq 1} \dom \delta^j \ \text{for all $a\in \sA$}.  
\end{equation*}
For instance, the Dirac spectral triple $(C^\infty(M), L^2(M, \shS), \shD)$ and the square-root Laplacian spectral triple $(C^\infty(M),L^2(M),\sqrt{\Delta_g})$ are  $QC^\infty$-regular. 

Throughout the rest of this section, we let $(\sA,\sH,D)$ be a $p$-summable spectral triple with $p>0$. We denote by $\overline{\sA}$ the closure of $\sA$ in $\sL(\sH)$, and let $\overline{\sA}_{+}$ be its cone of positive elements (i.e., selfadjoint elements with non-negative spectrum). We also denote by $\overline{\sA}_{++}$ the open cone of its invertible elements (i.e., selfadjoint elements with positive spectrum). The cone ${\sA}_{++}$ is then defined to be $\sA \cap \overline{\sA}_{++}$, i.e., it consists of elements of $\sA$ that are positive and invertible in $\overline{\sA}$.  

\subsection{Condition~(W)}
Let $\Pi_0$ be the orthogonal projection onto $\ker D$. This is a finite rank operator, since $D$ has compact resolvent. Let $D^{-1}$ be the partial inverse of $D$, so that $D^2D^{-2}=1-\Pi_0$ on $\sH$ and $D^{-2}D^2=1-\Pi_0$ on $\dom(D^2)$. If $a\in \overline{\sA}_{++}$, then the action of $a$ on $\sH$ is an isomorphism. Thus, $aD^2a$ is a selfadjoint densely defined operator on $\sH$ with non-negative spectrum whose domain is $a^{-1}(\dom(D^2))$. In particular, $\dom(aD^2a)\supseteq a^{-1}(\ran D^{-2})$, and so on $\sH$ we have 
\begin{equation*}
 \big(aD^2a\big)\big(a^{-1}D^{-2}a^{-1}\big)=aD^2D^{-2}a^{-1}=1-a\Pi_0a^{-1}. 
\end{equation*}
Likewise, on $\dom(aD^2a)=a^{-1}(\dom(D^2))$ we have
\begin{equation}\label{eq:ST.partial-inverse-aD2a2}
 (a^{-1}D^{-2}a^{-1})(aD^2a)= 1-a^{-1}\Pi_0a. 
\end{equation}
This shows that $aD^2a$ is invertible modulo finite rank operators, and so this is a Fredholm operator. Moreover, as $a^{-1}D^{-2}a^{-1}$ is compact, $aD^2a$ has compact resolvent. Thus, it has pure discrete spectrum, and we can list its \emph{positive} eigenvalues as a sequence, 
\begin{equation*}
0< \lambda_0\big(aD^2a\big) \leq  \lambda_1\big(aD^2a\big) \leq  \lambda_2\big(aD^2a\big) \leq \cdots, 
\end{equation*}
where each eigenvalue is repeated according to multiplicity. 

The spectral condition on which the approach relies is the following. 

 \begin{conditionW}
For every $a\in {\sA}_{++}$, we have
\begin{equation}\label{eq:ST.conf-Weyl}
 \lim_{j\rightarrow \infty}j^{-\frac{2}{p}}\lambda_j\big(aD^2a\big)= \tau\big[a^{-p}\big]^{-\frac{2}{p}}, 
\end{equation}
where $\tau:\overline{\sA}\rightarrow \C$ is a given non-zero positive linear map. 
\end{conditionW}

\begin{remark}
 For $a=1$ the Weyl laws~(\ref{eq:ST.conf-Weyl}) give 
\begin{equation*}
 \lambda_j(|D|)=\sqrt{\lambda_j(D^2)} \sim j^{1/p}\tau[1]^{-1/p} \qquad \text{as}\ j\rightarrow \infty. 
\end{equation*}
This automatically implies that $|D|^{-1}\in \sL_{p,\infty}$, i.e., the spectral triple $(\sA, \sH, D)$ is $p$-summable. 
\end{remark}

\begin{remark}
 The assumption on the linearity of $\tau$ is not necessary. In the main examples the linearity is immediate. Moreover, the map $\tau$ will be later identified with a linear map (see Remark~\ref{rmk:ST-identification-tau}). Therefore, there is no harm in assuming it beforehand. 
 \end{remark}

\begin{remark}
The fact that $\tau$ is a non-zero positive linear map on $\overline{\sA}$ ensures that $\tau$ is continuous and $\tau(a)>0$ for all $a\in\overline{\sA}_{++}$. 
\end{remark}

\begin{remark}\label{rmk:ST-bdd-perturbation-Condition}
 If $A=A^*\in \sL(\sH)$ and we set $D_A=D+A$, then $(\sA,\sH,D_A)$ is a $p$-summable spectral triple. It can also be shown that if $(\sA,\sH,D)$ satisfies Condition~\textup{(W)}, then $(\sA,\sH,D_A)$ still satisfies this condition with the same map $\tau$ (see Lemma~\ref{lem:ST-bdd-perturbation} below). Thus, Condition~\textup{(W)} is invariant under bounded perturbations $D\rightarrow D+A$. This includes the inner fluctuations of the metric in the sense of Connes (see, e.g., \cite{CM:AMS08}).
 \end{remark}

It is useful to relate the Weyl laws~(\ref{eq:ST.conf-Weyl}) to Weyl laws for the compact operators $aD^{-2}a$, $a\in \overline{\sA}$. 

\begin{lemma}\label{lem:ST.ComparisonD2}
 For every $a\in \overline{\sA}_{++}$, we have 
\begin{equation*}
 \lim_{j\rightarrow \infty} j^{\frac{2}{p}}\lambda_{j}\left(a^{-1}D^{-2}a^{-1}\right) =  \lim_{j\rightarrow \infty} j^{\frac{2}{p}}\lambda_{j}\left(aD^{2}a\right)^{-1}, 
\end{equation*}
provided any of the above limits exists in $[0,\infty]$. That is, if one limit exists, then the other limit exists and the limits agree.  
 \end{lemma}

\begin{proof}
 Let $a\in \overline{\sA}_{++}$, and denote by $\Pi_1$ the orthogonal projection onto $\ker (aD^2a)=a^{-1}(\ker D)$. By using~(\ref{eq:ST.partial-inverse-aD2a2}) we get
 \begin{align*}
 a^{-1}D^{-2}a^{-1}&=a^{-1}D^{-2}a^{-1}(1-\Pi_1)+ (a^{-1}D^{-2}a^{-1})\Pi_1 \\
 &= (a^{-1}D^{-2}a^{-1})(aD^2a)(aD^2a)^{-1} + (a^{-1}D^{-2}a^{-1})\Pi_1 \\
 &= (1-a^{-1}\Pi_0 a)(aD^2a)^{-1} + (a^{-1}D^{-2}a^{-1})\Pi_1 \\
 &= (aD^2a)^{-1} - (a^{-1}\Pi_0 a)(aD^2a)^{-1} +(a^{-1}D^{-2}a^{-1})\Pi_1. 
\end{align*}
 It then follows that the operator $R:=(aD^2a)^{-1}-a^{-1}D^{-2}a^{-1}$ has finite rank. 
 
 Set $N=\op{rk}(R)$. The min-max principle~(\ref{eq:min-max}) implies that $\mu_j(R)=0$ for $j\geq N$. Thus, if $j\geq N$, then by using~(\ref{eq:Quantized.properties-mun2}) we get
\begin{equation*}
\lambda_j\big(\left(aD^{2}a\right)^{-1}\big) = \mu_j \big(\left(aD^{2}a\right)^{-1}\big) \leq  \mu_{j-N}\left(a^{-1}D^{-2}a^{-1} \right)+\mu_N(R) =\lambda_{j-N}\left(a^{-1}D^{-2}a^{-1} \right) .
\end{equation*}
 Similarly, we have 
 \begin{equation*}
\lambda_j\big(\left(aD^{2}a\right)^{-1}\big) = \mu_j\big(\left(aD^{2}a\right)^{-1}\big)  \geq  \mu_{j+N}\left(a^{-1}D^{-2}a^{-1} \right)-\mu_N(R) =\lambda_{j+N}\left(a^{-1}D^{-2}a^{-1} \right) .
\end{equation*}
With our index conventions $\lambda_j[(aD^2a)^{-1}]=\lambda_{j}(aD^2a)^{-1}$. Thus, for all $j\geq N$, we have
\begin{equation*}
 \lambda_{j+N}\left(a^{-1}D^{-2}a^{-1} \right) \leq \lambda_{j}\left(aD^{2}a\right)^{-1} \leq  \lambda_{j-N}\left(a^{-1}D^{-2}a^{-1} \right). 
\end{equation*}
It then follows that 
\begin{equation*}
 \lim_{j\rightarrow \infty} j^{\frac{2}{p}}\lambda_{j}\left(aD^{2}a\right)^{-1} =  \lim_{j\rightarrow \infty} j^{\frac{2}{p}}\lambda_{j}\left(a^{-1}D^{-2}a^{-1} \right), 
\end{equation*}
provided any of these limits exists. This gives the result. 
\end{proof}

In what follows,  we set
\begin{equation*}
 \sA_{++}^q:=\left\{a^{q}; \ a\in {\sA}_{++}\right\}, \qquad q\neq 0. 
\end{equation*}
If $\sA$ is closed under holomorphic functional calculus, then $\sA_{++}^{q}=\sA_{++}$, but in general we do not have an equality for $q\neq 1$. 
  
We will need the following lemma. 

\begin{lemma}[compare~{\cite[Remark~3.5]{MSZ:LMP22}}]\label{lem:STC.density}
If $a\in \overline{\sA}_+$, then, given any $q\neq 0$, we can always find a sequence $(a_\ell)_{\ell\geq 0}\subseteq \sA_{++}^q$ such that $(a_\ell)^r \rightarrow a^r$ in $\overline{\sA}$ for all $r>0$. In particular,  $\sA_{++}^q$ is a dense subset of $\overline{\sA}_+$. 
 \end{lemma}
 \begin{proof}
 Note that if $a_\ell \rightarrow a$ in $\overline{\sA}_+$, then $f(a_\ell)\rightarrow f(a)$ for every continuous function $f:[0,\infty)\rightarrow \C$. This is true for polynomials. The result then can be extended to all continuous functions on $[0,\infty)$ by using Stone--Weierstrass theorem. In particular, the result holds for all power functions $t\rightarrow t^q$, $q>0$.  Therefore, in order to prove the lemma it is enough to show that 
 $\sA_{++}^q$ is a dense subset of $\overline{\sA}_+$ for all $q\neq 0$. 
 
We first prove the result for $q=1$.  Let $a\in \overline{\sA}_+$, and let $(b_\ell)_{\ell \geq 0}\subseteq \sA$ be such that $b_\ell \rightarrow \sqrt{a}$. Set $a_\ell=b_\ell^*b_\ell+(\ell+1)^{-1}$. Then $a_\ell\in \sA_{++}$ and $a_\ell \rightarrow (\sqrt{a})^*\sqrt{a}=a$ as $\ell \rightarrow \infty$. This shows that $\sA_{++}$ is dense in $\overline{\sA}_+$. 
 
 Let $q\neq 0$ and $b\in \overline{\sA}_{++}$. The density of $\sA_{++}$ in $\overline{\sA}_{+}$ ensures there is a sequence $(b_\ell)_{\ell \geq 0}\subseteq \sA_{++}$ such that $b_\ell\rightarrow b^{1/q}$. As $b^{1/q}$ is invertible, there is $\delta>0$ such that the spectra of the $b_\ell$ are eventually contained in $[\delta,\delta^{-1}]$. Arguing as above with continuous functions on $[\delta,\delta^{-1}]$ then shows that $b_\ell^q\rightarrow (b^{1/q})^q=b$. This shows that $\sA_{++}^q$ is dense in $\overline{\sA}_{++}$. As $\overline{\sA}_{++}$ is dense in $\overline{\sA}_{+}$, we deduce that $\sA_{++}^{q}$ is dense in $\overline{\sA}_+$, completing the proof. 
\end{proof}

We have the following characterizations of Condition~\textup{(W)}. 
 
 \begin{lemma}\label{lem:ST.equiv-conformal-Weyl}
 The following are equivalent:
\begin{enumerate}
\item[(i)]  The Weyl laws~(\ref{eq:ST.conf-Weyl}) hold for all $a\in {\sA}_{++}$, i.e., Condition~(W) is satisfied. 

\item[(ii)]  The Weyl laws~(\ref{eq:ST.conf-Weyl}) hold for all $a\in \overline{\sA}_{++}$. 

\item[(iii)] For all $a\in \overline{\sA}_+$, we have
\begin{equation}\label{eq:ST.conforrmal-Weyl-neg}
 \lim_{j\rightarrow \infty}j^{\frac{2}{p}}\lambda_j\big(aD^{-2}a\big)= \tau\big[a^{p}\big]^{\frac{2}{p}}. 
\end{equation}
\end{enumerate}
\end{lemma}
\begin{proof}
It is immediate that (ii)$\Rightarrow$(i). Moreover, if (iii) holds, then by Lemma~\ref{lem:ST.ComparisonD2}, for all $a\in \overline{\sA}_{++}$, we have
 \begin{equation*}
 \lim_{j\rightarrow \infty} j^{\frac{2}{p}}\lambda_{j}\left(aD^{2}a\right)^{-1} =  
 \lim_{j\rightarrow \infty}j^{\frac{2}{p}}\lambda_j\big(a^{-1}D^{-2}a^{-1}\big) = \tau\big[a^{-p}\big]^{\frac{2}{p}}. 
\end{equation*}
This shows that (iii)$\Rightarrow$(ii).

It remains to show that (i)$\Rightarrow$(iii). Suppose that (i) holds. Lemma~\ref{lem:ST.ComparisonD2} then ensures that, for all $a\in \sA_{++}$, we have
\begin{equation*}
 \lim_{j\rightarrow \infty}j^{\frac{2}{p}}\lambda_j\big(a^{-1}D^{-2}a^{-1}\big)= 
 \lim_{j\rightarrow \infty} j^{\frac{2}{p}}\lambda_{j}\left(aD^{2}a\right)^{-1} = \tau\big[a^{-p}\big]^{\frac{2}{p}}. 
\end{equation*}
This shows that the Weyl laws~\eqref{eq:ST.conforrmal-Weyl-neg} hold for all $a\in \sA_{++}^{-1}$. 

In general, if $a\in \overline{\sA}_+$, then by Lemma~\ref{lem:STC.density} there is a sequence $(a_\ell)\subseteq {\sA}_{++}^{-1}$ such that $a_\ell \rightarrow a$   
and $(a_\ell)^p\rightarrow a^p$ in $\overline{\sA}$. In particular, $a_\ell D^{-2}a_\ell\rightarrow aD^{-2}a$ in $\sL_{p/2,\infty}$. Combining this with Proposition~\ref{prop:NCInt.Bir-Sol} and using the continuity of $\tau$ we then get
\begin{equation*}
 \lim_{j\rightarrow \infty} j^{\frac{2}{p}}\lambda_j\big(aD^{-2}a\big)= \lim_{\ell \rightarrow \infty}  \lim_{j\rightarrow \infty} j^{\frac{2}{p}}\lambda_j\big(a_\ell D^{-2}a_\ell\big)
 =  \lim_{\ell \rightarrow \infty}\tau\big[a_\ell^{p}\big]^{\frac{2}{p}}= \tau\big[a^{p}\big]^{\frac{2}{p}}. 
\end{equation*}
That is, $aD^{-2}a$ satisfies the Weyl laws~(\ref{eq:ST.conforrmal-Weyl-neg}), and so (iii) is satisfied. This proves that (i)$\Rightarrow$(iii). The proof is complete. 
\end{proof}

\begin{remark}
The spectral asymptotics~(\ref{eq:ST.conforrmal-Weyl-neg}) identify on $\overline{\sA}_+$ the map $\tau$ with the functional,
\begin{equation*}
 \overline{\sA}_+\ni a \longrightarrow \lim_{j\rightarrow \infty} j\lambda_j\big(a^{\frac1{p}}D^{-2}a^{\frac1{p}}\big)^{\frac{p}{2}}=\Lambda\big(a^{\frac1{p}}D^{-2}a^{\frac1{p}}\big)^{\frac{p}{2}}. 
\end{equation*}
This functional is continuous on $\overline{\sA}_+$ thanks to Proposition~\ref{prop:NCInt.Bir-Sol}. Moreover, the density of ${\sA}_{++}^{-p}$ in $\overline{\sA}_+$ implies that the above functional is uniquely determined by its restriction to ${\sA}_{++}^{-p}$. It follows that the Weyl laws~(\ref{eq:ST.conf-Weyl}) uniquely define a continuous functional on $\overline{\sA}_+$. 
\end{remark}

We are now in a position to prove the following perturbation result (which was mentioned in Remark~\ref{rmk:ST-bdd-perturbation-Condition}). 

\begin{lemma}\label{lem:ST-bdd-perturbation}
 Let $A=A^*\in \sL(\sH)$, and set $D_A=D+A$. If Condition~\textup{(W)} is satisfied by the spectral triple $(\sA,\sH,D)$, then it is also satisfied by the spectral triple 
 $(\sA,\sH,D_A)$. 
\end{lemma}
\begin{proof}
 As $A$ is bounded, $\dom(D_A)=\dom(D)$. Modulo finite-rank operators, we then have the resolvent formula, 
 \begin{equation*}
 D_A^{-1}=D^{-1}-D^{-1}(D_A-D)D_A^{-1}=D^{-1}-D^{-1}AD_A^{-1}.
\end{equation*}
As $D^{-1}\in \sL_{p,\infty}$, it follows that $D_A^{-1}\in \sL_{p,\infty}$, and hence
\begin{equation*}
 D_A^{-1}-D^{-1}=-D^{-1}AD_A^{-1}\in \sL_{p/2,\infty}. 
\end{equation*}
Thus,
\begin{equation*}
 D^{-2}_A-D^{-2}=D^{-1}_A\left(D_A^{-1}-D^{-1}\right) + \left(D_A^{-1}-D^{-1}\right)D^{-1}\in \sL_{p/3,\infty}. 
\end{equation*}
It follows that, for all $a\in \overline{\sA}_+$, we have
\begin{equation*}
 aD^{-2}_Aa=aD^{-2}a \qquad \bmod \big(\sL_{\frac{p}{2},\infty}\big)_0.
\end{equation*}
Thus, if Condition~\textup{(W)} holds, then by using Proposition~\ref{prop:NCInt.Bir-Sol} and Lemma~\ref{lem:ST.equiv-conformal-Weyl} we get
\begin{equation*}
 \lim_{j\rightarrow \infty}j^{\frac{2}{p}}\lambda_j\big(aD_A^{-2}a\big)  =  \lim_{j\rightarrow \infty}j^{\frac{2}{p}}\lambda_j\big(aD^{-2}a\big) = \tau\big[a^{p}\big]^{\frac{2}{p}}. 
\end{equation*}
Using Lemma~\ref{lem:ST.equiv-conformal-Weyl} once more then shows that 
\begin{equation*}
 \lim_{j\rightarrow \infty}j^{-\frac{2}{p}}\lambda_j\big(aD_A^{2}a\big)  =   \tau\big[a^{-p}\big]^{-\frac{2}{p}} \qquad \forall a\in \sA_{++}. 
\end{equation*}
That is, the spectral triple $(\sA,\sH,D_A)$ satisfies Condition~\textup{(W)}. The proof is complete. 
\end{proof}

\section{Commutator Estimates}\label{sec:Com} 
In this section, we prove commutator estimates for spectral triples and establish several consequences of these estimates. They are important ingredients in the proof of the spectral asymptotics provided by Proposition~\ref{prop:Intro.Weyl-BirS} in the next section. We stress that these results hold for any (unital) spectral triple regardless of the degree of summability  and without any additional regularity properties. 

Throughout this section we let $(\sA, \sH, D)$ be a $p$-summable spectral triple, where $p$ is any positive number.  Given any $q> 0$ we define $|D|^{-q}$ by means of the Borel functional calculus for $D$ with the convention that $|D|^{-q}=0$ on $\ker D$. 
Thus, if $(\xi_\ell)_{\ell\geq 0}$ is any orthonormal eigenbasis of $\sH$ with $D\xi_\ell=\lambda_\ell\xi_\ell$ for all $\ell\geq 0$, then
\begin{equation*}
 |D|^{-q}\xi_\ell = \left\{ 
\begin{array}{cl}
 |\lambda_\ell|^{-q} \xi_\ell & \text{if}\ \lambda_\ell\neq 0,\\
 0 & \text{if}\ \lambda_\ell=0. 
\end{array}\right. 
\end{equation*}

The main result of this section is the commutator estimate of Proposition~\ref{prop:Intro.Com-Dq}, which we now prove. 

\begin{proof}[Proof of Proposition~\ref{prop:Intro.Com-Dq}]
 In what follows we denote by $C_{pq}$ and $C_{D,q}$ positive constants depending only on the data $(p,q)$ or $(D,q)$ which may vary from line to line. 
 
 Suppose that $q\in (0,1]$, and set $r=(q+1)^{-1}p$. Denote by $F=D|D|^{-1}$ the sign of $D$ and by $\Pi_0$ the orthogonal projection onto $\ker D$. By the Factorization Lemma (Lemma~\ref{lem:Intro.factorization}) mentioned in the introduction, there is a continuous linear operator $\Phi_D^{(q)}:\sL(\sH)\rightarrow \sL_{r,\infty}$, such that 
 for all $a\in \sA$, we have
  \begin{align}
  \left[ |D|^{-q},a\right]  =&\Phi_D^{(q)}\left([D,a]\right) 
  + |D|^{-(q+1)}F[D,a]\Pi_0 +\Pi_0[D,a]F|D|^{-(q+1)}. 
  \label{eq:Com.factorization2}
\end{align}
In particular, the continuity of $\Phi_D^{(q)}$ ensures that
 \begin{equation*}
 \big\|\Phi_D^{(q)}\left([D,a]\right)\big\|_{r,\infty}  \leq C_{D,q} \left\| \left[ D,a\right]\right\|. 
\end{equation*}

As $|D|^{-(q+1)}\in \sL_{r,\infty}$, we also get
\begin{equation*}
  \left\| |D|^{-(q+1)}F[D,a]\Pi_0\right\|_{r,\infty} \leq \big\| |D|^{-(q+1)}\big\|_{r,\infty}  \left\|F[D,a]\Pi_0\right\| \leq C_{D,q}\left\| \left[ D,a\right]\right\|. 
\end{equation*}
We have a similar estimate for $\Pi_0[D,a]F|D|^{-(q+1)}$. Combining all those estimates with~(\ref{eq:Com.factorization2}) yields the estimates~(\ref{eq:Com.Com-Dq-estimate}) for all $q\in (0,1]$. In particular, for $q=1$ we get 
 \begin{equation}\label{eq:Com.Com-Dq-estimate1}
 \big\| \big[ |D|^{-1},a\big]\big\|_{\frac{p}{2},\infty}\leq C_{D}  \left\| \left[ D,a\right]\right\| \qquad \forall a \in \sA. 
\end{equation}
 
 Suppose now that the result is true for some $q>0$. Set $r_1=(q+1)^{-1}p$ and $r_2=(q+2)^{-1}p$. By Leibniz's rule, we have 
\begin{equation}\label{eq:Com.Com-Dq1-Leibniz}
  \big[ |D|^{-(q+1)},a\big]= |D|^{-1}  \big[ |D|^{-q},a\big] +   \big[ |D|^{-1},a\big] |D|^{-q}.  
\end{equation}

Here $|D|^{-1}\in \sL_{p,\infty}$, and by assumption $[ |D|^{-q},a]$ is in $\sL_{r_1,\infty}$ and satisfies the estimate~(\ref{eq:Com.Com-Dq-estimate}).  As $p^{-1}+r_1^{-1}=r_2^{-1}$, by using H\"older's inequality (Proposition~\ref{prop:Hoelder}) we see that $|D|^{-1}[ |D|^{-q},a]$ is in $\sL_{r_2,\infty}$, and we have
\begin{equation}\label{eq:Com.Com-Dq-estimate-q1}
  \big\| |D|^{-1}  \left[ |D|^{-q},a\right]\big\|_{r_2,\infty} \leq C_{pq} \big\|  |D|^{-1}\big\|_{p,\infty}  \left\| \big[ |D|^{-q},a\big]\right\|_{r_1,\infty} 
  \leq C_{D,q} \left\| \left[ D,a\right]\right\|.
\end{equation}

In addition, $|D|^{-q}\in \sL_{q^{-1}p,\infty}$ and we know that $[|D|^{-1},a]$ is in $\sL_{\frac{p}{2},\infty}$ and satisfies~(\ref{eq:Com.Com-Dq-estimate1}). As $2p^{-1}+qp^{-1}=r_2^{-1}$, by using H\"older's inequality once more, we deduce that $[|D|^{-1},a]|D|^{-q}$ is in $\sL_{r_2,\infty}$ and we have 
\begin{equation}\label{eq:Com.Com-Dq-estimate-q2}
  \big\| \big[ |D|^{-1},a\big] |D|^{-q}\big\|_{r_2,\infty} \leq C_{pq} \left\| |D|^{-q}\right\|_{q^{-1}p,\infty}  \big\| \big[ |D|^{-1},a\big]\big\|_{\frac{p}{2},\infty} 
  \leq C_{D,q} \left\| \left[ D,a\right]\right\|.
\end{equation}

Combining~(\ref{eq:Com.Com-Dq1-Leibniz}) with~(\ref{eq:Com.Com-Dq-estimate-q1})--(\ref{eq:Com.Com-Dq-estimate-q2}) shows that $[|D|^{-(q+1)},a]$ is in $\sL_{r_2,\infty}$ and satisfies the estimate~(\ref{eq:Com.Com-Dq-estimate}). This shows that if the result holds for $q>0$, then it holds for $q+1$. As the result holds for all $q\in (0,1]$, an induction then shows that it holds for all $q$ in intervals $(0,k]$, $k=1,2,\ldots $, and hence it holds for all $q>0$. The proof is complete. 
\end{proof}

We have the following consequences of Proposition~\ref{prop:Intro.Com-Dq}. 

\begin{corollary}\label{cor:ST.Com-Dq-sL0}
Let $q>0$. For all $a\in \overline{\sA}$, we have
 \begin{equation*}
  \left[ |D|^{-q},a\right] \in \big(\sL_{q^{-1}p,\infty}\big)_0. 
\end{equation*}
\end{corollary}
\begin{proof}
 Thanks to Proposition~\ref{prop:Hoelder}, for all $\beta,\gamma\geq 0$ with $\beta+\gamma>0$, the map $a\rightarrow |D|^{-\beta}a|D|^{-\gamma}$ is a continuous linear map from $\overline{\sA}$ to $\sL_{(\beta +\gamma)^{-1}p,\infty}$. Thus, $a\rightarrow [|D|^{-q},a]$ is at least a continuous linear map from  $\overline{\sA}$ to $\sL_{q^{-1}p,\infty}$. 
 By Proposition~\ref{prop:Intro.Com-Dq}, if $a\in \sA$, then $[|D|^{-q},a]$ is contained in $\sL_{(q+1)^{-1}p,\infty}\subseteq (\sL_{q^{-1}p,\infty})_0$. As $\sA$ is dense in $\overline{\sA}$, it follows that $[|D|^{-q},a]\in(\sL_{q^{-1}p,\infty})_0$ for all $a\in \overline{\sA}$. The result is proved.    
\end{proof}

\begin{corollary}\label{cor:ST.Com-Dq-sL1}
 Let $q>0$. The following hold.
\begin{enumerate}
 \item For all $a\in \overline{\sA}$, we have
 \begin{equation}\label{eq:ST.Com-Dq-sL02}
 a|D|^{-q} = |D|^{-q}a = |D|^{-\frac{q}{2}}a|D|^{-\frac{q}{2}} \quad \bmod \big(\sL_{q^{-1}p,\infty}\big)_0.
\end{equation}

\item For all $a\in \overline{\sA}_+$, we have
\begin{equation}\label{eq:ST.Com-Dq-sL03}
 a^{\frac{1}{2}} |D|^{-q} a^{\frac12} = |D|^{-\frac{q}{2}}a|D|^{-\frac{q}{2}} \quad \bmod \big(\sL_{q^{-1}p,\infty}\big)_0.
\end{equation}
\end{enumerate}
\end{corollary}
\begin{proof}
Let $a\in \overline{\sA}$. We have
\begin{equation*}
 |D|^{-q}a = a|D|^{-q} + \big[|D|^{-q},a\big] = |D|^{-\frac{q}{2}}a|D|^{-\frac{q}{2}} + |D|^{-\frac{q}{2}}\big[|D|^{-\frac{q}{2}},a\big]. 
\end{equation*}
 It follows from Corollary~\ref{cor:ST.Com-Dq-sL0} that $[|D|^{-q},a]\in (\sL_{q^{-1}p,\infty})_0$ and $[|D|^{-q/2},a]\in  (\sL_{2q^{-1}p,\infty})_0$. As $|D|^{-q/2}\in \sL_{2q^{-1}p,\infty}$, the 2nd part of Proposition~\ref{prop:Hoelder} ensures that $|D|^{-q/2}[|D|^{-q/2},a]\in  (\sL_{q^{-1}p,\infty})_0$. This gives~(\ref{eq:ST.Com-Dq-sL02}). 
 
 Suppose now that $a\in \overline{\sA}_+$. Then 
\begin{equation*}
 a^{\frac{1}{2}} |D|^{-q} a^{\frac12} = |D|^{-q}a - \big[|D|^{-q},a^{\frac{1}{2}}\big]a^{\frac{1}{2}}. 
\end{equation*}
By Corollary~\ref{cor:ST.Com-Dq-sL0} the commutator $[|D|^{-q},a^{1/2}]$ is contained in $(\sL_{q^{-1}p,\infty})_0$. Combining this with~(\ref{eq:ST.Com-Dq-sL02}) gives~(\ref{eq:ST.Com-Dq-sL03}). The proof is complete. 
\end{proof}

\begin{remark}\label{rmk:ST.Com-Dq-sL}
 If $a\in \sA$, then by Proposition~\ref{prop:Intro.Com-Dq} the commutators $[|D|^{-q},a]$ and $[|D|^{-q/2},a]$ are contained in $\sL_{(q+1)^{-1}p,\infty}$ and $\sL_{(q/2+1)^{-1}p,\infty}$, respectively. Combining this with Proposition~\ref{prop:Hoelder} and arguing as in the proof of Corollary~\ref{cor:ST.Com-Dq-sL1} shows that in this case the equalities~(\ref{eq:ST.Com-Dq-sL02}) hold modulo $\sL_{(q+1)^{-1}p,\infty}$. Likewise, if $a\in \sA_+$ and $a^{1/2}\in \sA$, then~(\ref{eq:ST.Com-Dq-sL03}) holds modulo $\sL_{(q+1)^{-1}p,\infty}$. 
\end{remark}

\begin{corollary}[see also~\cite{MSZ:LMP22, SZ:FAA23}]
 For any $q>0$ and $a\in \overline{\sA}$, we have
 \begin{align}\label{eq:ST.abs}
  \big||D|^{-\frac{q}{2}}a|D|^{-\frac{q}{2}}\big| &- |D|^{-\frac{q}{2}}|a| |D|^{-\frac{q}{2}}\in  \big(\sL_{q^{-1}p,\infty}\big)_0, \\
  \big(|D|^{-\frac{q}{2}}a|D|^{-\frac{q}{2}}\big)_{\pm} &-|D|^{-\frac{q}{2}}a_\pm|D|^{-\frac{q}{2}}\in  \big(\sL_{q^{-1}p,\infty}\big)_0 
  \quad (\text{if}\ a^*=a). 
 \label{eq:ST.pos-neg-parts}
\end{align}
\end{corollary}
\begin{proof}
Let $q>0$ and $a\in\overline{\sA}$. We have
\begin{align}
 \big| |D|^{-\frac{q}{2}}a|D|^{-\frac{q}{2}}\big|^2- \big( |D|^{-\frac{q}{2}}|a||D|^{-\frac{q}{2}}\big)^2 & =  
 |D|^{-\frac{q}{2}} \big(a^*|D|^{-q}a-|a| |D|^{-q}|a|\big) |D|^{-\frac{q}{2}} \nonumber \\
 &= |D|^{-\frac{q}{2}} a^*\big[|D|^{-q},a\big] |D|^{-\frac{q}{2}}- |D|^{-\frac{q}{2}}|a| \big[|D|^{-q},|a|\big] |D|^{-\frac{q}{2}}. 
 \label{eq:ST.square-difference}
\end{align}
 By Corollary~\ref{cor:ST.Com-Dq-sL0} the commutators $[|D|^{-q},a]$ and $[|D|^{-q},|a|]$ are contained in $(\sL_{q^{-1}p,\infty})_0$. Proposition~\ref{prop:Hoelder} then ensures that both summands 
 $ |D|^{-\frac{q}{2}} a^*[|D|^{-q},a] |D|^{-\frac{q}{2}}$ and $|D|^{-\frac{q}{2}}|a| [|D|^{-q},|a|] |D|^{-\frac{q}{2}}$ are in $(\sL_{(2q)^{-1}p,\infty})_0$, and so the difference 
 $| |D|^{-q/2}a|D|^{-q/2}|^2- ( |D|^{-q/2}|a||D|^{-q/2})^2$ is in $(\sL_{(2q)^{-1}p,\infty})_0$. Applying Proposition~\ref{prop:BKS} then gives~(\ref{eq:ST.abs}). 
 
Suppose that $a^*=a$. We have
\begin{align*}
 \big( |D|^{-\frac{q}{2}}a|D|^{-\frac{q}{2}}\big)_{\pm}&=\frac{1}{2}\left(\big| |D|^{-\frac{q}{2}}a|D|^{-\frac{q}{2}}\big| \pm |D|^{-\frac{q}{2}}a|D|^{-\frac{q}{2}}\right),\\ 
 |D|^{-\frac{q}{2}}a_\pm|D|^{-\frac{q}{2}}&= \frac{1}{2}|D|^{-\frac{q}{2}}\left(|a|\pm a\right)|D|^{-\frac{q}{2}}. 
\end{align*}
Thus, 
\begin{equation*}
  \big( |D|^{-\frac{q}{2}}a|D|^{-\frac{q}{2}}\big)_{\pm}-  |D|^{-\frac{q}{2}}a_\pm|D|^{-\frac{q}{2}}= \frac{1}{2}\Big( 
  \big| |D|^{-\frac{q}{2}}a|D|^{-\frac{q}{2}}\big|- |D|^{-\frac{q}{2}}|a||D|^{-\frac{q}{2}}\Big).
\end{equation*}
Combining this with~(\ref{eq:ST.abs}) then gives~(\ref{eq:ST.pos-neg-parts}).  The proof is complete.
\end{proof}

\begin{remark}
 If $a\in \sA$ is such that $|a|\in\sA$ (e.g., $a$ is invertible and $\sA$ is closed under holomorphic functional calculus), we can actually get an even stronger result. In this case, by using Proposition~\ref{prop:Intro.Com-Dq} we see that the right-hand side of~(\ref{eq:ST.square-difference}) is in $\sL_{(2q+1)^{-1}p,\infty}$. 
Proposition~\ref{prop:BKS} then ensures that the relation~(\ref{eq:ST.abs}) holds in the smaller ideal $\sL_{(q+1/2)^{-1}p,\infty}$ rather than in  $(\sL_{pq^{-1},\infty})_0$. Likewise, the relation~(\ref{eq:ST.pos-neg-parts}) holds in  $\sL_{(q+1/2)^{-1}p,\infty}$. 
\end{remark}

\begin{remark}
 A version of~(\ref{eq:ST.pos-neg-parts}) is established in~\cite[Lemma~1.15]{MSZ:LMP22} (see also \cite[Lemma~5.10]{SZ:FAA23}). The approach in~\cite{MSZ:LMP22, SZ:FAA23} relies on a deep result of~\cite{HSZ:PLMS22} which extends the BKS inequality to $q$-convexifications of fully symmetric normed ideals. As the above proof shows, the original version of the BKS inequality as stated by Proposition~\ref{prop:BKS} is enough for our purpose. 
\end{remark}

\section{Spectral Asymptotics}\label{sec:SpAsymp} 
The aim of this section is to prove the spectral asymptotics provided by Proposition~\ref{prop:Intro.Weyl-BirS}.  

Throughout the rest of this paper we let $(\sA,\sH,D)$ be a spectral triple satisfying Condition~(W). We keep on using the notation introduced in the previous sections. 

\begin{proof}[Proof of Proposition~\ref{prop:Intro.Weyl-BirS}] 
We need to show that, given any $q>0$, for all $a\in \overline{\sA}$, we have
\begin{gather}
 \label{eq:ST.Weyl-|DpaDp|-tau}
 \begin{aligned}
 \lim_{j\rightarrow \infty}j^{\frac{q}{p}} \mu_j\big(|D|^{-\frac{q}{2}}a|D|^{-\frac{q}{2}}\big) & = \lim_{j\rightarrow \infty}j^{\frac{q}{p}} \mu_j\big(a|D|^{-q}\big) = \tau\left[|a|^{\frac{p}{q}}\right]^{\frac{q}{p}},
 \end{aligned}\\
 \label{eq:ST.Weyl-DpaDp-tau}
 \lim_{j\rightarrow \infty} j^{\frac{q}{p}} \lambda_j^\pm\big(|D|^{-\frac{q}{2}}a|D|^{-\frac{q}{2}}\big) =\tau\left[\big(a_\pm\big)^{\frac{p}{q}}\right]^{\frac{q}{p}} \quad (\text{if}\ a^*=a). 
\end{gather}

Let $a\in \overline{\sA}$. It follows from~(\ref{eq:ST.abs}) and Proposition~\ref{prop:NCInt.Bir-Sol}  that 
\begin{equation*}
  \lim_{j\rightarrow \infty}j^{\frac{q}{p}} \mu_j\big(|D|^{-\frac{q}{2}}a|D|^{-\frac{q}{2}}\big)= 
   \lim_{j\rightarrow \infty}j^{\frac{q}{p}} \lambda_j\big(   |D|^{-\frac{q}{2}}|a| |D|^{-\frac{q}{2}}\big),
\end{equation*}
provided any of these limits exists. Moreover, by Corollary~\ref{cor:ST.Com-Dq-sL1} the operator $a|D|^{-q}$ agrees with $|D|^{-q/2}a|D|^{-q/2}$ modulo $(\sL_{q^{-1}p,\infty})_0$. Therefore, by Proposition~\ref{prop:NCInt.Bir-Sol} the singular values of $a|D|^{-q}$ have the same asymptotics as those of $|D|^{-q/2}a|D|^{-q/2}$, provided any of these asymptotics exists. Likewise, if $a^*=a$, then by using~(\ref{eq:ST.pos-neg-parts}) and Proposition~\ref{prop:NCInt.Bir-Sol}  we see that
\begin{equation*}
  \lim_{j\rightarrow \infty}j^{\frac{q}{p}} \lambda_j^\pm\big(|D|^{-\frac{q}{2}}a|D|^{-\frac{q}{2}}\big) =   
  \lim_{j\rightarrow \infty}j^{\frac{q}{p}} \mu_j\left(\big(|D|^{-\frac{q}{2}}a|D|^{-\frac{q}{2}}\big)_\pm\right) = 
  \lim_{j\rightarrow \infty}j^{\frac{q}{p}} \lambda_j\big( |D|^{-\frac{q}{2}}a_\pm|D|^{-\frac{q}{2}} \big).
\end{equation*}
Therefore, in order to establish the spectral asymptotics~(\ref{eq:ST.Weyl-|DpaDp|-tau})--(\ref{eq:ST.Weyl-DpaDp-tau}) it is enough to show that
\begin{equation}\label{eq:ST.spectral-asymptotics-pos}
 \lim_{j\rightarrow \infty}  j^{\frac{q}{p}} \lambda_j\big(|D|^{-\frac{q}{2}}a|D|^{-\frac{q}{2}} \big)=  \tau\left[a^{\frac{p}{q}}\right]^{\frac{q}{p}} \qquad \forall a\in \overline{\sA}_+. 
\end{equation}

Let $a\in \overline{\sA}_+$. If we use~(\ref{eq:ST.Com-Dq-sL03}) and Proposition~\ref{prop:NCInt.Bir-Sol}, then we get 
\begin{equation*}
  \lim_{j\rightarrow \infty}  j^{\frac{q}{p}} \lambda_j\big(|D|^{-\frac{q}{2}}a|D|^{-\frac{q}{2}} \big) =  
  \lim_{j\rightarrow \infty}  j^{\frac{q}{p}} \lambda_j\big(a^{\frac{1}{2}}|D|^{-q}a^{\frac{1}{2}}\big), 
\end{equation*}
provided any of these limits exists. It follows that~(\ref{eq:ST.spectral-asymptotics-pos}) is equivalent to 
\begin{equation}\label{eq:ST.spectral-asymptotics-pos2}
 \lim_{j\rightarrow \infty}  j^{\frac{q}{p}} \lambda_j\big(a^{\frac{1}{2}}|D|^{-q}a^{\frac{1}{2}}\big)=  \tau\left[a^{\frac{p}{q}}\right]^{\frac{q}{p}} \qquad \forall a\in \overline{\sA}_+. 
\end{equation}
This reduces the proof of Proposition~\ref{prop:Intro.Weyl-BirS} to establishing~(\ref{eq:ST.spectral-asymptotics-pos2}). 

Let $a\in \sA_{++}$. If $\alpha \in (0,1]$ and $s>1$, then it follows from \emph{the enhanced CSZ Lemma} (Lemma~\ref{lem:Intro.enhanced-CSZ})  that
\begin{equation*}
  \left( a^{\frac12}|D|^{-\alpha} a^{\frac12}\right)^{s} - |D|^{-\alpha s} a^{s} \in \big(\sL_{(\alpha s)^{-1}p,\infty}\big)_0. 
\end{equation*}
Moreover, by~(\ref{eq:ST.Com-Dq-sL02})--(\ref{eq:ST.Com-Dq-sL03}) we have
\begin{equation*}
 |D|^{-\alpha s} a^{s}-a^{\frac{s}2}|D|^{-\alpha s} a^{\frac{s}2} \in  \big(\sL_{(\alpha s)^{-1}p,\infty}\big)_0. 
\end{equation*}
Thus, 
\begin{equation*}
  \left( a^{\frac12}|D|^{-\alpha} a^{\frac12}\right)^{s} - a^{\frac{s}2}|D|^{-\alpha s} a^{\frac{s}2}  \in \big(\sL_{(\alpha s)^{-1}p,\infty}\big)_0. 
\end{equation*}
Combining this with Proposition~\ref{prop:NCInt.Bir-Sol} we then get
\begin{equation}\label{eq:ST.refined-CSZ-limits}
  \left[ \lim_{j\rightarrow \infty}  j^{\frac{\alpha}{p}} \lambda_j\big( a^{\frac12}|D|^{-\alpha} a^{\frac12}\big)\right]^s= 
  \lim_{j\rightarrow \infty}  j^{\frac{\alpha s}{p}} \lambda_j\left(\big( a^{\frac12}|D|^{-\alpha} a^{\frac12}\big)^s\right)
  =  \lim_{j\rightarrow \infty}  j^{\frac{\alpha s}{p}} \lambda_j\big(a^{\frac{s}2}|D|^{-\alpha s} a^{\frac{s}2}\big),
  \end{equation}
provided any of those limits exists. 

Assume now that Condition~(W) holds. Lemma~\ref{lem:ST.equiv-conformal-Weyl} then ensures that, for all $a\in \overline{\sA}_+$, we have
\begin{equation}\label{eq:ST.spectral-asymptotics-pos2-q=2}
  \lim_{j\rightarrow \infty}  j^{\frac{2}{p}} \lambda_j\big(a^{\frac{1}{2}}D^{-2}a^{\frac{1}{2}}\big)=  \tau\left[a^{\frac{p}{2}}\right]^{\frac{2}{p}}.
\end{equation}
This proves~(\ref{eq:ST.spectral-asymptotics-pos2}) for $q=2$. 

Let $q\leq 1$ and $a\in \sA_{++}$. Specializing~(\ref{eq:ST.refined-CSZ-limits}) to $\alpha=q$ and $s=2q^{-1}$ (so as to have $\alpha s=2$), and using~(\ref{eq:ST.spectral-asymptotics-pos2-q=2}) we obtain
\begin{equation*}
\lim_{j\rightarrow \infty}  j^{\frac{q}{p}} \lambda_j\big( a^{\frac12}|D|^{-q} a^{\frac12}\big)= \left[  
\lim_{j\rightarrow \infty}  j^{\frac{2}{p}} \lambda_j\big(a^{\frac{1}{q}}D^{-2} a^{\frac{1}{q}}\big)\right]^{\frac{q}{2}} = 
\tau\left[a^{\frac{p}{q}}\right]^{\frac{q}{p}}. 
\end{equation*}
 This shows that the spectral asymptotics~(\ref{eq:ST.spectral-asymptotics-pos2}) hold for $a\in \sA_{++}$. Using Proposition~\ref{prop:NCInt.Bir-Sol} and Lemma~\ref{lem:STC.density} and arguing as in the proof of Lemma~\ref{lem:ST.equiv-conformal-Weyl} shows that~(\ref{eq:ST.spectral-asymptotics-pos2}) holds for all $a\in \overline{\sA}_+$ if $q\leq 1$. In particular, for $q=1$ we get
 \begin{equation}\label{eq:ST.spectral-asymptotics-pos2-q=1}
  \lim_{j\rightarrow \infty}  j^{\frac{1}{p}} \lambda_j\big(a^{\frac{1}{2}}|D|^{-1}a^{\frac{1}{2}}\big)=  \tau\left[a^{p}\right]^{\frac{1}{p}} \qquad \forall a\in \overline{\sA}_+.
\end{equation}

It remains to deal with the case $q>1$. In this case, given any $a\in\sA_{++}$, in~(\ref{eq:ST.refined-CSZ-limits}) we may take $\alpha=1$ and $s=q$ and use~(\ref{eq:ST.spectral-asymptotics-pos2-q=1}) to get
\begin{equation*}
 \lim_{j\rightarrow \infty}  j^{\frac{q}{p}} \lambda_j\big(a^{\frac{q}2}|D|^{-q} a^{\frac{q}2}\big)=  
 \left[ \lim_{j\rightarrow \infty}  j^{\frac{1}{p}} \lambda_j\big( a^{\frac12}|D|^{-1} a^{\frac12}\big)\right]^q=  \tau\left[a^{p}\right]^{\frac{q}{p}}.  
\end{equation*}
This shows that the spectral asymptotics~(\ref{eq:ST.spectral-asymptotics-pos2}) are satisfied by all $a\in \sA_{++}^q$. Using Lemma~\ref{lem:STC.density} and Proposition~\ref{prop:NCInt.Bir-Sol} as above shows that it is satisfied 
by all $a\in \overline{\sA}_{+}$. This completes the proof of Proposition~\ref{prop:Intro.Weyl-BirS}. 
\end{proof}

\begin{remark}\label{prop:ST.Weyl-DqaDq-q=2}
For general values $q>0$, the proof of Proposition~\ref{prop:Intro.Weyl-BirS} relies on the enhanced CSZ Lemma.  
However, as the proof of Proposition~\ref{prop:Intro.Weyl-BirS} shows, for $q=2$ we get the spectral asymptotics~(\ref{eq:ST.Weyl-|DpaDp|-tau})--(\ref{eq:Intro.Weyl-DpaDp-tau}) directly from Proposition~\ref{prop:Intro.Com-Dq}, which is a much easier result. Therefore, enhanced CSZ Lemma is not needed if we are interested in proving~(\ref{eq:ST.Weyl-|DpaDp|-tau})--(\ref{eq:ST.Weyl-DpaDp-tau}) only for $q=2$.
 \end{remark}

\section{Semiclassical Weyl Laws}\label{sec:SCWeyl} 
This section is devoted to proving the semiclassical Weyl laws provided by Theorem~\ref{thm:Intro.SC-Weyl-lawDq}. The proof relies on the spectral asymptotics of Proposition~\ref{prop:Intro.Weyl-BirS} and the Birman--Schwinger principle (Proposition~\ref{prop:NCInt.BSP-NCInt}). 

For $q>0$ and $V=V^*\in \sA$, we consider the Schr\"odinger operators,
\begin{equation*}
 H_V^{(q)}(h):= h^{2q}\left(D^2\right)^{q}+V, \qquad h>0. 
\end{equation*}
 Here $H_V^{(q)}(h)$ with domain $\dom(|D|^{2q})$ is a selfadjoint operator which is bounded from below and has pure discrete spectrum. 
 
 We are interested in the behavior of the spectrum of $H_V^{(q)}(h)$ under the semiclassical limit $h\rightarrow 0^+$.  We list the eigenvalues of $H_V^{(q)}(h)$ as a sequence, 
\begin{equation}\label{eq:ST.eigenvalue-Schroedinger}
 \lambda_0\big( H_V^{(q)}(h)\big) \leq   \lambda_1\big( H_V^{(q)}(h)\big) \leq \cdots, 
\end{equation}
where each eigenvalue is repeated according to multiplicity. Given any energy level $\lambda\in \R$, we then define
\begin{equation}\label{eq:ST.counting-Schroedinger1}
 N\big(H_V^{(q)};\lambda\big):= \#\big\{j; \  \lambda_j\big( H_V^{(q)}\big) <\lambda\big\}. 
\end{equation}
In particular, in the notation of Proposition~\ref{prop:NCInt.BSP-NCInt} we have $N^{-}(H_V^{(q)})=N(H_V^{(q)};0)$.  

\begin{proof}[Proof of Theorem~\ref{thm:Intro.SC-Weyl-lawDq}] 
Assume Condition~\textup{(W)} holds.  Let $q>0$ and  $V^*=V\in \overline{\sA}$. We have to determine the limit of $h^{p} N(H_V^{(q)}(h);\lambda)$ under the semiclassical limit $h\rightarrow 0^+$. 

Combining Proposition~\ref{prop:Intro.Weyl-BirS}  with the Birman--Schwinger principle~(\ref{eq:SC.NHV=NKV-SC1}) for $H_0=(D^2)^{q}$ gives
 \begin{equation*}
  \lim_{h\rightarrow 0^+} h^{p} N\big(H_V^{(q)}(h);0\big)= 
 \left[ \lim_{j\rightarrow \infty} j^{\frac{2q}{p}} \lambda_j^{-}\left(|D|^{-q}V|D|^{-q}\right)\right]^{\frac{p}{2q}}= 
  \tau\Big[V_{-}^{\frac{p}{2q}}\Big]. 
\end{equation*}
Given any $\lambda\in \R$, by substituting $V-\lambda$ for $V$ above we then get
 \begin{equation}\label{eq:ST.SC-Weyl-Dq1}
  \lim_{h\rightarrow 0^+} h^{p} N\big(H_V^{(q)}(h);\lambda\big)=  \lim_{h\rightarrow 0^+} h^{p} N\big(H_{V-\lambda}^{(q)}(h);0\big)
  =\tau\Big[\left(V-\lambda\right)_{-}^{\frac{p}{2q}}\Big]. 
  \end{equation}
 This proves the result. 
\end{proof}

\begin{remark}
The most important case for the semiclassical Weyl laws~(\ref{eq:ST.SC-Weyl-Dq1}) is $q=1$. In this case, the asymptotic is a consequence of the spectral asymptotics~(\ref{eq:Intro.Weyl-DpaDp-tau}) for $q=2$. As mentioned in Remark~\ref{prop:ST.Weyl-DqaDq-q=2}, for $q=2$ that spectral asymptotics can be deduced from Proposition~\ref{prop:Intro.Com-Dq} without using the enhanced CSZ Lemma. Therefore, for $q=1$ we have a much simpler proof of the semiclassical Weyl laws~(\ref{eq:ST.SC-Weyl-Dq1}). 
 \end{remark}

\begin{remark}\label{rmk:ST.SC-WLaw-interval}
It is a routine argument to deduce from~(\ref{eq:ST.SC-Weyl-Dq1}) semiclassical Weyl laws for the number of eigenvalues in bands $(\mu, \lambda)$, $\mu<\lambda$. Given $\mu<\lambda$, set 
\begin{equation*}
 N\big(H_V^{(q)}(h);\mu,\lambda\big):= \#\left\{j; \  \lambda_j\big( H_V^{(q)}(h)\big) \in (\mu,\lambda)\right\}. 
\end{equation*}
For all $\epsilon \in (0,\lambda-\mu)$, we have
\begin{equation*}
   h^{p}N\big(H_V^{(q)};\lambda\big) -  h^{p}N\big(H_V^{(q)};\mu+\epsilon\big)\leq h^{p}N\big(H_V^{(q)};\mu,\lambda\big)\leq  h^{p}N\big(H_V^{(q)};\lambda\big) -  h^{p}N\big(H_V^{(q)};\mu-\epsilon\big). 
\end{equation*}
Combining this with the semiclassical Weyl laws~(\ref{eq:ST.SC-Weyl-Dq1}) gives
\begin{gather*}
 \limsup_{h\rightarrow 0^+} h^{p}N\big(H_V^{(q)};\mu,\lambda\big) \leq \tau\Big[\left(V-\lambda\right)_{-}^{\frac{p}{2q}}\Big]- \tau\Big[\left(V-\mu+\epsilon\right)_{-}^{\frac{p}{2q}}\Big],\\
  \liminf_{h\rightarrow 0^+} h^{p}N\big(H_V^{(q)};\mu,\lambda\big) \geq \tau\Big[\left(V-\lambda\right)_{-}^{\frac{p}{2q}}\Big]- \tau\Big[\left(V-\mu-\epsilon\right)_{-}^{\frac{p}{2q}}\Big]. 
 \end{gather*}
We claim that 
\begin{equation}\label{eq:ST.SC-Weyl-continuity-tau}
 \lim_{\epsilon \rightarrow 0^+}\tau\Big[\left(V-\mu\pm\epsilon\right)_{-}^{\frac{p}{2q}}\Big]= \tau\Big[\left(V-\mu\right)_{-}^{\frac{p}{2q}}\Big]. 
\end{equation}
This would imply the semiclassical Weyl laws, 
 \begin{equation}\label{eq:ST.SC-WLaw-interval}
  \lim_{h\rightarrow 0^+}  h^{p}   N\big(H_V^{(q)}(h);\mu,\lambda\big)=  \tau\Big[\left(V-\lambda\right)_{-}^{\frac{p}{2q}}\Big]- \tau\Big[\left(V-\mu\right)_{-}^{\frac{p}{2q}}\Big].
\end{equation}

To get~(\ref{eq:ST.SC-Weyl-continuity-tau}), set $x=V-\mu$ and $r=p(2q)^{-1}$.  Using the positivity of $\tau$ and the uniform continuity of the function $t\rightarrow t_{-}^r$ on the compact set $\Sp(x)$ we get
\begin{align*}
 \left|\tau\left[(x\pm\epsilon)_{-}^r\right]-\tau\left[x_{-}^r\right]\right| & \leq \tau(1) \left\| (x\pm\epsilon)_{-}^r-x_{-}^r\right\| \\
 & \leq \tau(1) \sup_{t\in \Sp(x)} \left| (t\pm\epsilon)_{-}^r-t_{-}^r\right|  \longrightarrow 0 \qquad \text{as $\epsilon\rightarrow 0^+$}.   
\end{align*}
This gives~(\ref{eq:ST.SC-Weyl-continuity-tau}) and proves the semiclassical Weyl laws~(\ref{eq:ST.SC-WLaw-interval}). 
\end{remark}

\section{Integration Formulas}\label{sec:IntForm} 
In this section, we prove the integration formulas provided by Theorem~\ref{thm:Intro.Integration} as a further consequence of Proposition~\ref{prop:Intro.Weyl-BirS}.

\begin{proof}[Proof of Theorem~\ref{thm:Intro.Integration}] 
Let $a\in \overline{\sA}$ and $q>0$. The spectral asymptotics~(\ref{eq:ST.Weyl-|DpaDp|-tau}) for $q=p$ gives
 \begin{equation}\label{eq:ST.Integration-proof.muj}
 \lim_{j\rightarrow \infty} j \mu_j\left( |D|^{-\frac{p}{2}}a|D|^{-\frac{p}{2}}\right)= \tau\left[ |a|\right]. 
\end{equation}
 Combining this with the 3rd part of Proposition~\ref{prop:NCInt.Weyl-strong} shows that the operator $ ||D|^{-p/2}a|D|^{-p/2}|$ is spectrally measurable, and we have
  \begin{equation*}
  \bint \big| |D|^{-\frac{p}{2}}a|D|^{-\frac{p}{2}}\big|=  \tau\left[ |a|\right]. 
\end{equation*}
Similarly, the spectral asymptotics~(\ref{eq:ST.Weyl-|DpaDp|-tau}) for $q=p$ also give $\lim_{j\rightarrow \infty} j \mu_j(a|D|^{-p})= \tau[|a|]$. Thus, by the 3rd part of Proposition~\ref{prop:NCInt.Weyl-strong} the operator $|a|D|^{-p}|$ is spectrally measurable, and its NC integral is equal to $\tau[|a|]$. 
 
Likewise, if $a^*=a$, then by~(\ref{eq:Intro.Weyl-DpaDp-tau}) we have
\begin{equation*}
  \lim_{j\rightarrow \infty} j \lambda_j^\pm\left(|D|^{-\frac{p}{2}}a|D|^{-\frac{p}{2}}\right)= \tau\left[ a_\pm\right]. 
\end{equation*}
It then follows from Proposition~\ref{prop:NCInt.Weyl-strong} that the operator $ |D|^{-p/2}a|D|^{-p/2}$ is spectrally measurable, and we have
\begin{align}
 \bint |D|^{-\frac{p}{2}}a|D|^{-\frac{p}{2}} & =  \lim_{j\rightarrow \infty}j \lambda_j^+\big(|D|^{-\frac{p}{2}}a|D|^{-\frac{p}{2}}\big) - \lim_{j\rightarrow \infty}j \lambda_j^{-}\big(|D|^{-\frac{p}{2}}a|D|^{-\frac{p}{2}}\big) \nonumber\\ 
&  =\tau\big[a_+\big]-\tau\big[a_{-}\big]
\label{eq:Int-SC.selfadjoint}\\
& =\tau\big[a\big]. \nonumber
\end{align}

If $a$ is not selfadjoint, then the selfadjoint part of $ |D|^{-p/2}a|D|^{-p/2}$ is equal to 
\begin{equation*}
\frac{1}{2}\left(|D|^{-\frac{p}{2}}a|D|^{-\frac{p}{2}}+|D|^{-\frac{p}{2}}a^*|D|^{-\frac{p}{2}} \right)= 
 |D|^{-\frac{p}{2}}(\Re a)|D|^{-\frac{p}{2}}. 
\end{equation*}
Likewise, its imaginary part is equal to $ |D|^{-p/2}(\Im a)|D|^{-p/2}$. Therefore, the real and imaginary parts both are spectrally measurable, and so 
$ |D|^{-p/2}a|D|^{-p/2}$ is spectrally measurable. Moreover, by using~(\ref{eq:Int-SC.selfadjoint}) we get 
\begin{align*}
 \bint  |D|^{-\frac{p}{2}}a |D|^{-\frac{p}{2}} & = \bint \Re\left( |D|^{-\frac{p}{2}}a|D|^{-\frac{p}{2}}\right) + i \bint \Im\left( |D|^{-\frac{p}{2}}a|D|^{-\frac{p}{2}}\right)\\
& = \tau\left[\Re a\right]+i\tau\left[\Im a\right]\\
& =\tau[a]. 
\end{align*}

Finally, by Corollary~\ref{cor:ST.Com-Dq-sL1} we have
\begin{equation*}
  a|D|^{-p} = |D|^{-\frac{p}{2}}a|D|^{-\frac{p}{2}} \quad \bmod \big(\sL_{1,\infty})_0. 
\end{equation*}
It then follows from Proposition~\ref{prop:NCInt.Bir-Sol}  that $a|D|^{-p}$ is spectrally measurable. Moreover, 
as the NC integral vanishes on $(\sL_{1,\infty})_0$, we have 
\begin{equation}\label{eq:ST.bint-aDp-tau}
 \bint a|D|^{-p} = \bint |D|^{-\frac{p}{2}}a|D|^{-\frac{p}{2}}=\tau\big[a\big]. 
\end{equation}
This completes the proof of Theorem~\ref{thm:Intro.Integration}. 
\end{proof}

\begin{remark}
 Under the Tauberian condition~(H) the integration formulas~(\ref{eq:ST.bint-aDp-tau}) can also be deduced from~\cite[Theorem~1.5]{SZ:JFA17}. We can also deduce them from a slightly stronger form of Condition~\textup{(Z)} by using~\cite[Theorem~1.4]{SUZ:IUMJ17} (see also~\cite[Theorem~1.2.7]{SZ:Ast23}).  
\end{remark}

\begin{remark}\label{rmk:ST-identification-tau}
 The formula~(\ref{eq:ST.bint-aDp-tau}) identifies $\tau$ with the functional $a\rightarrow \bint a|D|^{-p}$. Thus, given any $q>0$ and $V^*=V\in \overline{\sA}$, we may rewrite the semiclassical Weyl laws~(\ref{eq:ST.SC-Weyl-Dq1}) in the form, 
\begin{gather*}
  \lim_{h\rightarrow 0^+} h^{p}N\big( H_V^{(q)}(h);\lambda\big) = \bint \left(V-\lambda\right)_{-}^{\frac{p}{2q}}|D|^{-p}, \quad \lambda\in \R. 
  \end{gather*}
\end{remark}

\begin{remark}
 By using the spectral asymptotics~(\ref{eq:Intro.Weyl-|DpaDp|-tau})--(\ref{eq:Intro.Weyl-DpaDp-tau}) and arguing as in the proof of Theorem~\ref{thm:Intro.Integration} above, it can be further shown that, for any $q>0$, $q\neq p$, the operators $||D|^{-q/2}a|D|^{-q/2}|^{p/q}$, $|a|D|^{-q}|^{p/q}$, and  $(|D|^{-q/2}a|D|^{-q/2})_\pm^{p/q}$ (if $a^*=a$), are spectrally measurable, and we have 
\begin{gather*}
   \bint \big| |D|^{-\frac{q}{2}}a|D|^{-\frac{q}{2}}\big|^{\frac{p}{q}}= \bint \big| a|D|^{-q}\big|^{\frac{p}{q}} =  \tau\left[ |a|^{\frac{p}{q}}\right],\\ 
   \bint \big( |D|^{-\frac{q}{2}}a|D|^{-\frac{q}{2}}\big)_\pm^{\frac{p}{q}}=  \tau\left[ (a_\pm)^{\frac{p}{q}}\right]. 
\end{gather*}
\end{remark}

\section{Extension to Unbounded Potentials}\label{sec:Unbounded}
In this section, we extend the previous main results of this paper to unbounded potentials. We shall prove Proposition~\ref{prop:Intro.Lq-asymptotics} and use it to obtain Theorem~\ref{thm:Intro.SC-Weyl-law-sVr} and Theorem~\ref{thm:Intro.Integration-sVr}. These results provide a general paradigm for extending to suitable classes of unbounded potentials the spectral asymptotics of Proposition~\ref{prop:Intro.Weyl-BirS} and the
semiclassical Weyl laws and integration formulas provided by Theorem~\ref{thm:Intro.SC-Weyl-lawDq} and Theorem~\ref{thm:Intro.Integration}. 

The approach of this section is a mere elaboration of the well-known approach of Birman--Solomyak (see, e.g., \cite{BS:AMST80})  and Simon~\cite{Si:TAMS76} to semiclassical Weyl laws for $L^p$-potentials. The only technical difference is the fact that, due to the noncommutativity of our general setup, we need to work with noncommutative $L^p$-spaces. 

\subsection{Cwikel-type estimates for spectral triples} On $\R^n$ or bounded domains in $\R^n$, the extension of the semiclassical Weyl laws to optimal $L^r$-spaces of potentials was made possible by the CLR inequality of Cwikel~\cite{Cw:AM77}, Lieb~\cite{Li:BAMS76, Li:1980}, and Rozenblum~\cite{Ro:SMD72,Ro:SM76}. The approach of Cwikel~\cite{Cw:AM77} relied on weak-Schatten quasi-norm estimates for Birman--Schwinger operators that were conjectured by Simon~\cite{Si:TAMS76}. Combining these estimates with the perturbation theory of Birman--Solomyak (see Proposition~\ref{prop:NCInt.Bir-Sol}) further allows one to extend the spectral asymptotics for Birman--Schwinger operators and the integration formulas to the relevant $L^r$-classes (see~\cite{LSZ:PLMS20, Ro:JST22, RS:EMS21, So:PLMS95, SZ:MS22} and the references therein for various extensions of the Cwikel estimates). In the setup of spectral triples we can formulate Cwikel-type estimates for spectral triples as follows. 

For $s>0$ denote by $\sH^{(s)}$ the Sobolev space with Bessel potential $(1+D^2)^{s/2}$, i.e., the Hilbert space consisting of $\dom(1+D^2)^{s/2}$ equipped with the Hilbert norm $\xi \rightarrow \|(1+D^2)^{s/2}\xi\|$. We denote by $\sH^{(-s)}$ its antilinear dual, i.e., the Hilbert space of continuous antilinear forms on $\sH^{(s)}$. 
Any operator $V\in \sL(\sH^{(s)},\sH^{(-s)})$ defines a quadratic form on $\sH^{(s)}$ by 
\begin{equation*}
 Q_V(\xi,\eta)=\acou{V\xi}{\eta}, \qquad \xi,\eta \in \sH^{(s)}, 
\end{equation*}
where $\acou{\cdot}{\cdot}:\sH^{(-s)}\times \sH^{(s)}\rightarrow \C$ is the duality pairing. 
The adjoint $V^*\in \sL(\sH^{(s)},\sH^{(-s)})$ is given by
\begin{equation*}
 \acou{V^*\xi}{\eta}= \overline{\acou{V\eta}{\xi}}, \qquad \xi,\eta \in \sH^{(s)}. 
\end{equation*}
In particular, $Q_{V^*}=\overline{Q_V}$, and so $Q_V$ is symmetric if and only if $V^*=V$. Note also that any $T\in \sL(\sH)$ induces an operator $T:\sH^{(s)}\rightarrow \sH^{(-s)}$ such that
\begin{equation*}
 \acou{T\xi}{\eta}=\scal{T\xi}{\eta} \qquad \forall \xi, \eta \in \sH^{(s)}. 
\end{equation*}
This yields a continuous $*$-embedding of $\sL(\sH)$ into $\sL(\sH^{(s)},\sH^{(-s)})$, and so this induces a continuous $*$-embedding of $\overline{\sA}$ into $\sL(\sH^{(s)},\sH^{(-s)})$. 

Given any $q>0$, the Hilbert space $\sH^{(q)}$ is contained in the domain of the quadratic form of $(D^2)^q$. In addition, 
$|D|^{-q}$ maps continuously $\sH$ to $\sH^{(q)}$. It also extends to a bounded operator $|D|^{-q}:\sH^{(-q)}\rightarrow \sH$ given by
\begin{equation*}
 \scal{|D|^{-q}\xi}{\eta}= \acou{\xi}{|D|^{-q}\eta}, \qquad \xi\in \sH^{(-q)}, \ \eta \in \sH. 
\end{equation*}
Thus, if $V\in \sL(\sH^{(q)},\sH^{(-q)})$, then the operator $|D|^{-q}V|D|^{-q}$ is bounded on $\sH$. This operator is selfadjoint (resp., compact) if $V$ is selfadjoint (resp., compact). 

It follows from this that if $V=V^*\in \sL(\sH^{(q)},\sH^{(-q)})$ is compact, then the symmetric quadratic form $Q_V$ is $(D^2)^q$-form compact. This allows us to define the operator $H_V^{(q)}:=(D^2)^q+V$ as a form sum, i.e., as the selfadjoint operator whose quadratic form is $Q_{(D^2)^q}+Q_V$ . This operator is bounded from below and has compact resolvent, and hence it has pure discrete spectrum (see, e.g., \cite{Si:AMS15}). 

Assume now that Condition~\textup{(W)} holds and $\tau$ is the restriction of a (positive) faithful normal finite trace $\tau:\sM\rightarrow \C$, where $\sM\subseteq \sL(\sH)$ is the von Neumann algebra generated by $\sA$, i.e., its double commutant in $\sL(\sH)$. A closed densely defined operator on $\sH$ is $\sM$-affiliated if it commutes with all unitary operators in the commutant $\sM'$ of $\sM$ in $\sL(\sH)$. Since $\tau$ is finite, every such operator is $\tau$-measurable in the sense of~\cite{FK:PJM86, Ne:JFA15}. 

For $r\in (0,\infty)$, the noncommutative $L^r$-space $L_r(\sM)$\footnote{Throughout this article we shall use subscripts for the exponents of NC $L^p$-spaces.} consists of $\sM$-affiliated operators $x$ on $\sH$ such that 
\begin{equation}\label{eq:Cr.definition-NCLr}
 \tau\left(|x|^r\right):= \int_0^\infty \lambda^r d\tau\big(E(\lambda)\big)<\infty, 
\end{equation}
where $E(\lambda):=\car_{[0,\lambda)}(|x|)$, $\lambda \geq 0$, is the spectral resolution of $|x|$. This is a quasi-Banach space with respect to the quasi-norm, 
\begin{equation}\label{eq:Cr.Lr-quasi-norm}
 \|x\|_r= \left(\tau\big[|x|^r\big]\right)^{\frac1{r}}, \qquad x\in L_r(\sM). 
\end{equation}
For $r\geq 1$, we get a Banach space, since in this case $\|\cdot\|_r$ is a norm. For instance, if $(X,\mu)$ is a finite measure space, then we may take $\sM=L^\infty(X,\mu)$ and $\tau(f)=\int_X f(x)d\mu(x)$, and in this case $L_r(\sM)=L^r(X,\mu)$. 

As $\tau$ is normal, $\sM$ is dense in $L_r(\sM)$ (see, e.g., \cite{FK:PJM86, PX:Handbook03}). Moreover, a standard argument using Kaplansky's density theorem shows that $\sA$ is dense in $\sM$ with respect to the $L_r$-quasi-norm. Therefore, the algebra $\sA$ is dense in $L_r(\sM)$. Thus, we may also regard $L_r(\sM)$ as the completion of $\sA$ with respect to the quasi-norm~(\ref{eq:Cr.Lr-quasi-norm}). We refer to~\cite{FK:PJM86, Ku:TAMS58, Ne:JFA15, PX:Handbook03} for more background on noncommutative $L^p$-spaces.

We note that for $r'\leq r$ we have a continuous inclusion of $L_r(\sM)$ into $L_{r'}(\sM)$. Moreover, the map $x\rightarrow x^*$ uniquely extends to an antilinear isometric map of $L_r(\sM)$ onto itself. It is a non-trivial fact that, for $r\geq 1$, the absolute value map  $x\rightarrow |x|$ uniquely extends to continuous map on $L_r(\sM)$ (see~\cite{Ko:TAMS84}). 

\begin{condition}[Condition ($\textup{C}_r$)] Let $r>0$, and set $\hat{r}=\max(r,1)$. There is a $*$-invariant norm $\|\cdot\|_{(r)}$ on $\sA$ such that
\begin{enumerate}
 \item[(i)] The norm $\|\cdot\|_{(r)}$ is continuous with respect to the operator norm. 
 
  \item[(ii)] The inclusion of $\sA$ into $L_{\hat{r}}(\sM)$ is continuous with respect to the $\|\cdot\|_{(r)}$-topology.
 
 \item[(iii)] There is a constant $C_r>0$ such that
\begin{equation}\label{eq:ST.Cwikel}
 \big\| (1+D^2)^{-\frac{p}{4r}}a(1+D^2)^{-\frac{p}{4r}}\big\|_{r,\infty} \leq C_r\|a\|_{(r)} \qquad \forall a\in \sA. 
\end{equation}
\end{enumerate}
 \end{condition}
 
\begin{remark}
In (ii) we use the space $L_{\hat{r}}(\sM)$ instead of $L_{r}(\sM)$ to ensure the continuity of the absolute value map $x\rightarrow |x|$ on  $L_{\hat{r}}(\sM)$  (\emph{cf}.\ proof of Proposition~\ref{prop:Intro.Lq-asymptotics} below). In various examples, for $r<1$ the best $\|\cdot\|_{(r)}$-norm actually agrees with the $L_1$-norm $\|\cdot\|_{1}$ (see, e.g., \cite{BS:AMST80, MP:JMP22}). 
\end{remark}
 
 \begin{remark}
 We refer to~\cite{LSZ:PLMS20} for extensions of the Cwikel estimates to a large class of noncommutative $L^p$-spaces. This was the main impetus for the Cwikel estimates for noncommutative tori in~\cite{MP:JMP22, MP:AIM23} (see also Section~\ref{sec:QT}). It would be interesting to use the framework of~\cite{LSZ:PLMS20} to get further examples of spectral triples satisfying Condition~\textup{($\textup{C}_{r}$)}.
 \end{remark}

\subsection{Spectral asymptotics}  
 If Condition~($\textup{C}_r$) holds, then we denote by $\sV_r$ the Banach space completion of $\sA$ with respect to the norm~$\|\cdot \|_{(r)}$. By~(i) and the density of $\sA$ in $\overline{\sA}$, the norm $\|\cdot\|_{(r)}$ uniquely extends to a continuous norm on $\overline{\sA}$, and so we have a continuous inclusion of $\overline{\sA}$ into $\sV_r$. We also have a continuous inclusion of $\sV_r$ into $L_{\hat{r}}(\sM)$ thanks to~(ii). Moreover, the involution $a\rightarrow a^*$ of $\sA$ uniquely extends to an isometric antilinear involution of $\sV_r$. 

Set $\Lambda=(1+D^2)^{1/2}$. The estimate~(\ref{eq:ST.Cwikel}) is an abstract analogue of Cwikel's estimates. It implies that 
$\sA\ni a\rightarrow \Lambda^{-\frac{p}{2r}} a\Lambda^{-\frac{p}{2r}}\in \sL_{r,\infty}$ uniquely extends to a continuous linear map on $\sV_r$. In particular, this provides us with a continuous inclusion of $\sV_r$ into the closed subspace of compact operators in $\sL(\sH^{(\frac{p}{2r})},\sH^{(-\frac{p}{2r})})$. Here any $x\in \sV_r$ is identified with the operator given by
\begin{equation*}
 \acou{x\xi}{\eta}=\bigscal{(\Lambda^{-\frac{p}{2r}}x\Lambda^{-\frac{p}{2r}})(\Lambda^{\frac{p}{2r}}\xi)}{\Lambda^{\frac{p}{2r}}\eta}, \qquad \xi,\eta \in \sH^{(\frac{p}{2r})}. 
\end{equation*}

We are now in a position to extend Proposition~\ref{prop:Intro.Weyl-BirS} to unbounded potentials.

\begin{proposition}\label{prop:Intro.Lq-asymptotics}
Suppose that Condition~\textup{(W)} is satisfied.  Let $q>0$, and assume further that Condition~\textup{($\textup{C}_{r}$)} holds with $r:=pq^{-1}$. 
Given any $x\in \sV_r$, we have 
\begin{gather}\label{eq:Intro.Weyl-|DpaDp|-tau-sVr}
 \lim_{j\rightarrow \infty}j^{\frac{q}{p}} \mu_j\big(|D|^{-\frac{q}{2}}x|D|^{-\frac{q}{2}}\big) = \tau\left[|x|^{\frac{p}{q}}\right]^{\frac{q}{p}},\\
 \lim_{j\rightarrow \infty} j^{\frac{q}{p}} \lambda_j^\pm\big(|D|^{-\frac{q}{2}}x|D|^{-\frac{q}{2}}\big) =\tau\left[\big(x_\pm\big)^{\frac{p}{q}}\right]^{\frac{q}{p}} \quad (\text{if}\ x^*=x). \label{eq:Intro.Weyl-DpaDp-tau-sVR}
\end{gather}
  \end{proposition}
\begin{proof}
We need to prove that the spectral asymptotics~(\ref{eq:Intro.Weyl-|DpaDp|-tau-sVr})--(\ref{eq:Intro.Weyl-DpaDp-tau-sVR}) hold for every $x\in \sV_r$.

By part~(ii) of Condition~\textup{($\textup{C}_{r}$)} we have a continuous inclusion of $\sV_r$ into $L_{\hat{r}}(\sM)$.  As $\hat{r}\geq 1$, the map $x\rightarrow |x|$ is continuous from $L_{\hat{r}}(\sM)$ to itself. 
We also have a continuous inclusion of $L_{\hat{r}}(\sM)$ into $L_r(\sM)$, since $\hat{r}\geq r$. Therefore, the absolute value map $x\rightarrow |x|$ induces a continuous map from $\sV_r$ to $L_r(\sM)$. Moreover, if we denote by $\sV_r^\sa$ the closed \emph{real} subspace of selfadjoint elements of $\sV_r$, we also get continuous maps 
 $x\rightarrow x_\pm= \frac{1}{2}(|x|\pm x)$ from $\sV_r^\sa$ to $L_r(\sM)$. Composing all those maps with the quasi-norm $\|x\|_r=\tau[|x|^{p/q}]^{q/p}$ we then get continuous functions, 
\begin{equation*}
\sV_r\ni x \longrightarrow  \tau\left[|x|^{\frac{p}{q}}\right]^{\frac{q}{p}}, \qquad \sV_r^\sa \ni x \longrightarrow  \tau\left[(x_\pm)^{\frac{p}{q}}\right]^{\frac{q}{p}}. 
\end{equation*}

Bearing this in mind, let $x=x^*\in \sV_r$ and $(x_\ell)_{\ell \geq 0}\subseteq \overline{\sA}$ be such that $x_\ell^*=x_\ell\rightarrow x$ in $\sV_r$. 
The estimate~(\ref{eq:ST.Cwikel}) ensures that $|D|^{-q/2}x_\ell |D|^{-q/2}\rightarrow |D|^{-q/2}x |D|^{-q/2}$ in $\sL_{q^{-1}p,\infty}$. 
Therefore, by using Proposition~\ref{prop:NCInt.Bir-Sol} and Proposition~\ref{prop:Intro.Weyl-BirS} and the continuity of the function 
$x \rightarrow  \tau[(x_\pm)^{p/q}]^{q/p}$ on $\sV_r$ we obtain
\begin{align*}
 \lim_{j\rightarrow \infty}j^{\frac{q}{p}} \lambda^\pm_j\big(|D|^{-q/2}x|D|^{-q/2}\big) & = \lim_{\ell \rightarrow \infty} 
 \lim_{j\rightarrow \infty}j^{\frac{q}{p}}  \lambda^\pm_j\big(|D|^{-q/2}x_\ell |D|^{-q/2}\big)\\
  &= \lim_{\ell \rightarrow \infty}  \tau\left[(x_\ell)_\pm^{\frac{p}{q}}\right]^{\frac{q}{p}} =  \tau\left[(x)_\pm^{\frac{p}{q}}\right]^{\frac{q}{p}}. 
\end{align*}
This proves~(\ref{eq:Intro.Weyl-DpaDp-tau-sVR}). The spectral asymptotics~(\ref{eq:Intro.Weyl-|DpaDp|-tau-sVr}) are proved similarly. This completes the proof of Proposition~\ref{prop:Intro.Lq-asymptotics}. 
\end{proof}

\subsection{Proof of Theorem~\ref{thm:Intro.SC-Weyl-law-sVr} and Theorem~\ref{thm:Intro.Integration-sVr}}
Let $q>0$, and set $r=p(2q)^{-1}$. If we assume Condition~\textup{($\textup{C}_{r}$)}, then, as mentioned above, we have a continuous inclusion of 
$\sV_r$ into the space of compact operators in $\sL(\sH^{(q)},\sH^{(-q)})$. Given any $V=V^*\in \sV_r$ this allows us to define the Schr\"odinger operators $H_V^{(q)}(h):= h^{2q}(D^2)^{q}+V$, $h>0$, as form sums as above.  As mentioned above, we get a selfadjoint operator which is bounded from below and has pure discrete spectrum. We list its eigenvalues as a non-decreasing sequence as in~(\ref{eq:ST.eigenvalue-Schroedinger}). We then define the counting functions $N(H_V^{(q)}(h);\lambda)$ as in~(\ref{eq:ST.counting-Schroedinger1}). 
 
 Granted this preparation we are ready to prove Theorem~\ref{thm:Intro.SC-Weyl-law-sVr}.  
  
\begin{proof}[Proof of Theorem~\ref{thm:Intro.SC-Weyl-law-sVr}] 
Let $q>0$, set $r=p(2q)^{-1}$, and assume Condition~(W) and Condition~\textup{($\textup{C}_{r}$)} are both satisfied. In addition, let $V=V^*\in \sV_r$.  Given any $\lambda\in \R$, by using Proposition~\ref{prop:Intro.Lq-asymptotics} and the Birman--Schwinger principle in its version of Proposition~\ref{prop:NCInt.BSP-NCInt} and by arguing along the same lines as in the proof of Theorem~\ref{thm:Intro.SC-Weyl-lawDq} we get
\begin{equation}\label{eq:ST.SC-Weyl-law-sVr}
  \lim_{h\rightarrow 0^+} h^{p} N\big(H_V^{(q)}(h);\lambda\big)= \tau\Big[\left(V-\lambda\right)_{-}^{\frac{p}{2q}}\Big]. 
\end{equation}
This gives the result. 
\end{proof}

\begin{remark}
As with (selfadjoint) potentials in $\overline{\sA}$ (\emph{cf}.\ Remark~\ref{rmk:ST.SC-WLaw-interval}), the semiclassical Weyl laws~(\ref{eq:ST.SC-Weyl-law-sVr}) imply semiclassical Weyl laws for the number of eigenvalues in bands $(\mu,\lambda)$, $\mu<\lambda$, for any potential $V\in\sV_r^\sa$. Indeed, as mentioned in the proof of Proposition~\ref{prop:Intro.Lq-asymptotics},  the map $x\rightarrow  \tau[(x_{-})^{{p}/{2q}}]^{2q/p}$ is continuous on $\sV_r^\sa$. It follows that 
\begin{equation*}
 \lim_{\epsilon \rightarrow 0^+}\tau\Big[\left(V-\mu\pm\epsilon\right)_{-}^{\frac{p}{2q}}\Big]= \tau\Big[\left(V-\mu\right)_{-}^{\frac{p}{2q}}\Big]. 
\end{equation*}
We can then argue along the same lines as in Remark~\ref{rmk:ST.SC-WLaw-interval} to see that for all  $V\in\sV_r^\sa$, we have
\begin{equation*}
  \lim_{h\rightarrow 0^+}  h^{p}   N\big(H_V^{(q)}(h);\mu,\lambda\big) =  \tau\Big[\left(V-\lambda\right)_{-}^{\frac{p}{2q}}\Big]- \tau\Big[\left(V-\mu\right)_{-}^{\frac{p}{2q}}\Big].
\end{equation*}
\end{remark}

Finally, we prove Theorem~\ref{thm:Intro.Integration-sVr}. 

\begin{proof}[Proof of Theorem~\ref{thm:Intro.Integration-sVr}]
The proof follows the same outline as that of the proof of Theorem~\ref{thm:Intro.Integration}. Assume that Condition~(W) and Condition~\textup{($\textup{C}_{1}$)} hold. Let $x\in \sV_1$. As with~(\ref{eq:ST.Integration-proof.muj}), specializing the spectral asymptotics~(\ref{eq:Intro.Weyl-|DpaDp|-tau-sVr}) to $q=p$ yields
 \begin{equation*}
 \lim_{j\rightarrow \infty} j \mu_j\left( |D|^{-\frac{p}{2}}x|D|^{-\frac{p}{2}}\right)= \tau\left[ |x|\right]. 
\end{equation*}
As in the proof of Theorem~\ref{thm:Intro.Integration}, Proposition~\ref{prop:NCInt.Weyl-strong} then ensures that $ ||D|^{-p/2}x|D|^{-p/2}|$ is spectrally measurable, and we have
  \begin{equation*}
  \bint \big| |D|^{-\frac{p}{2}}x|D|^{-\frac{p}{2}}\big|=  \tau\left[ |x|\right]. 
\end{equation*}

Suppose that $x^*=x$. The spectral asymptotics~(\ref{eq:Intro.Weyl-DpaDp-tau-sVR}) for $q=p$ then gives 
\begin{equation*}
 \lim_{j\rightarrow \infty} j \lambda_j^\pm\left(|D|^{-\frac{p}{2}}x|D|^{-\frac{p}{2}}\right)= \tau\left[ x_\pm\right]. 
\end{equation*}
In the same way as in the proof of Theorem~\ref{thm:Intro.Integration}, combining this with Proposition~\ref{prop:NCInt.Weyl-strong} shows that 
$ |D|^{-p/2}x|D|^{-p/2}$ is spectrally measurable, and we have
\begin{equation*}
 \bint |D|^{-\frac{p}{2}}x|D|^{-\frac{p}{2}}= \tau\big[x_+\big]-\tau\big[x_{-}\big] =\tau\big[x\big].
\end{equation*}
Arguing as in the proof of Theorem~\ref{thm:Intro.Integration} shows that properties actually hold for all $x\in \sV_1$. This completes the proof of Theorem~\ref{thm:Intro.Integration-sVr}.  
\end{proof}

\section{Examples}\label{sec:Examples}
In this section, we present in more detail the examples~\ref{ex:BD}--\ref{ex:Steklov} of the introduction, namely, bounded Euclidean domains, closed Riemannian manifolds, and the Steklov eigenvalue problem. We check that Condition~(W) and Condition~($\textup{C}_r$) are satisfied, and we deduce semiclassical Weyl laws, integration formulas, and Weyl laws for weighted Steklov problems from the main results of this paper. 

\subsection{Bounded Euclidean domains}\label{subsec:BD} 
Assume $\Omega$ is a bounded domain in $\R^n$, $n\geq 2$,  with smooth boundary $\partial \Omega$. 
We let $\Delta_\Omega$ be either the \emph{Dirichlet Laplacian} or the \emph{Neumann Laplacian} on $\Omega$. That is, $\Delta_\Omega$ is the selfadjoint realization of the Laplacian $\Delta:=-(\partial_1^2+\cdots+\partial_n^2)$ whose domain consists of all $u\in W^{2,2}(\Omega)$ satisfying the boundary condition $u_{|\partial \Omega}=0$ (Dirichlet case) or $(\partial_\nu u)_{|\partial \Omega}=0$ (Neumann case), where $\partial_\nu$ is the outward normal derivative. Equivalently, $\Delta_\Omega$ is the selfadjoint operator defined by the quadratic form
\begin{equation*}
 u\longrightarrow \int_\Omega |\nabla u(x)|^2dx
\end{equation*}
with form domain $W^{1,2}_0(\Omega)$ (Dirichlet case) or $W^{1,2}(\Omega)$ (Neumann case). 
 
 In both the Dirichlet and Neumann cases we get a strongly elliptic second-order differential operator with non-negative purely discrete spectrum. Its kernel is $\{0\}$ in the Dirichlet case and $\C\cdot 1$ in the Neumann case. 

To get a spectral triple we consider the $*$-algebra, 
\begin{equation*}
 \dot{C}^\infty(\Omega):=\C1+C_c^\infty(\Omega). 
\end{equation*}
That is, $\dot{C}^\infty(\Omega)$ is the algebra of $C^\infty$-functions on $\Omega$ that are constant near the boundary $\partial \Omega$.  This is a $*$-subalgebra of the $C^*$-algebra of bounded continuous functions on $\Omega$, and so it is represented in $L^2(\Omega)$ by bounded multiplication operators. Its norm closure is the $C^*$-algebra,
\begin{equation*}
 \dot{C}(\Omega):= \C + C_0(\Omega),
 \end{equation*}
 where $C_0(\Omega)$ is the (closed) ideal of continuous functions on $\Omega$ that converge to $0$ at the boundary $\partial \Omega$. In other words, $\dot{C}(\Omega)$ is the space of continuous functions on $\overline{\Omega}$ that are constant on  the boundary $\partial \Omega$. 

We define the square-root Laplacian $\Delta_\Omega^{1/2}$ and the (partial) inverse $\Delta_\Omega^{-1/2}$ by Borel functional calculus. The domain of $\Delta_\Omega^{1/2}$ is $W^{1,2}_0(\Omega)$ in the Dirichlet case and $W^{1,2}(\Omega)$ in the Neumann case.

\begin{proposition}\label{prop:BD.ST}
 The triple
 \begin{equation}\label{eq:Dom.ST}
 \left( \dot{C}^\infty(\Omega),L^2(\Omega), \sqrt{\Delta_\Omega}\right) 
\end{equation}
is an $n$-summable spectral triple. 
\end{proposition}
\begin{proof}
We need to show that the commutators $[\sqrt{\Delta_\Omega},f]$, $f\in \dot{C}^\infty(\Omega)$, are bounded, and that $\Delta_\Omega^{-1/2}\in \sL_{n,\infty}$. As constant functions commute with $\sqrt{\Delta_\Omega}$, for the first part we only need to consider the commutators $[\sqrt{\Delta_\Omega},\psi]$ with $\psi\in C^\infty_c(\Omega)$. 

Let $\psi \in C^\infty_c(\Omega)$. The multiplication by $\psi$ preserves $\dom(\sqrt{\Delta_\Omega})$ both in the Dirichlet case and Neumann case. Moreover, as pointed out to the author by Rozenblum~\cite{Ro:PC}, results of Seeley~\cite{Se:AMJ69, Se:ICM70} 
and Grubb~\cite[Theorem~4.4.4]{Gr:Book} show that
\begin{equation}\label{eq:BD.inverse-square-root}
 \Delta_\Omega^{-\frac12} = \car_{\Omega}\Delta_{\R^n}^{-\frac12}\car_{\Omega} + G, 
\end{equation}
where $\Delta_{\R^n}^{-1/2}$ is the inverse square-root Laplacian on $\R^n$ (i.e., the multiplication by $|\xi|^{-1}$ in the Fourier variable) and 
$G:C^\infty(\overline{\Omega})\rightarrow C^\infty(\overline{\Omega})$ is a generalized singular Green  operator in the sense of~\cite{Gr:Book} (up to a smoothing operator in the Neumann case, where $\Delta_\Omega^{-1/2}$ is defined to be zero on $\ker \Delta_\Omega$ as in~\cite[p.~407]{Gr:Book}). In particular, the singular support of the Schwartz kernel of $G$ is contained in $\partial \Omega\times \partial \Omega$. 

Let $u\in \dom(\Delta_\Omega)$. Note that $\psi u \in W^{2,2}(\Omega)$ and satisfies both the Dirichlet and Neumann boundary conditions, and hence is in $\dom(\Delta_\Omega)$. Thus, 
\begin{equation*}
\sqrt{\Delta_\Omega}(\psi u) = \Delta_\Omega^{-\frac12}\Delta_\Omega(\psi u)=  \Delta_\Omega^{-\frac12}\left[\Delta(\psi u)\right]. 
\end{equation*}
As $\psi u$ has compact support in $\Omega$, we have $\car_{\Omega}\left[\Delta(\psi u)\right]=\Delta_{\R^n}(\car_{\Omega}\psi u)$, where $\Delta_{\R^n}$ is the Laplacian on $\R^n$. Combining this with~(\ref{eq:BD.inverse-square-root}) then gives 
\begin{align*}
  \sqrt{\Delta_\Omega}(\psi u) & =   \car_{\Omega}\Delta_{\R^n}^{-\frac12}\Delta_{\R^n}(\car_{\Omega}\psi u) + G\left[\Delta(\psi u)\right]\\
  &=  \car_{\Omega}\sqrt{\Delta_{\R^n}}(\car_{\Omega}\psi u) + S_\psi u,  
  \end{align*}
where $\sqrt{\Delta_{\R^n}}$ is the operator of multiplication by $|\xi|$ in the Fourier variable and we have set $S_\psi:= G\Delta \psi$. This shows that, on $\dom (\Delta_\Omega)$, we have
\begin{equation}\label{eq:BD.sqrt-Delta-psi}
 \sqrt{\Delta_\Omega}\psi= \car_{\Omega}\sqrt{\Delta_{\R^n}}\psi  \car_{\Omega} + S_\psi. 
\end{equation}

Let $k_G(x,y)$ be the Schwartz kernel of $G$. Then the Schwartz kernel of $S_\psi$ is given by
\begin{equation*}
k_{S_\psi}(x,y)= \psi(y)\Delta_yk_G(x,y).  
\end{equation*}
As $\psi \in C^\infty_c(\Omega)$ and $k_G(x,y)$ is smooth away from $\partial \Omega \times \partial \Omega$, we see that $k_{S_\psi}(x,y)$ is smooth on 
$\overline{\Omega}\times \overline{\Omega}$. This ensures that $S_\psi$ is bounded on $L^2(\Omega)$.

Similarly, let $u \in \dom (\sqrt{\Delta_\Omega})$. Then $\Delta_\Omega^{-1/2}u\in \dom(\Delta_\Omega)$, and $\Delta_\Omega$ agrees with $\Delta$ on $\dom(\Delta_\Omega)$, and so we have
\begin{equation*}
 \psi \sqrt{\Delta_\Omega}\,u = \psi \Delta_\Omega  \Delta_\Omega^{-\frac12}u =  \psi \Delta  \Delta_\Omega^{-\frac12}u.  
\end{equation*}
As $\Delta_{\R^n}$ is a differential operator, in the sense of distributions on $\Omega$ we have $\Delta\big[\car_{\Omega}\Delta_{\R^n}^{-\frac12}(\car_{\Omega}u)\big]=\car_{\Omega}\sqrt{\Delta_{\R^n}}(\car_{\Omega}u)$. Combining this with~(\ref{eq:BD.inverse-square-root}) we then get
\begin{align*}
  \psi \sqrt{\Delta_\Omega}\,u & = \psi \Delta \left[  \car_{\Omega}\Delta_{\R^n}^{-\frac12}\car_{\Omega}u \right] + \psi \Delta Gu \\
  & = \car_\Omega\psi \sqrt{\Delta_{\R^n}}(\car_\Omega u) + \tilde{S}_\psi u,  
\end{align*}
where we have set $\tilde{S}_\psi:=\psi \Delta G$. Thus, on $\dom (\sqrt{\Delta_\Omega})$, we have
\begin{equation}\label{eq:BD.psi-sqrt-Delta}
  \psi \sqrt{\Delta_\Omega}= \car_\Omega \psi \sqrt{\Delta_{\R^n}}\car_\Omega + \tilde{S}_\psi.  
\end{equation}

The Schwartz kernel of $\tilde{S}_\psi$ is given by 
\begin{equation*}
 k_{\tilde{S}_\psi}(x,y)= \psi(x)\Delta_xk_G(x,y).  
\end{equation*}
Like the Schwartz kernel of $S_\psi$, it is smooth on $\overline{\Omega}\times \overline{\Omega}$, and so $\tilde{S}_\psi$ is bounded on $L^2(\Omega)$.

It follows from~(\ref{eq:BD.sqrt-Delta-psi}) and~(\ref{eq:BD.psi-sqrt-Delta}) that on $\dom(\Delta_\Omega)$ we have 
\begin{equation*}
\left[ \sqrt{\Delta_\Omega} ,\psi\right]=  \car_{\Omega} \left[ \sqrt{\Delta_{\R^n}} ,\psi\right]\car_{\Omega} +  S_\psi - \tilde{S}_\psi. 
\end{equation*}
As mentioned above, $S_\psi$ and $\tilde{S}_\psi$ are  bounded on $L^2(\Omega)$. Moreover, as $[\sqrt{\Delta_{\R^n}},\psi]$ is a zeroth-order \psido\ on $\R^n$, we see that $\car_{\Omega} \left[ \sqrt{\Delta_{\R^n}} ,\psi\right]\car_{\Omega}$ is bounded on $L^2(\Omega)$. Thus, $[ \sqrt{\Delta_\Omega},\psi]$ is bounded on $\dom(\Delta_\Omega)$. As $\dom(\Delta_\Omega)$ is a core for $\sqrt{\Delta_\Omega}$, it follows that $[ \sqrt{\Delta_\Omega},\psi]$ is bounded. This proves that  $( \dot{C}^\infty(\Omega),L^2(\Omega), \sqrt{\Delta_\Omega})$ is a spectral triple.

Finally, by the classical Weyl law for $\Delta_\Omega$ (see, e.g., \cite{Mi:CJM49} and~\cite[Corollary~4.5.7]{Gr:Book}) we have $\lambda_j(\Delta_\Omega)\sim (c(n)|\Omega|)^{-\frac{2}{n}}j^{\frac{2}{n}}$ as $j\rightarrow \infty$, where $c(n):=(2\pi)^{-n}|\bB^n|$. Thus, $\Delta_\Omega^{-1/2}\in \sL_{n,\infty}$, i.e., the spectral triple~(\ref{eq:Dom.ST}) is $n$-summable. The proof is complete. 
\end{proof}

\begin{remark}
 It is natural to expect that $(C^\infty(\overline{\Omega}), L^2(\Omega), \sqrt{\Delta_\Omega})$ is a spectral triple as well. We shall not need this, since $\dot{C}^\infty(\Omega)$ is dense in the spaces $L^r(\Omega)$ and $\LlogL(\Omega)$, and so using $C^\infty(\overline{\Omega})$ would not change the results below. 
\end{remark}

Bearing this in mind, let $\sB=C^\infty(\overline{\Omega})\cap \dot{C}(\Omega)$ be the algebra of functions in $C^\infty(\overline{\Omega})$ that are constant on $\partial \Omega$. This is a holomorphically closed Fr\'echet subalgebra of $\dot{C}(\Omega)$ containing $\dot{C}^\infty(\Omega)$. By the heat trace asymptotics of Minakshisundaram~\cite{Mi:CJM49} and Branson--Gilkey~\cite{BG:CPDE90} (see also~\cite{Ca:SMK34} for $n=2$ and~\cite{Se:AMJ69b} for the corresponding asymptotics for zeta functions), for every $f\in C^\infty(\overline{\Omega})$, as $t\rightarrow 0^+$ we have
\begin{equation*}
 \Tr\big[fe^{-t\Delta_\Omega}\big]= (4\pi t)^{-\frac{n}{2}}\left[ \int_\Omega f(x)dx + \op{O}\big(t^{\frac12}\big)\right]. 
\end{equation*}
As $\Gamma(1+\frac{n}{2})c(n)=(4\pi)^{-\frac{n}{2}}$, it follows that Condition~\textup{(H)} is satisfied with $\delta=\frac12$ and the algebra $\sB$ above. Therefore, by Theorem~\ref{thm:Intro.Tauberian} Condition~\textup{(W)} holds with 
\begin{equation*}
 \tau(f)= c(n)\int_\Omega f(x)dx, \qquad f\in \dot{C}(\Omega).  
\end{equation*}

Alternatively, if $0<f\in \dot{C}^\infty(\Omega)$, then $f\Delta_\Omega f$ is a second-order elliptic differential operator whose principal symbol is $f(x)^2|\xi|^2$. It agrees up to a positive constant multiple with $\Delta_\Omega$ outside a compact subset of $\Omega$. Therefore, the results of~\cite[Corollary~4.5.7]{Gr:Book} apply, and so we  directly get the Weyl laws, 
\begin{equation*}
 \lambda_j\big(f\Delta_\Omega f\big) \sim j^{\frac{2}{n}} \left(c(n)\int_\Omega f(x)^{-n} dx\right)^{-\frac{2}{n}} \qquad \text{as $j\rightarrow \infty$}. 
\end{equation*}
This also gives Condition~(\textup{W}). 

In addition, Condition~($\textup{C}_r$) holds as follows:
\begin{enumerate}
 \item[(i)] For $r>1$ with $\sV_r=L^r(\Omega)$ (see~\cite{BS:AMST80, Ro:SMD72, Ro:SM76}). 
 
 \item[(ii)] For $r<1$ with $\sV_r=L^1(\Omega)$ (see~\cite{BS:AMST80}). 
 
 \item[(iii)] For $r=1$ with $\sV_1=\LlogL(\Omega)$ if $n$ is even (see~\cite{RS:EMS21, So:IJM94, So:PLMS95}) and $\sV_1=L^{r'}(\Omega)$ for any $r'>1$ if $n$ is odd (see~\cite{BS:AMST80}).  
  \end{enumerate}
Here  $\LlogL(\Omega)$ is the space of $\LlogL$-Orlicz functions on $\Omega$, i.e., measurable functions $f(x)$ such that
\begin{equation*}
 \int_\Omega |f(x)|\log(1+|f(x)|)dx<\infty. 
\end{equation*}
In particular, $L^{r'}(\Omega) \subsetneq \LlogL(\Omega) \subsetneq L^1(\Omega)$ for all $r'>1$.

More precisely, set $l=n/2r$. In both the Dirichlet and Neumann cases, by elliptic regularity and interpolation $\dom((1+\Delta_\Omega)^{l/2})$ embeds continuously into the Sobolev space $W^{l,2}(\Omega)$. Therefore, by the min-max principle, it is enough to establish the corresponding eigenvalue estimates for the Rayleigh quotient, 
\begin{equation*}
 \frac{\int_\Omega V(x)|u(x)|^2dx}{\|u\|_{W^{l,2}(\Omega)}^2}, \qquad u\in W^{l,2}(\Omega)\setminus\{0\}. 
\end{equation*}

For $r\neq 1$, as well as for $r=1$ with $n$ odd, these estimates follow from the proof of~\cite[Theorem~4.1]{BS:AMST80} (see also~\cite{Ro:SMD72, Ro:SM76} for $r>1$). Although~\cite[Theorem~4.1]{BS:AMST80} is stated for integer $l$, its proof only uses the piecewise-polynomial approximation results of~\cite[Theorems~2.3 and 2.17]{BS:AMST80}, which hold for Sobolev spaces of arbitrary real order, together with an extension operator from $W^{l,2}(\Omega)$ to $W^{l,2}$ of a cube. For $r=1$ and $n$ even, the estimates are given by~\cite[Theorem~2.1]{So:PLMS95}. 

For the spectral triple~(\ref{eq:Dom.ST}) the spectral asymptotics~(\ref{eq:Intro.Weyl-|DpaDp|-tau-sVr})--(\ref{eq:Intro.Weyl-DpaDp-tau-sVR}) and the semiclassical Weyl laws~(\ref{eq:Intro.SC-Weyl-Dq-sVr1}) are well known 
(see~\cite{BS:FAA70, BS:AMST80, Ro:SMD72, Ro:SM76, RS:EMS21, So:PLMS95}). The only novelty here is the observation that, at least for continuous potentials, these results can be deduced from the results of Minakshisundaram and Branson--Gilkey mentioned above. In particular, we have the following semiclassical Weyl laws. 

\begin{proposition}[\cite{BS:FAA70, BS:AMST80, Ro:SMD72, Ro:SM76, RS:EMS21, So:PLMS95}] \label{prop:BD.SC-Weyl}
Assume that one of the following holds:
\begin{enumerate}
 \item[(i)] $q< n/2$ and $V(x)$ is a real-valued potential in $L^{\frac{n}{2q}}(\Omega)$. 

 \item[(ii)] $q> n/2$ and $V(x)$ is a real-valued potential in $L^1(\Omega)$. 
 
 \item[(iii)] $q=n/2$ and $V(x)$ is a real-valued potential in $\LlogL(\Omega)$ ($n$ even) or $L^{r}(\Omega)$ with $r>1$ ($n$ odd). 
\end{enumerate}
Then, for any energy level $ \lambda\in \R$, we have
 \begin{gather*}
 \lim_{h\rightarrow 0^+} h^n N\left(h^{2q}\Delta_\Omega^q+V;\lambda\right) = c(n) \int_{\Omega} (V(x)-\lambda)_{-}^{\frac{n}{2q}}dx.
 \end{gather*}
\end{proposition}

In this setting the integration formulas we get from Theorem~\ref{thm:Intro.Integration-sVr} do not seem to have been mentioned elsewhere, and so, in this sense, they are new. In the Neumann case $\Delta_\Omega^{-n/4}$ is defined as a partial inverse. 

\begin{theorem}\label{thm:BD.Integration-formulas}
Assume that, either $n$ is even and $f\in \LlogL(\Omega)$, or $n$ is odd and $f\in L^{r'}(\Omega)$ with $r'>1$. Then, the operators 
$\Delta_\Omega^{-n/4}f\Delta_\Omega^{-n/4}$ and $|\Delta_\Omega^{-n/4}f\Delta_\Omega^{-n/4}|$ are  spectrally measurable, and we have
\begin{equation*}
\bint \Delta_\Omega^{-\frac{n}{4}}f\Delta_\Omega^{-\frac{n}{4}}= c(n)\int_\Omega f(x)dx,\qquad 
\bint\left| \Delta_\Omega^{-\frac{n}{4}}f\Delta_\Omega^{-\frac{n}{4}}\right|=c(n)\int_\Omega |f(x)|dx. 
\end{equation*}
\end{theorem}

\begin{remark}
 For $f\in \dot{C}(\Omega)$, it follows from Theorem~\ref{thm:Intro.Integration} that the operators $f\Delta_\Omega^{-n/2}$ and $|f\Delta_\Omega^{-n/2}|$ are spectrally measurable as well, and we have
\begin{equation*}
\bint f\Delta_\Omega^{-\frac{n}{2}}= c(n)\int_\Omega f(x)dx,\qquad 
\bint\left| f\Delta_\Omega^{-\frac{n}{2}}\right|=c(n)\int_\Omega |f(x)|dx. 
\end{equation*}
\end{remark}
\subsection{Closed Riemannian manifolds}\label{subsec:CM} 
Suppose that $(M^n,g)$, $n\geq 1$, is a closed Riemannian manifold and $E$ is a Hermitian vector bundle over $M$. Let $D_E:C^\infty(M,E)\rightarrow C^\infty(M,E)$ be a selfadjoint first order elliptic  \psido\ with scalar principal symbol given by
\begin{equation*}
 \sigma_1\left(D_E\right)(x,\xi)= |\xi|_g\op{id}_{E_x}, \qquad (x,\xi)\in T^*M\setminus 0. 
\end{equation*}
As this principal symbol is scalar, the commutator $[D_E,u]$ is a zeroth-order \psido, and hence is bounded, for every $u\in C^\infty(M,\End(E))$. We then get an $n$-summable spectral triple, 
\begin{equation*}
 \left( C^\infty\left(M,\End(E)\right), L^2(M,E), D_E\right). 
\end{equation*}
It is more usual to use the algebra $C^\infty(M)$ rather than $C^\infty(M,\End(E))$. However, using the latter allows us to consider Schr\"odinger operators with matrix-valued potentials (see below).

This set-up encapsulates the following examples:
\begin{enumerate}
 \item[(i)] $M$ is spin, $E=\shS$ is the spinor bundle, and $D_E=|\shD|$ is the absolute value of the Dirac operator $\shD$. 
 
 \item[(ii)] $E$ is the trivial line bundle, and $D_E=\sqrt{\Delta_g}$ is the square-root Laplacian. 
 
 \item[(iii)] $\nabla^E$ is a Hermitian connection on $E$ and $D_E=\sqrt{(\nabla^E)^*\nabla^E}$ is the square root of the connection Laplacian $(\nabla^E)^*\nabla^E$. 
 
 \item[(iv)] $M=\partial X$ is the boundary of a compact manifold $\overline{X}$, $E$ is the trivial line bundle, and $D_E=\Lambda_g$ is the Dirichlet-to-Neumann operator (see~\cite{LU:CPAM89, Ta:Book2}). 
\end{enumerate}
With the algebra $C^\infty(M)$, the results of this subsection also hold for the Dirac spectral triple $(C^\infty(M), L^2(M,\shS),\shD)$. 
 
At least in the case of the square-root Laplacian spectral triple $(C^\infty(M), L^2(M), \sqrt{\Delta_g})$, the fact that Condition~\textup{(Z)} holds can be traced back to well-known results of Minakshisundaram--Pleijel~\cite{MP:CJM49} from the 40s (see also~\cite{Se:PSPM67, Sh:Springer01}). The results also hold for the square root $\sqrt{\Delta_E}$ of the connection Laplacian $\Delta_E:=(\nabla^E)^*\nabla^E$ associated with any Hermitian connection $\nabla^E$ on $E$. More precisely, let $(\xi_j)\subseteq C^\infty(M,E)$ be an orthonormal family such that $\Delta_E \xi_j= \lambda_j(\Delta_E)\xi_j$ for all $j\geq 0$. For $\Re s>n$, the series 
\begin{equation*}
 \rho_s(x)=\sum_{j\geq 0} \lambda_j\! \left(\Delta_E\right)^{-\frac{s}{2}} \scal{\ \cdot \ }{\xi_j(x)}_{\!x}
  \xi_j(x) \in \End(E_x), \qquad x\in M,
\end{equation*}
converges uniformly on $M$ and has a meromorphic extension to $\C$ (as a $C(M,\End(E))$-valued family) with at worst simple pole singularities on $n-\N_0$. 
Here $\scal{ \, \cdot  \,}{ \,\cdot \,}_{x}$ is the Hermitian inner product on $E_x$. At $s=n$ the residue is given by 
\begin{equation*}
 \Res_{s=n}  \rho_s(x) = nc(n)\id_{E_x}, \qquad c(n):= (2\pi)^{-n}|\bB^{n}|. 
\end{equation*}

Denote by $\nu_g(x)$ the Riemannian measure. Given any $u\in C(M,\End(E))$, we have
\begin{equation*}
 \Tr\big[u\Delta_E^{-\frac{s}{2}}\big] = \sum_{j\geq 0} \lambda_j\left(\Delta_E\right)^{-\frac{s}{2}} \scal{u\xi_j}{\xi_j} = \int_M \tr_{E_x}\left[u(x)\rho_s(x)\right]d\nu_g(x). 
\end{equation*}
Therefore, we see that the function
\begin{equation*}
 \Tr\big[u\Delta_E^{-\frac{s}{2}}\big] - nc(n)(s-n)^{-1}\int_M  \tr_{E_x}\left[u(x)\right] d\nu_g(x), \quad \Re s>n,
\end{equation*}
has a continuous extension to the closed half-plane $\Re s\geq n$. This shows that the spectral triple $(C^\infty(M,\End(E)), L^2(M,E), \sqrt{\Delta_E})$ satisfies Condition~\textup{(Z)}, and hence Condition~\textup{(W)} holds, with 
\begin{equation*}
 \tau(u)= c(n)\int_M \tr_{E_x}\left[u(x)\right]d\nu_g(x), \qquad u\in C\left(M,\End(E)\right). 
\end{equation*}

This also gives the result for the spectral triple $(C^\infty(M,\End(E)), L^2(M,E), D_E)$. Indeed, as $D_E$ and $\sqrt{\Delta_E}$ are 1st~order \psidos\ that have the same principal symbol, their difference is a zeroth-order \psido, and hence is bounded. As Condition~(W) is invariant under bounded perturbations (\emph{cf}.\ Lemma~\ref{lem:ST-bdd-perturbation}), we deduce that $(C^\infty(M,\End(E)), L^2(M,E), D_E)$ too satisfies Condition~(W) with the same functional $\tau$.  

Alternatively, Condition~(W) can be checked directly by using the Weyl law for positive elliptic \psidos, since $u(D_E)^2u$ is a positive elliptic second-order \psido\ with principal symbol $\sigma_2(x,\xi)=|\xi|_g^2u(x)^2$ (see, e.g., \cite{DG:IM75, Gu:AIM85}). 

As with bounded domains, Condition~($\textup{C}_r$) is a consequence of several Cwikel-type estimates. More precisely,  it holds as follows:
\begin{enumerate}
 \item[(i)] For $r>1$ with $\sV_r=L^r(M,\End(E))$ (see~\cite{Cw:AM77}). 
 
 \item[(ii)] For $r<1$ with $\sV_r=L^1(M,\End(E))$ (see~\cite{BS:UMN77}). 
 
 \item[(iii)] For $r=1$ with $\sV_1=\LlogL(M,\End(E))$ (see~\cite{So:IJM94, So:PLMS95, SZ:MS22}; see also~\cite{Po:MPAG22, Ro:JST22, RS:EMS21}). 
 \end{enumerate}
Here  $\LlogL(M,\End(E))$ is the space of $\LlogL$-Orlicz sections of $\End(E)$, i.e., measurable sections $u(x)$ such that 
 \begin{equation*}
 \int_M \|u(x)\|\log\left( 1+ \|u(x)\|\right)d\nu_g(x) <\infty. 
\end{equation*}

In this setting, Theorem~\ref{thm:Intro.SC-Weyl-lawDq} enables us to recover the following semiclassical Weyl laws. 

\begin{proposition}[\cite{BS:FAA70, Ro:SMD72, Ro:SM76,Ro:JST22,So:PLMS95}] 
 Assume that one of the following holds:
\begin{enumerate}
 \item[(i)] $q< n/2$ and $V(x)=V(x)^*\in L^{\frac{n}{2q}}(M,\End(E))$. 
 
 \item[(ii)] $q> n/2$ and $V(x)=V(x)^*\in L^{1}(M,\End(E))$. 
 
 \item[(iii)] $q=n/2$ and $V(x)=V(x)^*\in \LlogL(M,\End(E))$.
\end{enumerate}
Then, for any energy level $ \lambda\in \R$, we have
 \begin{equation}\label{eq:CM.SC-Weyl-law}
 \lim_{h\rightarrow 0^+} h^n N\left(h^{2q}|D_E|^{2q}+V; \lambda\right) = c(n) \int_M \tr_{E_x}\big[\left(V(x)-\lambda\right)_{-}^{\frac{n}{2q}}\big]d\nu_g(x). 
\end{equation}
\end{proposition}

\begin{remark}
For $q\neq n/2$, the semiclassical Weyl laws~(\ref{eq:CM.SC-Weyl-law}) go back to~\cite{BS:FAA70, Ro:SMD72, Ro:SM76} (see also~\cite{Cw:AM77, Si:TAMS76}). In those references the authors consider  operators on Euclidean spaces, but the results extend to closed manifolds. In the critical case $q=n/2$ this is a more recent result~\cite{Ro:JST22, So:PLMS95} (see also~\cite{Po:MPAG22, RS:EMS21}).   
\end{remark}

\begin{remark}
 For $q=1$, with $D_E=\sqrt{\Delta_g}$ and smooth scalar potentials $V(x)$, the earliest available proof of the semiclassical Weyl laws~(\ref{eq:CM.SC-Weyl-law}) on closed manifolds seems to have been provided by Colin de Verdi\`ere~\cite[Th\'eor\`eme~4.4]{CdV:DM79} (see also, e.g.,~\cite[Section~10]{GS:IP13} and \cite[Theorem~14.11]{Zw:AMS12} for more recent proofs). Note that thanks to the Birman--Schwinger principle this asymptotic can also be deduced from the Weyl laws for negative order \psidos\ on closed manifolds of Birman--Solomyak~\cite{BS:VLU77, BS:VLU79, BS:SMJ79}. However, the approach in~\cite{CdV:DM79} actually provides an asymptotic with a sharper remainder term. 
 \end{remark}

Furthermore, by applying Theorem~\ref{thm:Intro.Integration-sVr} we recover the following extensions of Connes' integration formula~(\ref{eq:NCInt.integ-formula}).  

\begin{proposition}[\cite{Po:MPAG22, Ro:JST22}; see also~\cite{SZ:FAA23}] 
For every section $u\in  \LlogL(M,\End(E))$, the operators $|D_E|^{-n/2}u|D_E|^{-n/2}$ and  
$||D_E|^{-n/2}u|D_E|^{-n/2}|$ are spectrally measurable, and we have
\begin{gather}\label{eq:Riem.int-formula1}
\bint |D_E|^{-\frac{n}{2}}u|D_E|^{-\frac{n}{2}}= c(n)\int_M \tr_{E_x}[u(x)] d\nu_g(x),\\
\bint\left| |D_E|^{-\frac{n}{2}}u|D_E|^{-\frac{n}{2}}\right|= c(n)\int_M \tr_{E_x}\left[|u(x)|\right] d\nu_g(x). 
\label{eq:Riem.int-formula2}
\end{gather}
\end{proposition}

\begin{remark}
 For $u\in C(M,\End(E))$, it follows from Theorem~\ref{thm:Intro.Integration} that the operators $u|D_E|^{-n}$ and $|u|D_E|^{-n}|$ are spectrally measurable as well, and we have
\begin{equation*}
\bint u|D_E|^{-n}= c(n)\int_M \tr_{E_x}[u(x)] d\nu_g(x),\qquad 
\bint\left| u|D_E|^{-n}\right|= c(n)\int_M \tr_{E_x}\left[|u(x)|\right] d\nu_g(x). 
\end{equation*}
\end{remark}

\begin{remark}
The integration formula~(\ref{eq:Riem.int-formula1}) in the form stated above for $\LlogL$-functions was established by Rozenblum~\cite{Ro:JST22}. The integration formulas~(\ref{eq:Riem.int-formula1})--(\ref{eq:Riem.int-formula2}) in the bundle case appeared in~\cite{Po:MPAG22}.  In the scalar case with $D_E=(1+\Delta_{g})^{1/2}$ the integration formula~(\ref{eq:Riem.int-formula1}) for $\LlogL$-functions is not stated explicitly in~\cite{SZ:FAA23}, but it follows from the spectral asymptotics proved there.  
\end{remark}

\begin{remark}
 The integration formula~(\ref{eq:Riem.int-formula1})  need not hold for sections in $L^1(M,\End(E))$ (see~\cite{KLPS:AIM13}). 
\end{remark}

\begin{remark}
The integration formulas~(\ref{eq:Riem.int-formula1})--(\ref{eq:Riem.int-formula2}) for $\LlogL$-potentials are special cases of more general results for operators of the form $Q^*uP$, where $P$ and $Q$ are \psidos\ of order~$-n/2$ (see~\cite{Po:MPAG22, Ro:JST22, RS:EMS21}). These results actually hold for an even larger class of potentials of the form $u=v\mu$ where  $\mu$ is an Ahlfors-regular measure supported on a  regular submanifold $\Sigma \subseteq M$ and 
$v$ is an $\LlogL$-section over $\Sigma$ (see~\cite{Ro:JST22, RS:EMS21}). 
\end{remark}

\subsection{Steklov eigenvalue problem}\label{subsec:Steklov}
Suppose now that $(M^n,g)$ is the boundary of a compact manifold $(\overline{X}^{n+1},\overline{g})$ with interior $X$ such that $\overline{g}_{|TM}=g$ on $M$. Let $\gamma(x)$ be a real-valued function in $L^{n}(M)$ (if $n\geq 2$) or in $\LlogL(M)$ (if $n=1$). The associated weighted Steklov problem is the eigenvalue problem, 
\begin{equation*}
\left\{  \begin{array}{ll}
 \Delta_{\overline{g}}u= 0 & \text{on $X$},\\
 \partial_\nu u=\sigma \gamma u & \text{on $M$},
\end{array}\right. 
 \end{equation*}
where $\partial_\nu$ is the outward normal derivative. For $\gamma=1$ this is the ordinary Steklov problem. We refer to~\cite{CGGS:RMC24, GP:JST17} for recent surveys on the Steklov eigenvalue problem. 

Recall that the Dirichlet-to-Neumann operator $\Lambda_g:C^\infty(M)\rightarrow C^\infty(M)$ is a selfadjoint first-order \psido\ with principal symbol 
$\sigma_1(\Lambda_g)(x,\xi)=|\xi|_g$. In terms of $\Lambda_g$ the above Steklov problem is equivalent to the eigenvalue problem,
\begin{equation*}
 \Lambda_gv=\sigma \gamma v \quad \text{on $M$}. 
\end{equation*}
For $\sigma\neq 0$, this is further equivalent to 
\begin{equation*}
 w= \sigma \Lambda_g^{-\frac12} \gamma \Lambda_g^{-\frac12}w, \qquad \int_M w(x)d\nu_g(x)=0. 
\end{equation*}
Thus, if we let $\sigma_j^\pm(\gamma)$, $j\geq 0$, be the positive/negative Steklov eigenvalues, then we have 
\begin{equation}\label{eq:Steklov.Steklov-DtN-eigenvalues}
 \sigma_j^\pm(\gamma) =\lambda_j^\pm\big( \Lambda_g^{-\frac12} \gamma \Lambda_g^{-\frac12}\big)^{-1}, \qquad j\geq 0. 
\end{equation}

As it follows from the previous subsection $(C^\infty(M), L^2(M), \Lambda_g)$ is a spectral triple satisfying Condition~(W) with 
\begin{equation*}
 \tau(f)= c(n)\int_M f(x)d\nu_g(x), \qquad f\in C(M). 
\end{equation*}
It also satisfies condition ($\textup{C}_n$) with $\sV_n=L^n(M)$ for $n\geq 2$ and $\sV_n=\LlogL(M)$ for $n=1$. In this case Proposition~\ref{prop:Intro.Lq-asymptotics} then asserts that, for all real-valued functions $\gamma$ in $\sV_n$ as above, we have 
 \begin{equation*}
 \lim_{j\rightarrow \infty} j^{\frac{1}{n}}\lambda_j^\pm\big( \Lambda_g^{-\frac12} \gamma \Lambda_g^{-\frac12}\big) = 
 \left( c(n) \int_M \gamma_{\pm}(x)^{n}d\nu_g(x)\right)^{\frac{1}{n}}. 
\end{equation*}

Combining the above spectral asymptotics with~(\ref{eq:Steklov.Steklov-DtN-eigenvalues}) enables us to recover the following Weyl law for weighted Steklov eigenvalue problems. 

\begin{proposition}[\cite{Ag:RJMP06, Ro:JST22, Su:RJMP99}] \label{prop:CM.Steklov-Weyl}
For   any real-valued weight $\gamma(x)$ in $L^{n}(M)$ (if $n\geq 2$) or in $\LlogL(M)$ (if $n=1$), we have
\begin{equation*}
 \lim_{j\rightarrow \infty} j^{-\frac{1}{n}}  \sigma_j^\pm(\gamma)= \left( c(n) \int_M \gamma_{\pm}(x)^{n}d\nu_g(x)\right)^{-\frac{1}{n}}. 
\end{equation*}
\end{proposition}

\begin{remark}
 For $n\geq 2$ the above result goes back to Suslina~\cite{Su:RJMP99} and Agranovich~\cite{Ag:RJMP06}.  For $n=1$ it was established by Rozenblum~\cite{Ro:JST22}. It is shown in~\cite{KLP:ARMA23} that this implies a Weyl law for ordinary Steklov eigenvalues on surfaces with Lipschitz boundaries. This was extended to higher dimension by Rozenblum~\cite{Ro:JST23}.  
\end{remark}

\section{Applications to Quantum Tori}\label{sec:QT}
In this section, we apply the main results of this article to quantum tori. We focus on the spectral triples associated with the square-root Laplacian. The main result is a proof of a conjecture from~\cite{MP:JMP22} on semiclassical Weyl laws on quantum tori as a consequence of the main results of this paper. We shall also obtain enhancements of the integration formulas of~\cite{MSZ:MA19, MP:JMP22, Po:JMP20}. 

\subsection{Quantum tori} 
We briefly present the main background on quantum tori.  We refer to~\cite{HLP:IJM19a,XXY:MAMS18}, and the references therein, for additional background. 

Let $\theta=(\theta_{jk})$ be a real anti-symmetric $n\times n$-matrix. The noncommutative torus $\T^n_\theta$ is the noncommutative space 
represented by a $C^*$-algebra $C(\T^n_\theta)$ generated by $n$ unitaries $U_1, \ldots, U_n$ and subject to the relations, 
\begin{equation*}
 U_kU_j= e^{2i\pi \theta_{jk}}U_jU_k, \qquad j,k=1,\ldots, n. 
\end{equation*}
For $\theta=0$ we recover the $C^*$-algebra $C(\T^n)$ of continuous functions on the ordinary torus $\T^n=(\R/2\pi\Z)^n$ with $U_j=e^{ix_j}$. 

In general, the following unitaries are linearly independent and their linear span is dense in $C(\T^n_\theta)$:
\begin{equation*}
 U^m:=U_1^{m_1} \cdots U_n^{m_n}, \qquad m\in \Z^n.
\end{equation*}
The standard trace $\tau_0:C(\T^n_\theta)\rightarrow \C$ is given by  
\begin{equation*}
 \tau_0(1)=1, \qquad \tau_0\big(U^m\big)=0 \quad \text{if $m\neq 0$}. 
\end{equation*}
This is a positive faithful trace on $C(\T^n_\theta)$. We then define the Hilbert space $L_2(\T^n_\theta)$ as the completion of $C(\T^n_\theta)$ with respect to the inner product, 
\begin{equation*}
 \scal{x}{y}=\tau_0\big[y^*x\big], \qquad x,y\in C(\T^n_\theta). 
\end{equation*}
Note that $\{U^m;\ m\in \Z^n\}$ is an orthonormal basis of $L_2(\T^n_\theta)$. Moreover, the action of $C(\T^n_\theta)$ on itself by left-multiplication uniquely extends to a $*$-representation of $C(\T^n_\theta)$ in $L_2(\T^n_\theta)$. This is just the GNS representation associated with $\tau_0$. 

Let $L_\infty(\T_\theta^n)$ be the von Neumann algebra generated by $U_1,\ldots, U_n$. The trace $\tau_0$ above uniquely extends to a weakly continuous trace $\tau_0:L_\infty(\T^n_\theta)\rightarrow \C$. Since this is a (positive) faithful normal finite trace on the von Neumann algebra $L_\infty(\T^n_\theta)$ we may define the spaces $L_r(\T^n_\theta)$, $r>0$, as the NC $L^r$-spaces associated with the pair $(L_\infty(\T^n_\theta),\tau_0)$ as in Section~\ref{sec:Unbounded}. In particular, $L_r(\T^n_\theta)$ is a quasi-Banach space with respect to the quasi-norm, 
\begin{equation*}
 \|x\|_r:= \tau_0\left[|x|^r\right]^{\frac1r}, \qquad x\in L_r(\T^n_\theta). 
\end{equation*}
Once again, for $r\geq 1$ this is a norm, and so in this case $L_r(\T^n_\theta)$ is a Banach space. Equivalently, $L_r(\T^n_\theta)$ is the completion of $C(\T^n_\theta)$ with respect to the above quasi-norm. In particular, for $r=2$ this is consistent with the previous definition of $L_2(\T^n_\theta)$. 

The space $\LlogL(\T^n_\theta)$ is the NC Orlicz space of the pair $(L_\infty(\T^n_\theta), \tau_0)$ associated with the function $\Phi(t)=t\log(2+t)$, $t\geq 0$ (see~\cite{LM:IJM11,LZ:arXiv26, STTZ:arXiv26}). Thus, it consists of all $L_\infty(\T^n_\theta)$-affiliated operators $x$ such that 
\begin{equation*}
 \tau_0\left[|x|\log\left(2+|x|\right)\right]: = \int_0^\infty \lambda \log(2+\lambda) d\tau_0\left[E(\lambda)\right]<\infty, 
\end{equation*}
where we have used the same notation as in~(\ref{eq:Cr.definition-NCLr}). This is a Banach space with respect to the norm, 
\begin{equation*}
 \|x\|_{\LlogL}:= \inf\left\{ t>0; \ \tau_0\left[t^{-1}|x|\log\left(2+t^{-1}|x|\right)\right]\leq 1\right\}, \qquad x \in \LlogL(\T^n_\theta). 
\end{equation*}

Note that we have continuous (strict) inclusions, 
\begin{equation*}
 L_r(\T^n_\theta) \subsetneq \LlogL(\T^n_\theta) \subsetneq L_1(\T^n_\theta) \qquad \text{for all}\ r>1. 
\end{equation*}
As $\Phi$ satisfies the $\Delta_2$-condition and $\tau_0$ is normal, $L_\infty(\T^n_\theta)$ is a dense subspace of $\LlogL(\T^n_\theta)$. Combining this with the continuity of the inclusion of $L_r(\T^n_\theta)$, $r>1$, into $ \LlogL(\T^n_\theta)$ and the density of $C^\infty(\T^n_\theta)$ in $L_\infty(\T^n_\theta)$ with respect to the $\|\cdot\|_r$-norm, we then deduce that $C^\infty(\T^n_\theta)$  is dense in $ \LlogL(\T^n_\theta)$ (see also~\cite[Lemma~5.2]{LZ:arXiv26}).

We also have a natural $C^*$-action of $\R^n$ on $C(\T^n_\theta)$ given by
\begin{equation*}
 \alpha_s(U^m)=e^{i s\cdot m}U^m, \qquad m\in \Z^n, \ s\in \R^n. 
\end{equation*}
For $\theta=0$ this is just the action arising from the action of $\R^n$ on $\T^n$ by translation. The canonical derivations $\partial_1, \ldots, \partial_n$  are the infinitesimal generators of this action with the convention that
\begin{equation*}
 \partial_j\big(U_j\big)=iU_j, \qquad  \partial_j\big(U_k\big)=0, \quad k\neq j. 
\end{equation*}

The smooth NC torus $C^\infty(\T^n_\theta)$ is precisely the $*$-subalgebra of smooth vectors for this action. Equivalently, 
\begin{equation*}
 C^\infty(\T^n_\theta)=\left\{u= \sum u_mU^m; (u_m)_{m\in \Z^n}\in \sS(\Z^n)\right\}, 
\end{equation*}
where $\sS(\Z^n)$ is the space of complex-valued sequences with rapid decay. For $\theta=0$ we recover the description of smooth functions on $\T^n$ in terms of the rapid decay of their
Fourier coefficients. Note that $C^\infty(\T^n_\theta)$ is closed under holomorphic functional calculus and is a Fr\'echet algebra with respect to the seminorms, 
\begin{equation*}
 u\rightarrow \sup_{m\in \Z^n}(1+|m|)^N|u_m|, \qquad N\geq 0. 
\end{equation*}
 In addition, given any $s\geq 0$, the Sobolev space $W_2^s(\T^n_\theta)$ is defined by 
\begin{equation*}
 W_2^s(\T^n_\theta):=\Big\{u =\sum_{m\in \Z^n} u_mU^m\in L_2(\T^n_\theta); \sum_{m\in \Z^n} (1+|m|^2)^s|u_m|^2<\infty\Big\}. 
\end{equation*}

\subsection{Proof of Theorem~\ref{thm:Intro.QT}} 
We shall now explain how the main results of this paper, and more specifically Theorem~\ref{thm:Intro.SC-Weyl-law-sVr}, enable us to prove Theorem~\ref{thm:Intro.QT}. The latter result is about semiclassical Weyl laws for fractional Schr\"odinger operators, 
\begin{equation*}
 h^{2q} \Delta^q + V, \qquad q>0, 
\end{equation*}
where $\Delta$ is the (flat) Laplacian. That is, $\Delta$ is the second-order differential operator given by
\begin{equation*}
 \Delta:=-(\partial_1^2+\cdots + \partial_n^2). 
\end{equation*}

Here $\Delta$ is (formally) selfadjoint and elliptic in the sense of~\cite{Co:CRAS80, HLP:IJM19b}. Therefore, it extends to a selfadjoint operator  
on $L_2(\T^n_\theta)$ with domain $W_2^2(\T^n_\theta)$. In fact, we have
\begin{equation}\label{eq:NCT.Laplacian}
 \Delta\big(U^k)=|k|^2U^k, \qquad k\in \Z^n. 
\end{equation}
Thus, $\Delta$ is isospectral to the Laplacian on the ordinary torus $\T^n$. In particular, it has a non-negative discrete spectrum with $\ker \Delta=\C\cdot 1$.

\begin{proposition}
The triple
 \begin{equation}\label{eq:QT.Spectral-triple}
 \left(C^\infty(\T^n_\theta), L_2(\T^n_\theta), \sqrt{\Delta}\right) 
\end{equation}
is an $n$-summable spectral triple. 
\end{proposition}
 \begin{proof}
 This result is well-known (see, e.g., \cite{GVF:Birkhauser01, MSZ:MA19}).  The bulk of the proof  is showing that the commutators $[\sqrt{\Delta},a]$ are bounded as $a$ ranges over $C^\infty(\T^n_\theta)$. We briefly sketch two proofs of this fact for the reader's convenience. 
 
The first proof relies on the pseudodifferential calculus on NC tori of Connes~\cite{Co:CRAS80} and Baaj~\cite{Ba:CRAS88}. We refer to~\cite{HLP:IJM19a, HLP:IJM19b} for a detailed account on this \psido\ calculus. The square root $\sqrt{\Delta}$ is a \psido\ of order~1 in this pseudodifferential calculus. Its principal symbol is $|\xi|$. Thus, if $a\in C^\infty(\T^n_\theta)$, then $a\sqrt{\Delta}$ and $\sqrt{\Delta}a$ are  \psidos\ of order~$1$ whose principal symbols are both equal to $|\xi|a$. It follows that their difference, i.e.,  the commutator $[\sqrt{\Delta},a]$, is a \psido\ of order~$0$, and hence is bounded. 

This fact can be checked directly without any reference to \psidos. First, $\dom(\sqrt{\Delta})=W_2^1(\T^n_\theta)$ is a (left-)module over $C^\infty(\T^n_\theta)$ (see, e.g., \cite{Sp:Padova92}), and so $a(\dom (\sqrt{\Delta})) \subseteq \dom (\sqrt{\Delta})$. Moreover, if we put  $a=\sum_{k\in \Z^n}a_kU^k$ with $(a_k)\in \sS(\Z^n)$, then 
\begin{gather*}
 \big(a\sqrt{\Delta})(U^m)=|m| aU^m= \sum_{k\in \Z^n} |m|a_k U^kU^m,\\
  \big(\sqrt{\Delta}a)(U^m)=\sum_{k\in \Z^n} a_k \sqrt{\Delta}(U^kU^m)=\sum_{k\in \Z^n} |m+k|a_k U^kU^m.
\end{gather*}
Thus, 
\begin{equation*}
 \big[\sqrt{\Delta},a\big](U^m)= \rho(m)U^m, \qquad \text{where}\ \rho(m):= \sum_{k\in \Z^n}\big(|m+k|- |m|\big)a_k U^k. 
\end{equation*}
Here $(a_k)\in \sS(\Z^n)$. Thus, given any $N\geq 1$, the triangle inequality $||m+k|-|m||\leq |k|$ gives
\begin{align*}
 \big\|(1+\Delta)^N \rho(m)\big\| &\leq \sum_{k\in \Z^n} \big||m+k|-|m|\big| \left(1+|k|^2\right)^N|a_k|\\
 & \leq  \sum_{k\in \Z^n}|k| \left(1+|k|^2\right)^N|a_k|<\infty. 
\end{align*}
We may then apply Schur's test as in the proof of~\cite[Proposition~10.1]{HLP:IJM19b} to check that  $[\sqrt{\Delta},a]$ is bounded. 

Finally,  as $\Delta$ is isospectral to the Laplacian on the ordinary torus $\T^n$, the partial inverse ${\Delta}^{-1}$ is in the weak Schatten class $\sL_{n/2,\infty}$. It then follows that $\Delta^{-1/2}\in \sL_{n,\infty}$, and so the triple~(\ref{eq:QT.Spectral-triple}) is an $n$-summable spectral triple. The proof is complete. 
\end{proof}

\begin{proposition}\label{prop:QT.heat-asymptotic}
 Given any $a\in C^\infty(\T^n_\theta)$, as $t\rightarrow 0^+$ we have 
\begin{equation*}
 \Tr\big[ae^{-t\Delta}\big]= \pi^{\frac{n}{2}}t^{-\frac{n}{2}}\left[\tau_0(a) + \op{O}\big(e^{- \frac{\pi^2}{t}}\big)\right].  
\end{equation*}
\end{proposition}
\begin{proof}
 In view of~(\ref{eq:NCT.Laplacian}) we have 
\begin{equation*}
  \Tr\big[ae^{-t\Delta}\big] = \sum_{m\in \Z^n} \bigscal{ae^{-t\Delta}(U^m)}{U^m} =  \sum_{m\in \Z^n}e^{-t|m|^2} \scal{aU^m}{U^m}.
 \end{equation*}
Here $ \scal{aU^m}{U^m}=\tau_0[(U^m)^*aU^m]=\tau_0[aU^m(U^m)^*]=\tau_0(a)$. Thus, 
\begin{equation}\label{eq:NCT.theta-vartheta}
 \Tr\big[ae^{-t\Delta}\big] =  \sum_{m\in \Z^n}e^{-t|m|^2} \tau_0(a)= \vartheta(t)^n \tau_0(a), \qquad \text{with}\ \vartheta(t) := \sum_{k\in \Z} e^{-tk^2}. 
\end{equation}

The function $\vartheta(t)$ is a Jacobi theta function. By Poisson's summation formula, we have
\begin{equation*}
 \vartheta(t) = \sqrt{\frac{\pi}{t}} \sum_{k\in \Z} e^{- \frac{\pi^2 k^2}{t}} = \sqrt{\frac{\pi}{t}}  + 2 \sqrt{\frac{\pi}{t}}\sum_{k\geq 1} e^{- \frac{\pi^2 k^2}{t}}. 
\end{equation*}
Note that, for $0<t\leq 1$, we have 
\begin{equation*}
 \sum_{k\geq 1} e^{- \frac{\pi^2 k^2}{t}} = \sum_{k\geq 1} e^{- \left(\frac{1}{t}-1\right)\pi^2 k^2} \cdot  e^{- \pi^2 k^2} \leq  
e^{- \left(\frac{1}{t}-1\right)\pi^2} \sum_{k\geq 1}   e^{- \pi^2 k^2}. 
\end{equation*}
Thus,
\begin{equation*}
 \vartheta(t) = \sqrt{\frac{\pi}{t}} \left(1 +\op{O}\big(e^{- \frac{\pi^2}{t}}\big)\right) \qquad \text{as $t\rightarrow 0^+$}. 
\end{equation*}
Combining this with~(\ref{eq:NCT.theta-vartheta}) gives the result. 
The proof is complete. 
\end{proof}

As $C^\infty(\T^n_\theta)$ is a holomorphically closed Fr\'echet subalgebra of $C(\T^n_\theta)$, we have the following consequence of Proposition~\ref{prop:QT.heat-asymptotic}. 

\begin{corollary}\label{cor:QT-ConditionW}
 Condition~\textup{(H)} and Condition~\textup{(W)} hold with 
\begin{equation*}
 \tau(a) := \frac{\pi^{\frac{n}{2}}}{\Gamma\left(\frac{n}{2}+1\right)} \tau_0(a)= \hat{c}(n)\tau_0(a), \qquad \text{where}\ \hat{c}(n):=|\bB^n|. 
\end{equation*}
\end{corollary}

In addition, we have the following Cwikel-type estimates.
 
 \begin{proposition}[\cite{MP:JMP22, MSX:CMP19}]\label{prop:QT.Cwikel-flat} 
 Suppose that $r\neq 1$, and set $\hat{r}:=\max(r,1)$. There is a constant $C_{r}>0$ such that
\begin{equation*}
 \big\| (1+\Delta)^{-\frac{n}{4r}} a(1+\Delta)^{-\frac{n}{4r}}\big\|_{r,\infty} \leq C_{r}\|a\|_{\hat{r}} \qquad \forall a\in C^\infty(\T^n_\theta).  
\end{equation*}
 \end{proposition}

Actually, the results of~\cite{MP:JMP22} also provide Cwikel-type estimates for $r=1$ on any space $L_{p}(\T^n_\theta)$ with $p>1$. This was recently improved by Liu--Zhang~\cite{LZ:arXiv26} who extended the aforementioned Cwikel-type estimate to the NC Orlicz space $\LlogL(\T^n_\theta)$. Namely, they proved the following result. 

\begin{proposition}[Liu--Zhang~\cite{LZ:arXiv26}]\label{prop:QT.Cwikel-LlogL} 
There is a constant $C>0$ such that 
 \begin{equation*}
 \big\| (1+\Delta)^{-\frac{n}{4}} a(1+\Delta)^{-\frac{n}{4}}\big\|_{1,\infty} \leq C\|a\|_{\LlogL} \qquad \forall a\in C^\infty(\T^n_\theta).  
\end{equation*}
\end{proposition}

\begin{remark}
The Cwikel estimate of Proposition~\ref{prop:QT.Cwikel-LlogL} is the analogue for NC tori of the critical Cwikel estimates of Solomyak~\cite{So:IJM94,So:PLMS95} and Sukochev--Zanin~\cite{SZ:MS22} for Euclidean spaces and ordinary tori. We stress that in the case of ordinary tori (i.e., $\theta=0$) the approach of~\cite{LZ:arXiv26} even simplifies the approach of~\cite{SZ:MS22}. The estimate holds more generally for the NC Orlicz spaces associated with faithful normal semifinite traces on von Neumann algebras whose $L_2$-spaces admit an orthonormal basis whose elements are unitaries. 
\end{remark}

Proposition~\ref{prop:QT.Cwikel-flat} and Proposition~\ref{prop:QT.Cwikel-LlogL} have the following consequence. 

\begin{corollary}\label{cor:QT.ConditionCr}
 For every $r>0$, Condition \textup{($\textup{C}_r$)} is satisfied,  with
\begin{equation}\label{eq:QT.ConditionCr}
 \sV_r=L_r(\T^n_\theta) \ (r>1), \qquad  \sV_r=\LlogL(\T^n_\theta) \ (r=1), \qquad \sV_r=L_1(\T^n_\theta) \ (r<1). 
\end{equation}
\end{corollary}

We are now in a position to prove Theorem~\ref{thm:Intro.QT}. 

\begin{proof}[Proof of Theorem~\ref{thm:Intro.QT}]
Let $q>0$, and assume that one of the following holds: 
\begin{enumerate}
 \item[(i)] $q< n/2$ and $V=V^*\in L_{\frac{n}{2q}}(\T^n_\theta)$. 
 
 \item[(ii)] $q> n/2$ and $V=V^*\in L_1(\T^n_\theta)$. 
 
 \item[(iii)] $q=n/2$ and  $V=V^* \in \LlogL(\T^n_\theta)$. 
\end{enumerate}
Thanks to Corollary~\ref{cor:QT-ConditionW}, Condition (W) holds with $\tau= \hat{c}(n)\tau_0$. Moreover, if we set 
$r=n(2q)^{-1}$, then by Corollary~\ref{cor:QT.ConditionCr} Condition~($\textup{C}_r$) holds with $\sV_r$ as in~(\ref{eq:QT.ConditionCr}). The conditions (i)--(iii) thus ensure that $V=V^*\in \sV_r$ in each case. Applying Theorem~\ref{thm:Intro.SC-Weyl-law-sVr} then shows that, for every energy level $\lambda\in \R$, we have
\begin{equation}\label{eq:NCT.SC-Weyl}
 \lim_{h\rightarrow 0^+} h^nN\big(h^{2q}\Delta^q+V;\lambda\big) = \hat{c}(n)\tau_0\big[(V-\lambda)_{-}^{\frac{n}{2q}}\big], \qquad \hat{c}(n):=|\bB^n|. 
  \end{equation} 
This proves Theorem~\ref{thm:Intro.QT}. 
\end{proof}

\begin{remark}
 Liu--Zhang~\cite[Corollary~5.8]{LZ:arXiv26} obtained a semiclassical Weyl law for the number of negative eigenvalues of the Schr\"odinger operators $h^n(1+\Delta)^{\frac{n}{2}}+V$ with $V=V^*\in \LlogL(\T^n_\theta)$. This corresponds to the case $q=n/2$ and $\lambda=0$ of Theorem~\ref{thm:Intro.QT}, with $\Delta$ replaced by $1+\Delta$. 
\end{remark}

 \begin{remark}
 We refer to~\cite{Po:curvedQT} for a version of Theorem~\ref{thm:Intro.QT} for curved quantum tori. 
 In the terminology of~\cite{MP:AIM23} a curved quantum torus is a quantum torus equipped with a Riemannian metric $g$ in the sense of~\cite{HP:JGP20, Ro:SIGMA13}. In this case the flat Laplacian is replaced by the Laplacian $\Delta_g$ associated with $g$ as in~\cite{HP:JGP20}. We stress that for an arbitrary Riemannian metric the commutators $[\sqrt{\Delta_g},a]$, $a\in C^\infty(\T^n_\theta)$, need not be bounded, and so $\sqrt{\Delta_g}$ does not provide a spectral triple. As a result, a curved version of Theorem~\ref{thm:Intro.QT} cannot be obtained as a direct consequence of the main results of this paper. Nevertheless, as explained in~\cite{Po:curvedQT}, it is still possible to get a curved version of the semiclassical Weyl laws~(\ref{eq:NCT.SC-Weyl}) by modifying the spectral triple~(\ref{eq:QT.Spectral-triple}). 
\end{remark}
 
\subsection{Integration formulas} 
By applying Theorem~\ref{thm:Intro.Integration-sVr} we obtain the following integration formulas. 
 
 \begin{theorem}[see also~\cite{LZ:arXiv26}]\label{thm:QT.Integration-LlogL}
For every $x\in \LlogL(\T^n_\theta)$, the operators $\Delta^{-n/4}x\Delta^{-n/4}$ and $|\Delta^{-n/4}x\Delta^{-n/4}|$ are spectrally measurable, and we have  
\begin{gather}\label{eq:NCT.integration-formula1}
  \bint \Delta^{-\frac{n}{4}}x \Delta^{-\frac{n}{4}} =\hat{c}(n)\tau_0\big[x\big],\\
 \bint \big|  \Delta^{-\frac{n}{4}}x \Delta^{-\frac{n}{4}}\big|= \hat{c}(n)\tau_0\big[|x|\big].
 \label{eq:NCT.integration-formula2}
 \end{gather}
\end{theorem}

\begin{remark}
 For $x\in C(\T^n_\theta)$, it follows from Theorem~\ref{thm:Intro.Integration} that the operators $x\Delta^{-n/2}$ and $|x\Delta^{-n/2}|$ are spectrally measurable as well, and we have
\begin{equation*}
 \bint x\Delta^{-\frac{n}{2}} =\hat{c}(n)\tau_0\big[x\big],\qquad 
 \bint \big|  x \Delta^{-\frac{n}{2}}\big|= \hat{c}(n)\tau_0\big[|x|\big].
\end{equation*}
\end{remark}
 
\begin{remark}
 Theorem~\ref{thm:QT.Integration-LlogL} enhances the integration formulas of~\cite{MSZ:MA19, MP:JMP22, Po:JMP20} in two ways: it holds for all elements of the NC Orlicz space $\LlogL(\T^n_\theta)$, and it provides spectral measurability, which is stronger than the strong measurability established in those articles. Both enhancements were anticipated in~\cite[Remark~6.7]{Po:MPAG22}. In the first three versions of this paper posted on arXiv, Theorem~\ref{thm:QT.Integration-LlogL} was established for elements of $L_r(\T^n_\theta)$, $r>1$. Subsequently, Liu--Zhang~\cite{LZ:arXiv26} established the critical Cwikel estimate of Proposition~\ref{prop:QT.Cwikel-LlogL}. For $n\geq 2$, they also obtained eigenvalue and singular-value asymptotics for $\Delta^{-n/4}x\Delta^{-n/4}$, $x\in \LlogL(\T^n_\theta)$, using results of~\cite{Zh:arXiv26}, and derived the integration formulas~(\ref{eq:NCT.integration-formula1})--(\ref{eq:NCT.integration-formula2}) in terms of strong measurability (called PT measurability there). Spectral measurability is not discussed in~\cite{LZ:arXiv26}. Here, both spectral measurability and the integration formulas for $\LlogL$-elements follow directly from Theorem~\ref{thm:Intro.Integration-sVr}. 
\end{remark}
 
 \begin{remark}
 We refer to~\cite{Po:curvedQT} for a version of Theorem~\ref{thm:QT.Integration-LlogL} for curved quantum tori, enhancing the integration formulas of~\cite{MP:AIM23, Po:JMP20}. 
\end{remark}

\begin{remark}
 The main results of this section hold \emph{mutatis mutandis} for the Dirac spectral triple on quantum tori (see, e.g.,~\cite[\S12.3]{GVF:Birkhauser01} for the construction of this spectral triple).  
\end{remark}

\section{Final Comments} 
Further examples are presented in the survey~\cite{Po:NYC25} and in the companion papers~\cite{Po:Open, Po:Bir-Sol, Po:Carnot}. (Those examples were originally included in the first two versions of this paper on arXiv.) Curved quantum tori are dealt with in the companion paper~\cite{Po:curvedQT}. 

The focus of this paper is on \emph{unital} spectral triples. Extending the results to \emph{non-unital} spectral triples would allow us to include Euclidean spaces, infinite volume geometries, and quantum Euclidean spaces. While Condition~(W) is specific to the unital case, the Tauberian conditions (Z) and (H) and the double operator-integral arguments carry over with minor modifications.

Several examples, including the Dirac operator $\sD$ whose eigenvalues are the imaginary parts of the zeros of Riemann's zeta function~\cite{Co:AFA24,CM:PNAS22}, give rise to \emph{nonclassical} Weyl laws (in the sense of Simon~\cite{Si:JFA83}) involving logarithmic corrections. Such examples are naturally summable in weak Lorentz ideals of the form
\begin{equation*}
 \sL_g:=\left\{T; \ \mu_j(T)=\op{O}\left(g(j)\right)\right\}, \qquad g(t)=t^{-p}(\log(t+1))^q, \quad p,q>0,
\end{equation*}
rather than in weak Schatten classes. The relationships between Connes' integration in such weak Lorentz ideals and nonclassical Weyl laws are studied in the preprints~\cite{PT:Part1, PT:Part2} (see also~\cite{GS:JFA14,GU:JMAA20,LU:JMAA25}).

Finally, as pointed out by Magnus Goffeng~\cite{Go:PC}, it would be interesting to extend the results to generalized spectral triples~\cite[Appendix~A]{GM:DM15} in which the boundedness of $[D,a]$ is replaced by the boundedness of $(1+D^2)^{(1-m)/4m}[D,a](1+D^2)^{(1-m)/4m}$ for some $m>0$. For example, for a closed Riemannian manifold with $\sA=C^\infty(M)$, this condition holds for $D=\Delta_g$ with $m=2$.

\end{document}